\documentclass{amsart}
\usepackage{array}
\newcolumntype{P}[1]{>{\raggedright\let\newline\\\arraybackslash\hspace{0pt}}m{#1}}
\usepackage{graphicx,verbatim, amsmath, amssymb, amsthm, amsfonts, epsfig, amsxtra,ifthen,mathtools,%epstopdf,
caption,enumitem,hhline,bbm,capt-of,longtable}	
\usepackage[bookmarks=true, bookmarksopen=false,%
    colorlinks=true,%
    linkcolor=darkblue,%
    citecolor=darkblue,%
    filecolor=darkblue,%
    menucolor=darkblue,%
    linktoc=all,%
    urlcolor=darkblue
]{hyperref}
\usepackage[usenames]{xcolor}
\definecolor{darkblue}{cmyk}{1,0.3,0,0.1}  %blue
\DeclareFontFamily{U}{mathx}{\hyphenchar\font45}
\DeclareFontShape{U}{mathx}{m}{n}{<-> mathx10}{}
\DeclareSymbolFont{mathx}{U}{mathx}{m}{n}
\DeclareMathAccent{\widebar}{0}{mathx}{"73}

\newtheorem{proposition}{Proposition}[section]
\newtheorem{theorem}[proposition]{Theorem}
\newtheorem{probable}[proposition]{Probable Theorem}

\newtheorem{lemma}[proposition]{Lemma}

\theoremstyle{definition}

\theoremstyle{remark}
\newtheorem{remark}[proposition]{Remark}

\numberwithin{equation}{section}

\newcommand{\newword}[1]{\textbf{\emph{#1}}}

\newcommand{\integers}{\mathbb Z}

\newcommand{\reals}{\mathbb R}

\newcommand{\edge}{\,\,\rule[2.7pt]{20pt}{0.5pt}\,\,}

\newcommand{\ep}{\varepsilon}
\newcommand{\thet}{\vartheta}
\newcommand{\col}{\operatorname{col}}

\newcommand{\sgn}{\operatorname{sgn}}

\newcommand{\nnspan}{{\!\!\tiny\begin{array}{r}\mathbf{span}\\\ge0\end{array}\hspace*{-0.57em}}}

\newcommand{\set}[1]{{\lbrace #1 \rbrace}}
\newcommand{\sett}[1]{{\bigl\lbrace #1 \bigr\rbrace}}
\newcommand{\settt}[1]{{\Bigl\lbrace #1 \Bigr\rbrace}}

\newcommand{\br}[1]{{\langle #1 \rangle}}
\newcommand{\brr}[1]{{\bigl\langle #1 \bigr\rangle}}

\newcommand{\I}{{\mathcal I}}
\newcommand{\GG}{{\mathbf G}}
\newcommand{\CC}{{\mathbf C}}

\newcommand{\F}{{\mathcal F}}

\newcommand{\ck}{\spcheck}

\renewcommand{\th}{^\text{th}}

\newcommand{\0}{{\mathbf{0}}}

\newcommand{\Cone}{\mathrm{Cone}}
\newcommand{\Comp}{\mathrm{Comp}_C}
\newcommand{\CompPlus}{\overline{\mathrm{Comp}}_C}

\newcommand{\Proj}{\mathrm{Proj}}

\newcommand{\g}{\mathbf{g}}

\renewcommand{\c}{\mathbf{c}}
\renewcommand{\b}{\mathbf{b}}

\newcommand{\kk}{{\boldsymbol{k}}}
\renewcommand{\ll}{{\boldsymbol\ell}}

\newcommand{\e}{\mathbf{e}}

\newcommand{\tB}{{\widetilde{B}}}

\newcommand{\BB}{\mathbf{B}}

\newcommand{\R}{\mathcal{R}}
\renewcommand{\L}{\mathcal{L}}

\renewcommand{\H}{\mathcal{H}}
\renewcommand{\P}{\mathcal{P}}
\newcommand{\IP}{\mathcal{Z}}

\newcommand{\Fan}{\operatorname{Fan}}

\newcommand{\re}{\mathrm{re}}

\renewcommand{\d}{{\mathfrak d}}

\renewcommand{\c}{{\mathbf c}}
\newcommand{\dd}{{\mathbf d}}

\renewcommand{\th}{^\text{th}}

\newcommand{\DF}{{\mathcal {DF}}}

\newcommand{\symmetrizer}{{D}}

\newcommand{\RSChar}{\Phi}
\newcommand{\RS}{\RSChar}

\newcommand{\RSpos}{\RS^+}

\newcommand{\SimplesChar}{\Pi}
\newcommand{\Simples}{\SimplesChar}

\newcommand{\RSTChar}{\Upsilon}
\newcommand{\RST}[1]{\RSTChar^{#1}}

\newcommand{\SimplesTChar}{\Xi}
\newcommand{\SimplesT}[1]{\SimplesTChar^{#1}}

\newcommand{\SuppT}{\operatorname{Supp}_\SimplesTChar}
\newcommand{\TravInfChar}{\Psi}

\newcommand{\TravProj}[1]{\overrightarrow{\TravInfChar}^{#1}}

\newcommand{\TravInj}[1]{\overleftarrow{\TravInfChar}^{#1}}
\newcommand{\AP}[1]{\RS_{#1}}
\newcommand{\APre}[1]{\AP{#1}^\re}
\newcommand{\APTChar}{\Lambda}
\newcommand{\APT}[1]{\APTChar_{#1}}      
\newcommand{\APTre}[1]{\APT{#1}^\re}

\newcommand{\afftype}[1]{{#1^{(1)}}}

\newcommand{\margincolor}{red}      
\definecolor{darkgreen}{rgb}{0,0.7,0}

\newcounter{margincounter}
\newcommand{\marginnum}{
\ifnum\value{margincounter}<10
\textcolor{\margincolor}{\begin{picture}(0,0)\put(2.2,2.4){\circle{9}}\end{picture}\footnotesize\arabic{margincounter}}
\else\ifnum\value{margincounter}<100
\textcolor{\margincolor}{\begin{picture}(0,0)\put(4.256,2.5){\circle{11}}\end{picture}\footnotesize\arabic{margincounter}}
\else
\textcolor{\margincolor}{\begin{picture}(0,0)\put(6.8,2.5){\circle{14}}\end{picture}\footnotesize\arabic{margincounter}}
\fi\fi
}

\makeatletter
\newcommand{\switchmargin}{
\if@reversemargin
\normalmarginpar
\else
\reversemarginpar
\fi
}
\makeatother

\author{Nathan Reading}
\author{Dylan Rupel}
\author{Salvatore Stella}
\title{Dominance regions for affine cluster algebras}
\address[N. Reading]{Department of Mathematics, North Carolina State University, %Raleigh, NC, 
USA}
\address[D. Rupel]{Math Academy, Pasadena Unified School District, Pasadena, CA, USA}
\address[Salvatore Stella]{Dipartimento di Ingegneria e Scienze dell'Informazione e Matematica, Universit\`{a} degli Studi dell'Aquila, IT}
\thanks{Nathan Reading was partially supported by the Simons Foundation under award number 581608 and by the National Science Foundation under award number DMS-2054489.
Salvatore Stella was partially supported by PRIN 20223FEA2E, \lq\lq Cluster algebras and Poisson Lie groups\rq\rq and INdAM/GNSAGA}
\allowdisplaybreaks

\begin{document}

\begin{abstract}
We determine dominance regions associated to cluster algebras of affine type.
In the most interesting cases, the dominance region is a line segment, which we describe explicitly.
Motivations for this work include a project to determine all pointed bases for cluster algebras of affine type and a separate application that determines all theta functions in the affine case.
The proofs draw on known results from the doubled Cambrian fan and almost-positive roots models, as well as a new tool that we develop: a detailed description of neighboring seeds of affine type (seeds that are, in some sense, as close as possible to the boundary of the $\g$-vector fan).
\end{abstract}

\maketitle

%\vspace{-8pt}

\setcounter{tocdepth}{2}
\tableofcontents

\section{Introduction}  
Given an $m\times n$ extended exchange matrix $\tB$ and a point $\tilde\lambda\in\reals^m$, the dominance region $\P_{\tilde\lambda}^\tB$ is the intersection of an infinite collection of piecewise polyhedral sets, each defined by applying mutation maps to a (translated) cone.
Dominance regions were defined by Rupel and Stella~\cite{RSDom}, recasting the definition of a partial order called dominance on $\integers^m$ due to Fan Qin~\cite{FanQin}.
Dominance regions/orders are the crucial discrete-geometric construction that characterizes the class of pointed bases of the (upper) cluster algebra associated to $\tB$, a class that includes the theta basis~\cite{GHKK}, the generic basis~\cite{Dup11}, and the common triangular basis~\cite{QinTri}. 
(More specifically, Fan Qin characterized all pointed bases of the upper cluster algebra in terms of this order, assuming that $\tB$ has linearly independent columns and admits a green-to-red sequence \cite[Theorem~1.2.1]{FanQin}.)

Rupel and Stella~\cite{RSDom} computed dominance regions in the case where $n=m=2$.
They found that, in the $n=m=2$ affine case, dominance regions are either points or line segments.
In this paper, we generalize that fact to all extended exchange matrices of affine type for arbitrary $m\ge n$.
Our results are motivated by two applications.  
One is an ongoing project to determine all pointed bases of cluster algebras of affine type using \cite[Theorem~1.2.1]{FanQin} and to relate them to certain generalized minors in double Bruhat cells as in \cite{RSW19}.
The other application combines the results of this paper with a result of \cite{canonical} and uses dominance regions to greatly simplify the computation of some structure constants for multiplication of theta functions.
This allows us, in~\cite{afftheta}, to determine all theta functions for cluster scattering diagrams of affine type, determine ``imaginary'' exchange relations among them, and identify an ``imaginary subalgebra'' that is essentially a product of generalized cluster algebras of finite type $C$.

Let $B$ be the $n\times n$ submatrix of $\widetilde B$ associated to mutable variables.
A preliminary step to understanding dominance regions in affine type is to understand the mutation fan for $B^T$ in the sense of~\cite{universal}, which coincides in affine type \mbox{\cite[Theorem~2.10]{affscat}} with the cluster scattering fan in the sense of~\cite{scatfan}.
In any type, the mutation fan for $B^T$ contains the $\g$-vector fan for $B$ as a subfan~\cite[Theorem~8.7]{universal}, and in affine type, the complement of the $\g$-vector fan is a codimension-$1$ cone \mbox{\cite[Corollaries~1.3, 4.9]{afframe}}.
This cone is called the imaginary wall, because it is a wall in the cluster scattering diagram~\cite{affscat}.
There is a unique ray of the mutation fan, called the imaginary ray, that is in the relative interior of the imaginary wall.  

Up to translation, the dominance region of $\tilde\lambda$ only depends on the projection of~$\tilde\lambda$ into $\reals^n$ that ignores the last $m-n$ entries.
We prove that in affine types, when $\tB$ has linearly independent columns, the dominance region $\P_{\tilde\lambda}^\tB$ is a point if $\tilde\lambda$ projects to the $\g$-vector fan or otherwise (i.e.\ when $\tilde\lambda$ projects to the relative interior of the imaginary wall) is a line segment that we describe explicitly.

Computing dominance regions without the assumption that $\tB$ has linearly independent columns is in general much harder. 
However, in affine type, for points that project to the relative interior of the imaginary wall, our argument requires us to first understand the coefficient-free case.
In that case, the line segments are parallel to the imaginary ray and have one endpoint at $\lambda\in\reals^n$ and the other endpoint on the relative boundary of the imaginary wall.
%However, in affine type, our argument to deal with points that project to the relative interior of the imaginary wall passes through understanding the coefficient-free case where we can describe the line segments explicitly:
%They are parallel to the imaginary ray and have one endpoint at $\lambda$ and the other endpoint on the relative boundary of the imaginary wall.
From there, we can give the same characterization, for vectors $\tilde\lambda$ that project to the imaginary wall, for arbitrary $\tB$ of affine type, with no condition on linear independence of columns.

%In general, it appears to be much harder to compute dominance regions without the assumption that $\tB$ has linearly independent columns.
%However, our characterization of dominance regions in affine type in the case of linearly independent columns requires us first to characterize, in the coefficient-free case, dominance regions for $\lambda$ in the imaginary cone.
%(From there, we can give the same characterization for $\lambda$ in the imaginary cone for arbitrary $\tB$ of affine type, with no condition on linear independence of columns.)
%In the coefficient-free case, we can describe the line segments explicitly:
%They are parallel to the imaginary ray and have one endpoint at $\lambda$ and the other endpoint on the relative boundary of the imaginary wall.

A crucial ingredient for the characterization of dominance regions in affine type is a detailed study of the exchange matrices of neighboring seeds.  
(One aspect of this study appears in work of Felikson and Tumarkin~\cite{FeTuCoeff} and of Greenberg and Kaufman~\cite{GreenbergKaufman}.
See Remark~\ref{related remark}.)
These are seeds that have $n-2$ of their $\g$-vectors in the imaginary wall (the largest possible number).
It is apparent already from \cite{affdenom} that the combinatorics of neighboring seeds involves a product of finite type~B or~C generalized associahedra \cite[Propositions~5.13, 6,13]{affdenom}.  
(The work of McCammond and Sulway on Artin groups of Euclidean type~\cite{McSul} finds, in an affine version of noncrossing partition lattices, products of finite type-B noncrossing partition lattices in a closely analogous way.)
In this paper, we make the appearance of finite type B/C combinatorics in mutation fans of affine type completely explicit.
Specifically, we show that with respect to a neighboring seed having exchange matrix $B$, the imaginary wall is a half-hyperplane and the restriction of the mutation fan to the imaginary wall is (\emph{metrically}, not merely \emph{combinatorially}) a product of the imaginary ray and some number ($1$, $2$ or $3$) of mutation fans of finite type~C.
We also show that certain sequences of mutations of $B$ correspond to (and give the same mutation maps as) mutations of the corresponding finite type-C exchange matrices.

To use these insights to compute dominance regions in affine type, we were obliged to prove that, in the finite-type coefficient-free case, each dominance region is a single point.
A proof of this statement, which we had initially assumed would be trivial, turned out to be very difficult to find.
The one we give involves a type-by-type argument using the marked surfaces model, folding, and computations.
As a consequence, we also show that dominance regions are singletons in finite type with arbitrary coefficients.
%As a consequence, we also show that dominance regions in finite type are singletons for arbitrary coefficients.

For our applications, the relevant notion is the integral dominance region: the set of points in $\P_{\tilde\lambda}^\tB$ that are obtained as $\tilde\lambda$ plus an \emph{integer}-linear combination of columns of $\tB$.
We describe the integral dominance region in affine type in Section~\ref{sec:integral}.

\section{Dominance regions}
In this section, we define dominance regions and prove some foundational facts about them.
We assume the basic background on cluster algebras, following the approach of~\cite{ca4}.
Also, throughout the section, we use results from~\cite{NZ}.

\subsection{Definitions and background}\label{def sec}
We begin with a skew-symmetrizable $n \times n$ exchange matrix~$B=[b_{ij}]$ and write $d_1,\ldots,d_n$ for the skew-symmetrizing constants of $B$ (so that $d_i b_{ij}=-d_j b_{ji}$ for all $i,j$).
The diagonal matrix $\symmetrizer$ with diagonal entries $d_1,\ldots,d_n$ has $\symmetrizer B\symmetrizer^{-1}=-B^T$.

Given a sequence $\kk=k_r\cdots k_1$ of indices in $\set{1,\ldots,n}$, we read the sequence from right to left for the purposes of matrix mutation.
That is, $\mu_\kk(B)$ means $\mu_{k_r}(\mu_{k_{r-1}}(\cdots(\mu_{k_1}(B))\cdots))$ and, when a sequence $\kk$ has been chosen, we often denote this mutated exchange matrix simply as $B_r$.
We write $\kk^{-1}$ for $k_1\cdots k_r$, the reverse of $\kk$.
Throughout, we will use without comment the fact that matrix mutation commutes with the maps $B\mapsto-B$ and $B\mapsto B^T$.

Given an exchange matrix $B$, the \newword{mutation map} $\eta^B_\kk:\reals^n\to\reals^n$ takes the input vector in $\reals^n$, places it as an additional row below $B$, mutates the resulting matrix according to the sequence $\kk$, and outputs the bottom row of the mutated matrix \cite[Definition~4.1]{universal}.
In this paper, it is convenient to think of vectors in $\reals^n$ as column vectors and thus the mutation maps we need use transposes $B^T$ of exchange matrices, i.e. we use the mutation maps $\eta_\kk^{B^T}$.
One such map takes a vector, places it as an additional \emph{column} to the right of $B$ (not $B^T$), does mutations according to $\kk$, and reads the rightmost column of the mutated matrix.
The inverse of the mutation map $\eta^{B^T}_\kk$ is $\eta_{\kk^{-1}}^{\mu_\kk(B)^T}$.
In this paper we seldom use mutation maps without a transpose and so we could have opted to drop the superscript $^T$ in our notation.
We refrained from this modification because we will need to discuss and adapt parts of some proofs already written elsewhere with our current notation.

Given $\lambda\in\reals^n$, define $\P^B_{\lambda,\kk}=\bigl(\eta_{\kk}^{B^T}\bigr)^{-1}\sett{\eta_\kk^{B^T}(\lambda)+\mu_\kk(B)\alpha:\alpha\in\reals^n,\alpha\ge0}$, where the symbol $\ge$ denotes componentwise comparison.
(Throughout the paper, we will define sets indexed by vectors $\alpha\in\reals^n$ with $\alpha\ge0$, or sometimes $\alpha\in\reals^m$ with $\alpha\ge0$.
When we can do so without confusion, we will omit the explicit statement that $\alpha\in\reals^n$ or $\alpha\in\reals^m$.)
Define the \newword{dominance region} of $\lambda$ with respect to $B$ to be $\P^B_\lambda=\bigcap_\kk\P^B_{\lambda,\kk}$, where the intersection is over all sequences~$\kk$.
\begin{lemma}\label{shift}
Suppose $\lambda'=\eta^{B^T}_\kk(\lambda)$ and $B'=\mu_\kk(B)$.
\begin{enumerate}[label=\bf\arabic*., ref=\arabic*]
\item \label{shift one}
$\eta^{B^T}_\kk\!\!(\P^B_{\lambda,\ll})=\P^{B'}_{\lambda',\ll\kk^{-1}}$ for any $\ll$.
\item \label{shift all}
$\eta^{B^T}_\kk\!\!(\P^B_\lambda)=\P^{B'}_{\lambda'}$.
\end{enumerate}
\end{lemma}
\begin{proof}
For any $\ll$,
\begin{align*}
\eta^{B^T}_\kk\!\!(\P^B_{\lambda,\ll})
&=\eta^{B^T}_\kk\!\!\left(\bigl(\eta_{\ll}^{B^T}\bigr)^{-1}\settt{\eta_\ll^{B^T}(\lambda)+\mu_\ll(B)\alpha:\alpha\ge0}\right)\\
%&=\left(\eta^{B^T}_\ll\!\!\left(\eta_{\kk}^{B^T}\right)^{-1}\right)^{-1}\set{\eta_\ll^{B^T}(\lambda)+\mu_\ll(B)\alpha:\alpha\ge0}\\
&=\left(\eta^{B^T}_\ll\!\!\eta_{\kk^{-1}}^{\mu_\kk(B)^T}\right)^{-1}\settt{\eta_\ll^{B^T}\!\!\eta_{\kk^{-1}}^{\mu_\kk(B)^T}\!\!\left(\eta^{B^T}_\kk(\lambda)\right)+\mu_\ll(B)\alpha:\alpha\ge0}\\
%&=\left(\eta^{\mu_\kk(B)^T}_{\ll\kk^{-1}}\right)^{-1}\set{\eta_\ll^{B^T}(\lambda)+\mu_\ell(B)\alpha:\alpha\ge0}\\
&=\left(\eta^{\mu_\kk(B)^T}_{\ll\kk^{-1}}\right)^{-1}\settt{\eta^{\mu_\kk(B)^T}_{\ll\kk^{-1}}\!\!\left(\eta^{B^T}_\kk(\lambda)\right)+\mu_{\ll\kk^{-1}}(B')\alpha:\alpha\ge0}\\
&=\P^{B'}_{\lambda',\ll\kk^{-1}}.
\end{align*}
  Thus $\eta^{B^T}_\kk\!\!(\P^B_\lambda)=\bigcap_\ll\eta^{B^T}_\kk\!\!(\P^B_{\lambda,\ll})=\bigcap_\ll\P^{B'}_{\lambda',\ll\kk^{-1}}=\P^{B'}_{\lambda'}$.
\end{proof}

\begin{lemma}\label{block dom}
If $B$ has a diagonal block decomposition $\begin{bsmallmatrix}B_1&0\\0&B_2\end{bsmallmatrix}$ and $\lambda\in\reals^n$ has a corresponding decomposition $\begin{bsmallmatrix}\lambda_1\\\lambda_2\end{bsmallmatrix}$, then $\P_\lambda^B=\P_{\lambda_1}^{B_1}\times\P_{\lambda_2}^{B_2}$.
\end{lemma}
\begin{proof}
  For any column $k$ of $B_1$, we have $\mu_k\begin{bsmallmatrix}B_1&0\\0&B_2\end{bsmallmatrix}=\begin{bsmallmatrix}\mu_k(B_1)&0\\0&B_2\end{bsmallmatrix}$ and similarly for columns of $B_2$.
  This implies that the mutations corresponding to columns of $B_1$ and to columns of $B_2$ will commute and thus mutations corresponding to columns of $B_i$ only affect $\lambda_i$.
\end{proof}

For seeds $t_0$ and $t$ and an exchange matrix $B$, let $C_t^{B;t_0}$ be the matrix whose columns are the $C$-vectors at $t$ relative to the initial seed $t_0$ with exchange matrix~$B$.
Each column of $C_t^{B;t_0}$ is nonzero and all of its nonzero entries have the same sign.
(This is ``sign-coherence of $C$-vectors'', which was implicitly conjectured in \cite{ca4} and proved as \cite[Corollary~5.5]{GHKK}.)
Thus we may refer to the \newword{sign} of a column of $C_t^{B;t_0}$.
For $\kk=k_r\cdots k_1$, define seeds $t_1,\ldots,t_r$ by $t_0\overset{k_1}{\edge}t_1\overset{k_2}{\edge}\,\cdots\,\overset{k_r}{\edge}t_r$.
The sequence $\kk$ is a \newword{green sequence} for the exchange matrix $B$ if column $k_p$ of $C_{t_{p-1}}^{B;t_0}$ is \emph{positive} for all $p$ with $1\le p\le r$.
A \newword{maximal green sequence} for $B$ is a green sequence that cannot be extended.
That is, a green sequence $\kk$ is a maximal green sequence if every column of $C_{t_r}^{B;t_0}$ is \emph{negative}.

We will call $\kk$ a \newword{red sequence} for~$B$ if it is a green sequence for $-B$.
A \newword{maximal red sequence} is a red sequence that cannot be extended.
(A red sequence relates to antiprincipal coefficients: 
if we were to define a version of $C$-vectors recursively starting with the negative of the identity matrix, the requirement for a red sequence would be that the $k_p$ column is negative at every step.)

Let $G_t^{B;t_0}$ be the matrix whose columns are the $\g$-vectors at $t$ relative to the initial seed $t_0$ with exchange matrix~$B$.
Let $\Cone^{B;t_0}_t$ be the nonnegative linear span of the columns of $G_t^{B;t_0}$.
For each $k\in\set{1,\ldots,n}$, the entries in the $k\th$ row of $G_t^{B;t_0}$ are not all zero and the nonzero entries have the same sign.
(This is ``sign-coherence of $\g$-vectors'', conjectured as \cite[Conjecture~6.13]{ca4} and proved as \cite[Theorem 5.11]{GHKK}.)
Thus vectors in $\Cone^{B;t_0}_t$ all have weakly the same sign in the $k\th$ position.
The inverse of $G_t^{B;t_0}$ is $\bigl(C_t^{-B^T;t_0}\bigr)^T$.
(This is \cite[Theorem~1.2]{NZ} or \cite[Theorem~1.1]{framework} and \cite[Theorem~3.30]{framework}.)
Thus $\Cone^{B;t_0}_t=\sett{x\in\reals^n:x^TC_t^{-B^T;t_0}\ge0}$, where $0$ is a row vector and ``$\ge$'' means componentwise comparison. 

Given $\kk$ with $t_0\overset{k_1}{\edge}t_1\overset{k_2}{\edge}\,\cdots\,\overset{k_r}{\edge}t_r$, let $B_p$ be the exchange matrix at~$t_p$, so that in particular, $B_0=B$.
The map $\eta_{\kk}^{B^T}$ is ${\eta_{k_r}^{B_{r-1}^T}\circ\cdots\circ\eta_{k_2}^{B_1^T}\circ\eta_{k_1}^{B_0^T}}$.
The definition of each $\eta_{k_p}^{B_{p-1}^T}$ has two cases, separated by the hyperplane $x_{k_p}=0$.
Two vectors are in the same \newword{domain of definition} of $\eta_\kk^{B^T}$ if the same case applies for the two vectors at every step.
(Both cases apply on the separating hyperplanes, so domains of definition are closed.)
In particular, $\eta_\kk^{B^T}$ is linear in each of its domains of definition, but the domains of linearity of $\eta_\kk^{B^T}$ can be larger than its domains of definition.

There is a fan $\F_{B^T}$ called the \newword{mutation fan} for $B^T$ \cite[Definition~5.12]{universal}.
We will not need the details of the definition, but roughly, the cones of $\F_{B^T}$ are the intersections of domains of definition of mutation maps $\eta_\kk^{B^T}$ for all sequences $\kk$.
Thus for each $\kk$, each cone of $\F_{B^T}$ is contained in a domain of definition of $\eta_\kk^{B^T}$, and the mutation map $\eta_\kk^{B^T}$ is linear on every cone of $\F_{B^T}$ \cite[Proposition~5.3]{universal}.
Every cone $\Cone^{B;t_0}_t$ is a maximal cone in the mutation fan $\F_{B^T}$ \cite[Proposition~8.13]{universal}.
Thus in particular, the mutation map $\eta_\kk^{B^T}$ is linear on every cone $\Cone^{B;t_0}_t$.
Furthermore, $\Cone_t^{B_r;t_r}=\eta_\kk^{B^T}\bigl(\Cone_t^{B;t_0}\bigr)$ for every seed~$t$.
(This amounts to the initial seed mutation formula for $\g$-vectors, conjectured as \cite[Conjecture~7.12]{ca4} and shown in \cite[Proposition~4.2(v)]{NZ} to follow from sign-coherence of $C$-vectors.
The restatement in terms of mutation maps is \cite[Conjecture~8.11]{universal}.)
A similar fact applies to the entire mutation fan:
for any sequence $\kk$, the map $\eta_\kk^{B^T}$ induces an isomorphism of fans \cite[Proposition~7.3]{universal} from $\F_{B^T}$ to $\F_{\mu_\kk(B)^T}$.

\begin{remark}\label{conditional}
As written, \cite[Proposition~8.13]{universal} is conditional on the conjectured sign-coherence of $C$-vectors, which is now a theorem \cite[Corollary~5.5]{GHKK}.
\end{remark}

We will need to relate the cones $\Cone^{B;t_0}_t$ and $\Cone^{-B^T;t_0}_t$.
Recall that $-B^T=\symmetrizer B\symmetrizer^{-1}$, where $\symmetrizer$ is the diagonal matrix with positive entries $d_1,\ldots,d_n$.
In the language of \cite[Section~7]{universal}, $-B^T$ is a \newword{rescaling} of $B$.
Therefore, \cite[Proposition~8.20]{universal} says that the $i\th$ column of $G_t^{-B^T;t_0}$ is a positive scalar multiple of the $i\th$ column of $\symmetrizer^{-1}G_t^{B;t_0}$.
(In the statement of \cite[Proposition~8.20]{universal}, the roles of $\symmetrizer$ and $\symmetrizer^{-1}$ are reversed, and $\symmetrizer$ is multiplied on the right because, in~\cite{universal}, the $\g$-vectors are row vectors rather than column vectors.)
Thus we have the following fact.
\begin{lemma}\label{B or -BT}
For any $k$, the $k\th$ entries of vectors in $\Cone^{B;t_0}_t$ have the same sign as the $k\th$ entries of vectors in $\Cone^{-B^T;t_0}_t$.
\end{lemma}
%Therefore, \cite[Proposition~8.20]{universal} implies the following lemma.
%(In the statement of \cite[Proposition~8.20]{universal}, $\symmetrizer$ is multiplied on the right, because there $\g$-vectors are row vectors rather than column vectors.)
%\begin{lemma}\label{SigmaG}
%For $i=1,\ldots,n$, the $i\th$ column of $G_t^{-B^T;t_0}$ is a positive scalar multiple of the $i\th$ column of $\symmetrizer G_t^{B;t_0}$.
%\end{lemma}
%The following lemma is an immediate consequence.
%\begin{lemma}\label{B or -BT}
%The $k\th$ entries of vectors in $\Cone^{B;t_0}_t$ have the same sign as the $k\th$ entries of vectors in $\Cone^{-B^T;t_0}_t$.
%\end{lemma}

\begin{lemma}\label{max green neg cone}
If $\kk=k_r\cdots k_1$ is a maximal green sequence for $B$ corresponding to seeds $t_0\overset{k_1}{\edge}t_1\overset{k_2}{\edge}\,\cdots\,\overset{k_r}{\edge}t_r=t$ in the cluster pattern with $B$ at $t_0$, then ${\Cone_t^{B;t_0}=\Cone_t^{-B^T;t_0}=\reals_{\le0}^n}$.
\end{lemma}
\begin{proof}
Since $\kk$ is a maximal green sequence for $B$, every column of $C_t^{B;t_0}$ has negative sign and thus $\Cone_t^{-B^T;t_0}=\set{x\in\reals^n:x^TC_t^{B;t_0}\ge0}$ consists of vectors with nonpositive entries.
Since $\reals_{\le0}^n$ is a cone in the mutation fan $\F_B$ (for example, combining \mbox{\cite[Proposition~7.1]{universal}}, \mbox{\cite[Proposition~8.9]{universal}}, and sign-coherence of $C$-vectors) and also $\Cone_t^{-B^T;t_0}$ is a cone in $\F_B$, we see that $\Cone_t^{-B^T;t_0}=\reals_{\le0}^n$.
Lemma~\ref{B or -BT} implies that $\Cone_t^{B;t_0}=\Cone_t^{-B^T;t_0}$.
\end{proof}

\begin{lemma}\label{max red green}
For a sequence $t_0\overset{k_1}{\edge}t_1\overset{k_2}{\edge}\,\cdots\,\overset{k_r}{\edge}t_r=t$ of mutations in a cluster pattern, the sequence $\kk=k_r\cdots k_1$ of indices is a maximal green sequence for $B_0$ if and only if $\kk^{-1}=k_1\cdots k_r$ is a maximal red sequence for~$B_r$.
\end{lemma}
\begin{proof}
By the definition of a red sequence, the lemma is equivalent to the statement that  $\kk=k_r\cdots k_1$ is a maximal green sequence for $B_0$ if and only if $\kk^{-1}=k_1\cdots k_r$ is a maximal green sequence for~$-B_r$.

Suppose $\kk$ is a maximal green sequence for $B_0$.
Lemma~\ref{max green neg cone} implies that, up to permuting columns, $C_t^{B_0;t_0}$ is the negative of the identity matrix.
In other words, mutating the matrix $\begin{bsmallmatrix}B_0\\I\end{bsmallmatrix}$ by the sequence $\kk$ yields the matrix $\begin{bsmallmatrix}B_r\\-P\end{bsmallmatrix}$ for some permutation matrix $P$.
Each mutation in the sequence is at the index of a column such that the bottom half of the matrix has nonnegative entries, and the mutation turns those entries nonpositive.
Since mutation commutes with negation and with reordering of the rows $\{n+1,\dots 2n\}$, mutating the matrix $\begin{bsmallmatrix}-B_r\\I\end{bsmallmatrix}$ by the sequence $\kk^{-1}$ yields the matrix $\begin{bsmallmatrix}-B_0\\-P^{-1}\end{bsmallmatrix}$, and each mutation is at the index of a column where the bottom half of the matrix is positive.
Thus $\kk^{-1}$ is a maximal green sequence for $-B_r$.
The opposite implication is symmetric.
\end{proof}

For $k\in\set{1,\ldots,n}$, let $J_k$ be the $n\times n$ matrix that agrees with the identity matrix except that $J_k$ has $-1$ in position $kk$.
For an $n\times n$ matrix $M$ and $k\in\set{1,\ldots,n}$, let $M^{\bullet k}$ be the matrix that agrees with $M$ in column $k$ and has zeros everywhere outside of column $k$.
Let $M^{k\bullet}$ be the matrix that agrees with $M$ in row $k$ and has zeros everywhere outside of row $k$.
Given a real number $a$, let $[a]_+$ denote $\max(a,0)$.
Given a matrix $M=[m_{ij}]$, define $[M]_+$ to be the matrix whose $ij$-entry is $[m_{ij}]_+$.
Given an exchange matrix~$B$, an index $k\in\set{1,\ldots,n}$ and a sign $\ep\in\set{\pm1}$, define
\begin{align*}
E_{\ep,k}^B&=J_k+[\ep B]_+^{\bullet k}\\
F_{\ep,k}^B&=J_k+[-\ep B]_+^{k\bullet}.
\end{align*}
Each matrix $E_{\ep,k}^B$ is its own inverse, and each $F_{\ep,k}^B$ is its own inverse.
We note that $(E_{\ep,k}^B)^T=J_k+[\ep B^T]_+^{k\bullet}=F_{\ep,k}^{-B^T}=\symmetrizer F_{\ep,k}^B\symmetrizer^{-1}$, where~$\symmetrizer$ is the diagonal matrix of skew-symmetrizing constants of~$B$.
The following is a special case of \cite[(3.2)]{ca3}.
\begin{lemma}\label{EBF trick}
For $k\in\set{1,\ldots,n}$ and either choice of $\ep\in\set{\pm1}$, the mutation of $B$ at $k$ is $\mu_k(B)=E_{\ep,k}^BBF_{\ep,k}^B$.
\end{lemma}
%\begin{proof}
%We expand the product $(J_k+[\ep B]_+^{\bullet k})B(J_k+[-\ep B]_+^{k\bullet})$ to four terms.
%The term $[\ep B]_+^{\bullet k}B[-\ep B]_+^{k\bullet}$ is zero because $B$ is skew-symmetric.
%The term $[\ep B]_+^{\bullet k}BJ_k$ is $[\ep B]_+^{\bullet k}B^{k\bullet}J_k$, which equals $[\ep B]_+^{\bullet k}B^{k\bullet}$ because $b_{kk}=0$.
%Similarly, the term $J_kB[-\ep B]_+^{k\bullet}$ equals $B^{\bullet k}[-\ep B]_+^{k\bullet}$.
%%\begin{align*}
%%(J_k+[\ep B]_+^{\bullet k})B(J_k+[-\ep B]_+^{k\bullet})
%%&=J_kBJ_k+J_kB[-\ep B]_+^{k\bullet}+[\ep B]_+^{\bullet k}BJ_k+[\ep B]_+^{\bullet k}B[-\ep B]_+^{k\bullet}.
%%\end{align*}
%Thus the $ij$-entry of $E_{\ep,k}^BBF_{\ep,k}^B$ is 
%\[
%\left\{\begin{aligned}
%-b_{ij}&\quad\text{if }k\in\set{i,j}\\
%b_{ij}&\quad\text{otherwise}
%\end{aligned}\right\}
%+
%\left\{\begin{aligned}
%|b_{ik}|b_{kj}&\quad\text{if }\sgn b_{ik}=\ep\\
%0&\quad\text{otherwise}
%\end{aligned}\right\}
%+
%\left\{\begin{aligned}
%b_{ik}|b_{kj}|&\quad\text{if }\sgn b_{kj}=-\ep\\
%0&\quad\text{otherwise}
%\end{aligned}\right\}.
%\]
%This coincides with the $ij$-entry of $\mu_k(B)$.
%\end{proof}

Given a matrix $M$, write $M_{\col j}$ for the $j\th$ column of $M$.
For matrices $M$ and~$N$ such that $MN$ is defined, we observe that $(MN)_{\col j}=M(N)_{\col j}$.
\begin{lemma}\label{columns lem}
Suppose $B=[b_{ij}]$ is an exchange matrix, let $k\in\set{1,\ldots,n}$, and choose a sign $\ep\in\set{\pm1}$.
\begin{enumerate}[label=\bf\arabic*., ref=\arabic*]
\item \label{col j}
$(E_{\ep,k}^BB)_{\col j}=J_k(B)_{\col j}+b_{kj}([\ep B]_+)_{\col k}$.
\item \label{col k}
$(E_{\ep,k}^BB)_{\col k}=(E_{-\ep,k}^BB)_{\col k}=B_{\col k}$.
\item \label{cols k}
$(E_{-\ep,k}^BB)_{\col j}=(E_{\ep,k}^BB)_{\col j}-\ep b_{kj}B_{\col k}$.
\end{enumerate}
\end{lemma}
\begin{proof}
The first two assertions hold because $(MN)_{\col j}=M(N)_{\col j}$ and because $b_{kk}=0$.
By the first assertion, $(E_{-\ep,k}^BB)_{\col j}=(E_{\ep,k}^BB)_{\col j}-b_{kj}([\ep B]_+-[-\ep B]_+)_{\col k}$.  
The third assertion follows.
\end{proof}

%We will also need the following simple fact about nonnegative linear spans.
Let~$\nnspan(S)$ denote the nonnegative linear span of a set $S$ of vectors.
For ${k\in\set{1,\ldots,n}}$ and $\ep\in\set{\pm1}$, let $S_{k,\ep}$ be the set of vectors in $S$ whose $k\th$ entry has sign strictly agreeing with $\ep$.

\begin{lemma}\label{ps lemma}
Suppose $\lambda$ is a vector in $\reals^n$ whose $k\th$ entry $\lambda_k$ has $\ep\lambda_k\le0$.
Then %any vector in $\set{\lambda+\nnspan(S)}\cap\set{x\in\reals^n:\sgn x_k=\ep}$ can be written as a vector in $\set{\lambda+\nnspan(S)}\cap\set{x\in\reals^n:x_k=0}$ plus a vector in $nnspan(S_{k,\ep})$.
\begin{multline*}
\sett{\lambda+\nnspan(S)}\cap\set{x\in\reals^n:\ep x_k\ge0}\\
=\bigl(\sett{\lambda+\nnspan(S)}\cap\set{x\in\reals^n:x_k=0}\bigr)+\nnspan(S_{k,\ep}).
\end{multline*}
\end{lemma}
\begin{proof}
The set on the right side is certainly contained in the set on the left side.
If~$x$ is an element of the left side, then $x$ is $\lambda$ plus a nonzero element $y$ of $\nnspan(S_{k,\ep})$ plus an element $z$ of $\nnspan(S\setminus S_{k,\ep})$.
Since $\ep x_k\ge0$ and $\ep\lambda_k\le0$, there exists~$t$ with ${0\le t\le1}$ such that $\lambda+ty+z$ has $k\th$ entry $0$.
Thus ${x=(\lambda+ty+z)+(1-t)y}$ is an element of the right side.
\end{proof}

\subsection{Folding the dominance region}\label{fold sec}
Folding is a standard technique from the theory of Cartan matrices and root systems (see, for example,~\cite{StembridgeFolding}) that can sometimes be applied to exchange matrices.
We provide details, following the exposition in \cite[Sections~2.3--2.4]{VielThesis}, which cites results from \cite{CaoHuangLi,DemonetThesis,DupontNonSimp,NS14}.
An \newword{admissible automorphism} of an exchange matrix $B$ is a permutation $\sigma$ of the indices of~$B$ such that $b_{ij}=0$ whenever~$i$ and~$j$ are in the same $\sigma$-orbit and, for $i$ and $j$ in different orbits, the weak sign of $b_{ij}$ depends only on the $\sigma$-orbits of $i$ and $j$.
For each index $i$, define~$\bar\imath$ to be the $\sigma$-orbit of $i$ and reuse the notation $\bar\imath$ to stand for a word that lists the $\sigma$-orbit of $i$ in any order.
Given an admissible automorphism, an \newword{orbit sequence} is a sequence $\kk$ of indices of $B$ of the form $\kk=\bar\imath_1\cdots\bar\imath_p$.
A \newword{stable automorphism} is an admissible automorphism $\sigma$ that is also an admissible automorphism of $\mu_\kk(B)$ for any \emph{orbit sequence}~$\kk$.

The \newword{folding} of $B$ with respect to $\sigma$ is the exchange matrix $\pi_\sigma(B)$ indexed by the $\sigma$-orbits, whose $\bar\imath\bar\jmath$-entry is $\sum_{i'\in\bar\imath}b_{i'j'}$ for any $j'\in\bar\jmath$.
An orbit sequence can be viewed as a sequence of indices of $B$ or as a sequence of indices of $\pi_\sigma(B)$.
We pass freely between the two points of view, even within the same equation, because the correct point of view is easily seen from context.
For any orbit sequence $\kk$, we have $\pi_\sigma(\mu_\kk(B))=\mu_\kk(\pi_\sigma(B))$.

The permutation $\sigma$ of indices can be interpreted as a permutation of the standard basis $\e_1,\ldots,\e_n$ of~$\reals^n$ and thus as a linear map on $\reals^n$.
Write~$(\reals^n)^\sigma$ for the subspace of $\reals^n$ consisting of points fixed by~$\sigma$.
The fixed space has a basis indexed by the $\sigma$-orbits, with the basis vector for $\bar\imath$ given by $\sum_{i'\in\bar\imath}\e_{i'}$.
We call this basis the \newword{orbit basis} of~$(\reals^n)^\sigma$.
The following fact is \cite[Lemma~2.4.6]{VielThesis}. 
(In fact, we quote the lemma with an additional transpose, but all of the folding constructions commute with taking the transpose of~$B$.)
%A permutation $\sigma$ is a stable automorphism of $B$ if and only if it is a stable automorphism of~$B^T$.
\begin{proposition}\label{VielLemma}
Suppose $\sigma$ is a stable automorphism of $B$.
Then for any orbit sequence $\kk$, the mutation map $\eta^{B^T}_\kk$, defined as usual in terms of the standard basis of $\reals^n$, fixes $(\reals^n)^\sigma$ as a set and, on $(\reals^n)^\sigma$, agrees with the mutation map $\eta_\kk^{\pi_\sigma(B)^T}$ defined in terms of the orbit basis of $(\reals^n)^\sigma$.
\end{proposition}

For a vector $\alpha\in(\reals^n)^\sigma$, we say $\alpha\ge0$ to mean that $\alpha$ has nonnegative coordinates in the orbit basis and we interpret the notation $\pi_\sigma(B)\alpha$ as the matrix $\pi_\sigma(B)$ times the column vector of orbit-basis coordinates of $\alpha$.
We can thus interpret $\pi_\sigma(B)\alpha$ as a vector in $(\reals^n)^\sigma$.
We observe that when $\sigma$ is a stable automorphism of $B$,
\[\sett{B\alpha:\alpha\in\reals^n,\alpha\ge0}\cap(\reals^n)^\sigma=\sett{\pi_\sigma(B)\alpha:\alpha\in(\reals^n)^\sigma,\alpha\ge0}.\]
This observation, combined with Proposition~\ref{VielLemma}, immediately implies the following theorem.
In the theorem, we interpret $\P_\lambda^{\pi_\sigma(B)}$ as the dominance region of $\lambda$ inside $(\reals^n)^\sigma$ defined in terms of the orbit basis of $(\reals^n)^\sigma$.
\begin{theorem}\label{fold dom reg}
Suppose $\sigma$ is a stable automorphism of $B$ and $\lambda\in(\reals^n)^\sigma$.
Then $\P_{\lambda,\kk}^{\pi_\sigma(B)}=\P_{\lambda,\kk}^B\cap(\reals^n)^\sigma$ for all orbit sequences~$\kk$.
\end{theorem}

\subsection{Linearizing the dominance region}\label{lin sec}
The difficult thing about computing the regions $\P^B_{\lambda,\kk}$ whose intersection is the dominance region $\P^B_\lambda$ is applying the piecewise linear map $\bigl(\eta^{B^T}_\kk\bigr)^{-1}$, because the number of its domains of definition may grow without bound as the length of $\kk$ increases.
In this section, we describe the cones (at $\lambda$) that coincide with $\P^B_{\lambda,\kk}$ inside the domain of definition of $\bigl(\eta^{B^T}_\kk\bigr)^{-1}$ containing~$\lambda$ and show, in some cases, that these cones contain $\P^B_{\lambda,\kk}$.
We work here with square exchange matrices.
In Section~\ref{ext sec}, we extend the result to tall extended exchange matrices.
%As an application, we prove in Section~\ref{ext sec} that (when $B$ is replaced by an extended exchange matrix with linearly independent columns) that the dominance region of a point $\lambda$ in the $\g$-vector fan... hypotheses... never mind

Let $B_0$ be an exchange matrix.
For a sequence $\kk=k_r\cdots k_1$ of indices, define seeds $t_1,\ldots,t_r=t$ by $t_0\overset{k_1}{\edge}t_1\overset{k_2}{\edge}\,\cdots\,\overset{k_r}{\edge}t_r=t$.
We will prove the following theorem.

\begin{theorem}\label{P in B0C}
Suppose $\kk=k_r\cdots k_1$ and $t_0\overset{k_1}{\edge}t_1\overset{k_2}{\edge}\,\cdots\,\overset{k_r}{\edge}t_r=t$.
If $\kk^{-1}=k_1\cdots k_r$ is a red sequence for $B_r$, then for any~$\lambda$ in the domain of definition of $\eta_\kk^{B_0^T}$ that contains $\Cone^{B_0;t_0}_t$,
\[\P^{B_0}_{\lambda,\kk}\subseteq\set{\lambda+G_t^{B_0;t_0}B_r\alpha:\alpha\in\reals^n,\alpha\ge0}=\set{\lambda+B_0C_t^{B_0;t_0}\alpha:\alpha\in\reals^n,\alpha\ge0}.\]
\end{theorem}

Since $\bigl(\eta_{\kk}^{B_0^T}\bigr)^{-1}=\eta_{\kk^{-1}}^{B_r^T}$, we have $\P^{B_0}_{\lambda,\kk}=\eta_{\kk^{-1}}^{B_r^T}\sett{\eta_\kk^{B_0^T}(\lambda)+B_r\alpha:\alpha\ge0}$.
Let $\Sigma$ be the domain of definition of $\eta_{\kk}^{B_0^T}$  that contains $\Cone^{B_0;t_0}_t$.
Then $\eta_{\kk^{-1}}^{B_r^T}$ is linear on $\eta_{\kk}^{B_0^T}(\Sigma)$.
Let~$\L_{\kk^{-1}}^{B_r^T}$ be the linear map that agrees with $\eta_{\kk^{-1}}^{B_r^T}$ on~$\eta_{\kk}^{B_0^T}(\Sigma)$.

\begin{proposition}\label{L mat}
The matrix for $\L_{\kk^{-1}}^{B_r^T}$, acting on column vectors, is $G_t^{B_0;t_0}$.
\end{proposition}
\begin{proof}
By \cite[Proposition~8.13]{universal} we have $\Cone^{B_0;t_0}_t=\eta_{\kk^{-1}}^{B_r^T}(\reals_{\ge0}^n)$, and therefore also ${\eta_\kk^{B_0^T}\bigl(\Cone^{B_0;t_0}_t\bigr)=\reals_{\ge0}^n}$.
The proof of \cite[Proposition~8.13]{universal} shows not only an equality of cones, but also that $\eta_{\kk^{-1}}^{B_r^T}$ takes the extreme ray of $\reals_{\ge0}^n$ spanned by $\e_i$ to the extreme ray of $\Cone^{B_0;t_0}_t$ spanned by the $i\th$ $\g$-vector at $t$ relative to $B_0;t_0$, where the total order on these $\g$-vectors at $t$ is obtained from the order $\e_1,\ldots,\e_n$ on $\g$-vectors at $t_0$ by the sequence of mutations $\kk$.
\end{proof}

The following is a result of \cite{NZ}.
It follows from the proof of \cite[Proposition~1.3]{NZ}, or from \cite[(6.14)]{ca4}, as explained in \cite[Remark~2.1]{NZ}.

\begin{proposition}\label{GBBC}
$G_t^{B_0;t_0}B_r=B_0C_t^{B_0;t_0}$.
\end{proposition}

We can rewrite Proposition~\ref{GBBC} as follows.

\begin{proposition}\label{BGCB}
$B_r\bigl(G_t^{-B_0^T;t_0}\bigr)^T=\bigl(C_t^{-B_0^T;t_0}\bigr)^TB_0$.
\end{proposition}
\begin{proof}
We use \cite[Theorem~1.2]{NZ} to rewrite $G_t^{B_0;t_0}$ as $\bigl(C_t^{-B_0^T;t_0}\bigr)^{-T}$ and $C_t^{B_0;t_0}$ as $\bigl(G_t^{-B_0^T;t_0}\bigr)^{-T}$, so Proposition~\ref{GBBC} says $\bigl(C_t^{-B_0^T;t_0}\bigr)^{-T}B_r=B_0\bigl(G_t^{-B_0^T;t_0}\bigr)^{-T}$.
\end{proof}

Since $G_t^{B_0;t_0}$ is the matrix for~$\L_{\kk^{-1}}^{B_r^T}$ and since $\L_{\kk^{-1}}^{B_r^T}\eta_\kk^{B_0^T}(\lambda)=\lambda$, the following result is immediate from Proposition~\ref{GBBC}.

\begin{proposition}\label{B0C}
\begin{align*}
\L_{\kk^{-1}}^{B_r^T}\sett{\eta_\kk^{B_0^T}(\lambda)+B_r\alpha:\alpha\in\reals^n,\alpha\ge0}
&=\sett{\lambda+G_t^{B_0;t_0}B_r\alpha:\alpha\in\reals^n,\alpha\ge0}\\
&=\sett{\lambda+B_0C_t^{B_0;t_0}\alpha:\alpha\in\reals^n,\alpha\ge0}.
\end{align*}
\end{proposition}

In light of Proposition~\ref{B0C}, the conclusion of Theorem~\ref{P in B0C} is equivalent to
\[\P^{B_0}_{\lambda,\kk}\subseteq\L_{\kk^{-1}}^{B_r^T}\sett{\eta_\kk^{B_0^T}(\lambda)+B_r\alpha:\alpha\ge0}.\]

\begin{proof}[Proof of Theorem~\ref{P in B0C}]
We will prove that $P^{B_0}_{\lambda,\kk}\subseteq\set{\lambda+B_0C_t^{B_0;t_0}\alpha:\alpha\ge0}$ by induction on $r$ (the length of $\kk$).
The base case, where $\kk=\varnothing$, is true because $C_{t_0}^{B_0;t_0}$ is the identity matrix and $\P_{\lambda,\varnothing}=\set{\lambda+B_0\alpha:\alpha\ge0}$.
 
\cite[Proposition~1.4]{NZ} says that $C_t^{B_0;t_0}=F^{B_1}_{\ep,k_1}C_t^{B_1;t_1}$, where $\ep$ is the sign of the $k_1$-column of $C_{t_1}^{-B_r;t}$.  
(The hypothesis that $\kk^{-1}$ is a red sequence for $B_r$ determines $\ep$, but we leave $\ep$ unspecified for now in order to highlight later where this hypothesis is relevant.)
By Lemma~\ref{EBF trick} and because $E^{B_1}_{\ep,k_1}$ and $F^{B_1}_{\ep,k_1}$ are their own inverses,
\begin{equation}\label{ind B0C}\begin{aligned}
\set{\lambda+B_0C_t^{B_0;t_0}\alpha:\alpha\ge0}
&=\set{\lambda+B_0F^{B_1}_{\ep,k_1}C_t^{B_1;t_1}\alpha:\alpha\ge0}\\
&=\set{\lambda+E^{B_1}_{\ep,k_1}B_1C_t^{B_1;t_1}\alpha:\alpha\ge0}\\
&=E^{B_1}_{\ep,k_1}\set{E^{B_1}_{\ep,k_1}\lambda+B_1C_t^{B_1;t_1}\alpha:\alpha\ge0}.
\end{aligned}\end{equation}

The map $\eta_{\kk}^{B_0^T}$ is linear on $\Cone_t^{B_0;t_0}$.  
This map is $\eta_{\kk}^{B_0^T}={\eta_{k_r}^{B_{r-1}^T}\circ\cdots\circ\eta_{k_2}^{B_1^T}\circ\eta_{k_1}^{B_0^T}}$.
The map $\eta_{k_1}^{B_0^T}$ restricts to a linear map from $\Cone_t^{B_0;t_0}$ to $\Cone_t^{B_1;t_1}$.
The inverse of $\eta_{k_1}^{B_0^T}$ is $\eta_{k_1}^{B_1^T}$.
We will find the matrix for the linear map on column vectors that agrees with $\eta_{k_1}^{B_1^T}$ on $\Cone_t^{B_1;t_1}$.
%Since $E^{B_1}_{\ep,k_1}$ is its own inverse, the claim is equivalent to saying that $E^{B_1}_{\ep,k_1}$ is the linear map that agrees with $\eta_{k_1}^{B_0^T}$ on $\Cone_t^{B_0;t_0}$.

By \cite[(1.13)]{NZ}, $\ep$ is the sign of the $k_1$-column of $\bigl(G_t^{-B_1^T;t_1}\bigr)^T$.
That is, $\ep$ is the sign of the $k_1$-row of $G_t^{-B_1^T;t_1}$, or in other words, the sign of the $k_1$-entry of vectors in $\Cone_t^{-B^T_1;t_1}$.
By Lemma~\ref{B or -BT}, $\ep$ is the sign of the $k_1$-entry of vectors in $\Cone_t^{B_1;t_1}$, which is the sign that determines how $\eta_{k_1}^{B_1^T}$ acts on $\Cone_t^{B_1;t_1}$.
Now, one easily checks that the action of $\eta_{k_1}^{B_1^T}$ on vectors whose $k_1$-entry has sign $\ep$ is precisely the action of $E^{B_1}_{\ep,k_1}$.
%It negates the $k_1$ entry.
%The $i\th$ ($i\neq k$) entry gets sent to itself plus the $k_1$ entry times $(B_1)_{ik}$ times $\ep$, if $(B_1)_{ik}$ also has sign $\ep$.

Let $\lambda'=\eta_{k_1}^{B_0^T}(\lambda)$, so that $\lambda'$ is in the same domain of definition of $\eta_{k_r\cdots k_2}^{B_1^T}$ as $\Cone_t^{B_1;t_1}$ and so that $\lambda'=E^{B_1}_{\ep,k_1}\lambda$.
By induction on $r$, 
\[\eta_{k_2\cdots k_r}^{B_r^T}\set{\eta_{k_r\cdots k_2}^{B_1^T}(\lambda')+B_r\alpha:\alpha\ge0}\subseteq\set{\lambda'+B_1C_t^{B_1;t_1}\alpha:\alpha\ge0}.\]
Applying the homeomorphism $\eta_{k_1}^{B_1^T}$ to both sides, we obtain
\[\eta_{\kk^{-1}}^{B_r^T}\set{\eta_\kk^{B_0^T}(\lambda)+B_r\alpha:\alpha\ge0}\subseteq\eta_{k_1}^{B_1^T}\set{\lambda'+B_1C_t^{B_1;t_1}\alpha:\alpha\ge0}.\]
In light of \eqref{ind B0C}, we can complete the proof by showing that
\[\eta_{k_1}^{B_1^T}\sett{\lambda'+B_1C_t^{B_1;t_1}\alpha:\alpha\ge0}\subseteq E^{B_1}_{\ep,k_1}\sett{\lambda'+B_1C_t^{B_1;t_1}\alpha:\alpha\ge0}.\]

We have seen that $E^{B_1}_{\ep,k_1}$ is the linear map that agrees with $\eta_{k_1}^{B_1^T}$ on the set $\set{x\in\reals^n:\sgn x_{k_1}=\ep}$.
We can similarly check that $E^{B_1}_{-\ep,k_1}$ is the linear map that agrees with $\eta_{k_1}^{B_1^T}$ on $\set{x\in\reals^n:\sgn x_{k_1}=-\ep}$.
Thus $\eta_{k_1}^{B_1^T}\set{\lambda'+B_1C_t^{B_1;t_1}\alpha:\alpha\ge0}$ is
\[(U\cap\set{x\in\reals^n:\sgn x_{k_1}=-\ep})\cup(V\cap\set{x\in\reals^n:\sgn x_{k_1}=\ep}),\]
where %, writing $\nnspan$ for the nonnegative linear span of a set of vectors,
{\small
\begin{align*}
U&=E^{B_1}_{\ep,k_1}\sett{\lambda'+B_1C_t^{B_1;t_1}\alpha:\alpha\ge0}=E^{B_1}_{\ep,k_1}\lambda'+\nnspan\settt{\bigl(E^{B_1}_{\ep,k_1}B_1C_t^{B_1;t_1}\bigr)_{\col i}}_{i=1}^n\\
V&=E^{B_1}_{-\ep,k_1}\sett{\lambda'+B_1C_t^{B_1;t_1}\alpha:\alpha\ge0}=E^{B_1}_{-\ep,k_1}\lambda'+\nnspan\settt{\bigl(E^{B_1}_{-\ep,k_1}B_1C_t^{B_1;t_1}\bigr)_{\col i}}_{i=1}^n.
\end{align*}
}

We need to show that $V\cap\set{x\in\reals^n:\sgn x_{k_1}=\ep}\subseteq U$.
Since $\eta_{k_1}^{B_1^T}$ is a homeomorphism that fixes $\set{x\in\reals^n:x_{k_1}=0}$ pointwise, $U\cap\set{x\in\reals^n:x_{k_1}=0}$ equals $V\cap\set{x\in\reals^n:x_{k_1}=0}$.
By Lemma~\ref{ps lemma}, any vector in $V\cap\set{x\in\reals^n:\sgn x_{k_1}=\ep}$ equals a vector in $V\cap\set{x\in\reals^n:x_{k_1}=0}$ plus a positive combination of vectors $\bigl(E^{B_1}_{-\ep,k_1}B_1C_t^{B_1;t_1}\bigr)_{\col i}$ whose $k_1$-entry has sign $\ep$.
%Write $v_i$ for the vector $\bigl(E^{B_1}_{-\ep,k_1}B_1C_t^{B_1;t_1}\bigr)_{\col i}$.
Therefore, it suffices to show that every vector $\bigl(E^{B_1}_{-\ep,k_1}B_1C_t^{B_1;t_1}\bigr)_{\col i}$ whose $k_1$-entry has sign~$\ep$ is contained in $\nnspan\sett{\bigl(E^{B_1}_{\ep,k_1}B_1C_t^{B_1;t_1}\bigr)_{\col i}}_{i=1}^n$.

As a temporary shorthand, write $b_{ij}$ for the entries of $B_1$ and write $k$ for $k_1$.
Suppose $v_i=\bigl(E^{B_1}_{-\ep,k}B_1C_t^{B_1;t_1}\bigr)_{\col i}$ for some~$i$ and suppose the $k$-entry of $v_i$ has sign $\ep$.
Write $M$ for $E^{B_1}_{-\ep,k}B_1$ and write $N$ for $E^{B_1}_{\ep,k}B_1$.
Lemma~\ref{columns lem}.\ref{col j} implies that $M_{kj}=-b_{kj}$ for all~$j$.
Lemma~\ref{columns lem}.\ref{cols k} implies that if $\ep M_{kj}\ge0$, then $M_{\col j}=N_{\col j}+|b_{kj}|N_{\col k}$.
Similarly, if $\ep M_{kj}\le0$, then $M_{\col j}=N_{\col j}-|b_{kj}|N_{\col k}$.

%Now $v_i=E^{B_1}_{-\ep,k}B_1\bigl(C_t^{B_1;t_1}\bigr)_{\col i}$, and $\bigl(C_t^{B_1;t_1}\bigr)_{\col i}$ has a sign $\delta\in\set{\pm1}$, meaning that it is not zero and all of its nonzero entries have sign $\delta$.
%(This is ``sign-coherence of $C$-vectors''.  
%See Remark~\ref{conditional}.)
%Thus there are nonnegative numbers $\gamma_j$ such that $v_i=\delta\sum_{j=1}^n\gamma_jM_{\col j}$.
%Write $\set{1,\ldots,n}=S\cup T$ with $S\cap T=\varnothing$ such that $\ep M_{kj}\ge0$ for all $j\in S$ and $\ep M_{kj}\le0$ for all $j\in T$.
%Then
%\begin{align*}
%v_i
%&=\delta\sum_{j\in S}\gamma_jM_{\col j}+\delta\sum_{j\in T}\gamma_jM_{\col j}\\
%&=\delta\sum_{j\in S}\gamma_j(N_{\col j}+|b_{kj}|N_{\col k})+\delta\sum_{j\in T}\gamma_j(N_{\col j}-|b_{kj}|N_{\col k})\\
%&=\delta\sum_{j=1}^n\gamma_jN_{\col j}-\delta\sum_{j=1}^n\ep\gamma_jb_{kj}N_{\col k}\\
%&=N\bigl(C_t^{B_1;t_1}\bigr)_{\col i}+\delta\sum_{j=1}^n\ep\gamma_jM_{kj}N_{\col k}\\
%&=N\bigl(C_t^{B_1;t_1}\bigr)_{\col i}+\sigma N_{\col k},
%\end{align*}
%where $\sigma=\ep\delta\sum_{j=1}^n\gamma_jM_{kj}$ is a positive scalar, because $\delta\sum_{j=1}^n\gamma_jM_{kj}$ is the $k$-entry of $v_i$, which has sign $\ep$.

Now $v_i=E^{B_1}_{-\ep,k}B_1\bigl(C_t^{B_1;t_1}\bigr)_{\col i}$, so writing $\gamma_j$ for the entries of $\bigl(C_t^{B_1;t_1}\bigr)_{\col i}$, we have $v_i=\sum_{j=1}^n\gamma_jM_{\col j}$.
Write $\set{1,\ldots,n}=S\cup T$ with $S\cap T=\varnothing$ such that $\ep M_{kj}\ge0$ for all $j\in S$ and $\ep M_{kj}\le0$ for all $j\in T$.
Then
\begin{align*}
v_i
&=\sum_{j\in S}\gamma_jM_{\col j}+\sum_{j\in T}\gamma_jM_{\col j}\\
&=\sum_{j\in S}\gamma_j(N_{\col j}+|b_{kj}|N_{\col k})+\sum_{j\in T}\gamma_j(N_{\col j}-|b_{kj}|N_{\col k})\\
&=\sum_{j=1}^n\gamma_jN_{\col j}-\sum_{j=1}^n\ep\gamma_jb_{kj}N_{\col k}\\
&=N\bigl(C_t^{B_1;t_1}\bigr)_{\col i}+\sum_{j=1}^n\ep\gamma_jM_{kj}N_{\col k}\\
&=N\bigl(C_t^{B_1;t_1}\bigr)_{\col i}+\sigma N_{\col k},
\end{align*}
where $\sigma=\ep\sum_{j=1}^n\gamma_jM_{kj}$ is a positive scalar, because $\sum_{j=1}^n\gamma_jM_{kj}$ is the $k$-entry of $v_i$, which has sign $\ep$.

As noted above, $\ep$ is the sign of the $k_1$-entry of vectors in $\Cone_t^{-B_1^T;t_1}$.
Since $\Cone^{-B_1^T;t_1}_t=\set{x\in\reals^n:x^TC_t^{B_1;t_1}\ge0}$, the rows of $\bigl(C_t^{B_1;t_1}\bigr)^{-1}$ span the extreme rays of $\Cone_t^{-B_1^T;t_1}$.
In particular $\bigl(C_t^{B_1;t_1}\bigr)^{-1}(\ep e_k)$ has nonnegative entries.
Thus $C_t^{B_1;t_1}\bigl(C_t^{B_1;t_1}\bigr)^{-1}(\ep e_k)=\ep e_k$ is a nonnegative linear combination the columns of~$C_t^{B_1;t_1}$.

Now, the hypothesis that $\kk^{-1}$ is a red sequence for $B_r$, or equivalently a green sequence for $-B_r$, says that $\ep=+1$, so that $e_k$ is a nonnegative linear combination of columns of~$C_t^{B_1;t_1}$.
Thus $N_{\col k}=Ne_k$ is a nonnegative linear combination of columns of~$NC_t^{B_1;t_1}$.
We have shown that $v_i=N\bigl(C_t^{B_1;t_1}\bigr)_{\col i}+\sigma N_{\col k}$ is a nonnegative linear combination of columns of~$NC_t^{B_1;t_1}$.
In other words, $v_i$ is in $\nnspan\sett{\bigl(E^{B_1}_{\ep,k_1}B_1C_t^{B_1;t_1}\bigr)_{\col i}}_{i=1}^n$, as desired.
\end{proof}

\subsection{Dominance regions for extended exchange matrices}\label{ext sec}
We follow \cite{ca4} in considering $m\times n$ extended exchange matrices $\tB$ that are ``tall'', in the sense that $m\ge n$.
We will also consider an $m\times m$ matrix related to $\tB$.
Write $\tB$ in block form $\begin{bsmallmatrix}B\\R\end{bsmallmatrix}$ and again take $\symmetrizer$ to be the diagonal matrix of skew-symmetrizing constants of~$B$.
Let $\symmetrizer'$ be the integer diagonal matrix with minimal entries such that $R':=\symmetrizer'R\symmetrizer^{-1}$ has integer entries.
Define $\BB$ to be the matrix with block form $\begin{bsmallmatrix}B\,&-(R')^T\\R\,&0\end{bsmallmatrix}$.
Most importantly, $\BB$ is skew-symmetrizable by the diagonal matrix $\begin{bsmallmatrix}\symmetrizer&0\\0&\symmetrizer'\end{bsmallmatrix}$ and agrees with $\tB$ in columns $1$ to $n$.
  Throughout, if we have defined an extended exchange matrix $\tB$, without comment we will take $B$ to be the underlying exchange matrix and $\BB$ to be the associated $m\times m$ matrix.
The following is \cite[(3.2)]{ca3}, the version of Lemma~\ref{EBF trick} for extended exchange matrices.
\begin{lemma}\label{EBF trick ext}
For $k\in\set{1,\ldots,n}$ and either choice of $\ep\in\set{\pm1}$, the mutation of~$\tB$ at $k$ is $\mu_k(\tB)=E_{\ep,k}^\BB \tB F_{\ep,k}^B$.
\end{lemma}

The matrix $\BB$ defines mutation maps $\eta_\kk^{\BB^T}$ that act on $\reals^m$ rather than $\reals^n$, but without exception we will only consider mutations in positions $1,\ldots,n$.
Also, given $\BB$, a sequence $\kk=k_r\cdots k_1$ of indices in $\set{1,\ldots,n}$, and seeds $t_1,\ldots,t_r$ such that $t_0\overset{k_1}{\edge}t_1\overset{k_2}{\edge}\,\cdots\,\overset{k_r}{\edge}t_r=t$, there are associated matrices of $\g$-vectors and $\c$-vectors, which we write as $\GG_t^{\BB;t_0}$ and $\CC_t^{\BB;t_0}$.
Since $\kk$ only contains indices in $\set{1,\ldots,n}$, these matrices have block forms
\[
\GG_t^{\BB;t_0}=\begin{bsmallmatrix}G_t^{B;t_0}&0\\H_t^{\tB;t_0}&I_{m-n}\end{bsmallmatrix}
\quad\text{ and }\quad
\CC_t^{\BB;t_0}=\begin{bsmallmatrix}C_t^{B;t_0}&K_t^{\tB;t_0}\\0&I_{m-n}\end{bsmallmatrix},
\]
where $H_t^{\tB;t_0}$ is an $(m-n)\times n$ matrix, $K_t^{\tB;t_0}$ is an $n\times(m-n)$ matrix, and $I_{m-n}$ is the identity matrix.

\begin{proposition}\label{BGCB ext}
$\tB_r\bigl(G_t^{-B_0^T;t_0}\bigr)^T=\bigl(\CC_t^{-\BB_0^T;t_0}\bigr)^T\tB_0$.
\end{proposition}
\begin{proof}
Proposition~\ref{BGCB} says that $\BB_r\bigl(\GG_t^{-\BB_0^T;t_0}\bigr)^T=\bigl(\CC_t^{-\BB_0^T;t_0}\bigr)^T\BB_0$.
Restricted to the first $n$ columns, this says $\BB_r\bigl(\begin{bsmallmatrix}G_t^{-B^T;t_0}&\,0\end{bsmallmatrix}\bigr)^T=\bigl(\CC_t^{-\BB_0^T;t_0}\bigr)^T\tB_0$.
\end{proof}

Given $\tilde\lambda\in\reals^m$, define $\P^\tB_{\tilde\lambda,\kk}=\bigl(\eta_{\kk}^{\BB^T}\bigr)^{-1}\sett{\eta_\kk^{\BB^T}(\tilde\lambda)+\mu_\kk(\tB)\alpha:\alpha\in\reals^n,\alpha\ge0}$.
The \newword{dominance region} $\P^\tB_{\tilde\lambda}$ of $\tilde\lambda$ with respect to $\tB$ is the intersection $\bigcap_\kk\P^\tB_{\tilde\lambda,\kk}$ over all sequences~$\kk$ of indices in $\set{1,\ldots,n}$.

We relate the dominance region $\P_{\tilde\lambda}^\tB$ with respect to the extended exchange matrix $\tB$ to the dominance region $\P_\lambda^B$ with respect to the (non-extended) exchange matrix $B$.
Let $\Proj_n$ be the projection from $\reals^m$ to $\reals^n$ that ignores the last $m-n$ coordinates.
The following lemma is immediate.

\begin{lemma}\label{eta proj}
Let $\tB$ be an extended exchange matrix with underlying exchange matrix $B$.
Then $\Proj_n\circ\eta_\kk^{\BB^T}=\eta_\kk^{B^T}\circ\Proj_n$ as maps on $\reals^n$, for any sequence $\kk$ of indices in $\set{1,\ldots,n}$.
\end{lemma}

\begin{proposition}\label{contains proj}
Let $\tB$ be an extended exchange matrix with underlying exchange matrix $B$.
  If $\tilde\lambda\in\reals^m$ and $\lambda=\Proj_n\tilde\lambda\in\reals^n$, then $\Proj_n(\P_{\tilde\lambda}^\tB)\subseteq\P_{\lambda}^B$.
\end{proposition}
\begin{proof}
The proposition follows from Lemma~\ref{eta proj}, from the analogous fact for $\bigl(\eta_\kk^{\BB^T}\bigr)^{-1}=\eta_{\kk^{-1}}^{\mu_\kk(\BB)^T}$, and from the fact that 
\[\Proj_n\set{\mu_\kk(\tB)\alpha:\alpha\in\reals^n,\alpha\ge0}=\set{\mu_\kk(B)\alpha:\alpha\in\reals^n,\alpha\ge0}.\qedhere\]
\end{proof}

Since $\kk$ consists only of indices in $\set{1,\ldots,n}$, the domains of definition of $\eta_\kk^{\BB^T}$ are determined by the domains of definition of $\eta_\kk^{B^T}$.
Specifically, each domain of definition of $\eta_\kk^{\BB^T}$ is $\Proj_n^{-1}\Sigma$ for some domain of definition $\Sigma$ of $\eta_\kk^{B^T}$.
Accordingly, we define $\Cone_t^{\tB;t_0}$ to be %the nonnegative span of the columns of $\GG_t^{\BB;t_0}$ and the columns of $\begin{bsmallmatrix}0\\-I_{m-n}\end{bsmallmatrix}$.
$\Proj_n^{-1}\Cone_t^{B;t_0}$.
Since $\Cone_t^{B_r;t_r}=\eta_\kk^{B^T}\bigl(\Cone_t^{B;t_0}\bigr)$ for every seed~$t$, also $\Cone_t^{\tB_r;t_r}=\eta_\kk^{\BB^T}\bigl(\Cone_t^{\tB;t_0}\bigr)$ for every seed~$t$.

Similarly, we define the mutation fan for $\tB^T$ to be the set $\F_{\tB^T}$ of cones $\Proj_n^{-1}C$ such that $C$ is a cone in the mutation fan $\F_{B^T}$.
Since $\Proj_n\circ\eta_\kk^{\BB^T}=\eta_\kk^{B^T}\circ\Proj_n$ for sequences $\kk$ of indices in $\set{1,\ldots,n}$, the map $\eta_\kk^{\BB^T}$ is linear on every cone of $\F_{\tB^T}$ and acts as an isomorphism of fans from $\F_{\tB^T}$ to $\F_{\mu_\kk(\tB)^T}$.

To understand dominance regions $\P^\tB_{\tilde\lambda}$, it is enough to consider the case where $\tilde\lambda$ has nonzero entries only in positions $1,\ldots,n$.
Other dominance regions are obtained by translation, as explained in the following lemma.
The lemma is immediate, because the domains of definition of a mutation map $\eta_\kk^{\BB^T}$ depend only on the first $n$ coordinates and each $\eta_\kk^{\BB^T}$ is linear on each domain of definition.
\begin{lemma}\label{after all coefficients are just coefficients}
  If $\tilde\lambda$ and $\tilde\lambda'$ are vectors in $\reals^m$ that agree in the first $n$ coordinates, then $\P^\tB_{\tilde\lambda'}=\P^\tB_{\tilde\lambda}-\tilde\lambda+\tilde\lambda'$.
\end{lemma}

Lemma~\ref{shift} immediately implies the following lemma.
\begin{lemma}\label{shift extended}
Suppose $\tilde\lambda'=\eta_\kk^{\BB^T}(\tilde\lambda)$ and $\tB'=\mu_\kk(\tB)$.
\begin{enumerate}[label=\bf\arabic*., ref=\arabic*]
\item \label{shift extended one}
$\eta^{\BB^T}_\kk\!\!(\P^\tB_{\tilde\lambda,\ll})=\P^{\tB'}_{\tilde\lambda',\ll\kk^{-1}}$ for any $\ll$.
\item \label{shift extended all}
  $\eta^{\BB^T}_\kk\!\!(\P^\tB_{\tilde\lambda})=\P^{\tB'}_{\tilde\lambda'}$.
\end{enumerate}
\end{lemma}

The following lemma is proved similarly to Lemma~\ref{block dom} because the relevant matrix $\BB$ can be rearranged to have a diagonal block decomposition.

\begin{lemma}\label{block dom ext}
If $\tB$ has a block decomposition $\begin{bsmallmatrix}B_1&0\\0&B_2\\R_1&0\\0&R_2\end{bsmallmatrix}$ and $\tilde\lambda\in\reals^m$ has a corresponding decomposition $\begin{bsmallmatrix}\kappa_1\\\kappa_2\\\nu_1\\\nu_2\end{bsmallmatrix}$, define $\tB_i=\begin{bsmallmatrix}B_i\\R_i\end{bsmallmatrix}$ and $\tilde\lambda_i=\begin{bsmallmatrix}\kappa_i\\\nu_i\end{bsmallmatrix}$ for $i=1,2$.
  Then $\P_{\tilde\lambda}^\tB=\P_{\tilde\lambda_1}^{\tB_1}\times\P_{\tilde\lambda_2}^{\tB_2}$.
\end{lemma}

We will prove the following extension of Theorem~\ref{P in B0C}.

\begin{theorem}\label{P in B0C extended}
Suppose $\kk=k_r\cdots k_1$ is a sequence of indices in $\set{1,\ldots, n}$ and $t_0\overset{k_1}{\edge}t_1\overset{k_2}{\edge}\,\cdots\,\overset{k_r}{\edge}t_r=t$.
If $\kk^{-1}=k_1\cdots k_r$ is a red sequence for $B_r$, then for any~$\tilde\lambda$ in the domain of definition of $\eta_\kk^{\BB_0^T}$ that contains $\Cone^{\tB_0;t_0}_t$,
\[\P^{\tB_0}_{\tilde\lambda,\kk}\subseteq\set{\tilde\lambda+\GG_t^{\BB_0;t_0}\tB_r\alpha:\alpha\in\reals^n,\alpha\ge0}=\set{\tilde\lambda+\tB_0C_t^{B_0;t_0}\alpha:\alpha\in\reals^n,\alpha\ge0}.\]
\end{theorem}
\begin{proof}
First, we notice that $\kk^{-1}=k_1\cdots k_r$ is a red sequence for $\BB_r$, or in other words, $\kk$ is a green sequence for $-\BB_r$.
Indeed, since $\CC_{t_{p-1}}^{-\BB;t_0}=\begin{bsmallmatrix}C_{t_{p-1}}^{-B;t_0}&*\\\\0&I_{m-n}\end{bsmallmatrix}$, the sign of column $k_p$ of $\CC_{t_{p-1}}^{-\BB;t_0}$ equals the sign of column $k_p$ of $C_{t_{p-1}}^{-B;t_0}$ whenever $1\le p<r$.
Thus Theorem~\ref{P in B0C} says that
\[\P^{\BB_0}_{\tilde\lambda,\kk}\subseteq\set{\tilde\lambda+\GG_t^{\BB_0;t_0}\BB_r\alpha:\alpha\in\reals^m,\alpha\ge0}=\set{\tilde\lambda+\BB_0\CC_t^{\BB_0;t_0}\alpha:\alpha\in\reals^m,\alpha\ge0}.\]
The assertion of Theorem~\ref{P in B0C extended} is that the same holds even when, in each term, the conditions $\alpha\in\reals^m,\alpha\ge0$ are strengthened by requiring that $\alpha$ is zero in coordinates $n+1,\ldots,m$.

Thus we run through the proof of Theorem~\ref{P in B0C} with $\BB$ replacing $B$ and $m$ replacing $n$ throughout and these additional conditions on $\alpha$ in all relevant expressions.
There is no effect on the argument until the point of showing that $V\cap\set{x\in\reals^m:\sgn x_{k_1}=\ep}\subseteq U$.
Here, we need to show that every vector $v_i=\bigl(E^{\BB_1}_{-\ep,k_1}\BB_1\CC_t^{\BB_1;t_1}\bigr)_{\col i}$ with $i\in\set{1,\ldots,n}$ whose $k_1$-entry has sign~$\ep$ is contained in $\nnspan\sett{\bigl(E^{\BB_1}_{\ep,k_1}\BB_1\CC_t^{\BB_1;t_1}\bigr)_{\col i}}_{i=1}^n$.
We argue as in the proof of Theorem~\ref{P in B0C} that $v_i=N\bigl(\CC_t^{\BB_1;t_1}\bigr)_{\col i}+\sigma N_{\col k}$ and that $\ep e_k$ is a nonnegative linear combination of columns of~$\CC_t^{\BB_1;t_1}$.
Since $\CC_t^{\BB;t_0}=\begin{bsmallmatrix}C_t^{B;t_0}&*\\0&I_{m-n}\end{bsmallmatrix}$, we conclude that $\ep e_k$ is a nonnegative linear combination of columns $1$ through $n$ of~$\CC_t^{\BB_1;t_1}$.
Thus $v_i$ is a nonnegative linear combination of columns $1$ through $n$ of~$N\CC_t^{\BB_1;t_1}$ as desired.
\end{proof}

\subsection{Sufficient conditions for the dominance region to be a point}\label{point sec}  
We now use the tools developed so far to give sufficient conditions for the dominance region to be a point.
We use the conditions to show that the dominance region is always a point when $B$ is of finite type and $\tB$ has linearly independent columns.
In Section~\ref{rescue sec}, we prove the same fact for $B$ of finite type with no conditions on~$\tB$ (Theorem~\ref{finite P point}).
In Section~\ref{aff sec}, we prove the same fact for $B$ of affine type, $\tB$ with linearly independent columns, and $\tilde\lambda$ in the $\g$-vector fan (Theorem~\ref{affine P point indep}). 

With the appropriate change in notations and conventions, the following theorem coincides with \cite[Lemma 3.4.12]{FanQin}.
Here, we give a direct, linear-algebraic proof. 
%Taking care of the appropriate change in notations and conventions, the result we just established provides a completely combinatorial proof of \cite[Lemma 3.4.12]{FanQin}.

\begin{theorem}\label{P point}  
Suppose $\tB_0$ is an extended exchange matrix with linearly independent columns.
Suppose $t$ is a seed in the exchange graph for $\tB_0;t_0$ and take $\tilde\lambda\in\Cone^{\tB_0;t_0}_t$.
If there exists a maximal red sequence for $B_t$, then $\P^{\tB_0}_{\tilde\lambda}=\set{\tilde\lambda}$.
\end{theorem}

\begin{proof}%[Proof of Theorem~\ref{P point}]
Let $t'$ be the seed at the end of the maximal red sequence for $B_t$.
There exists $\ll=\ell_q\ell_{q-1}\cdots\ell_1$ with $t_0=t'_0\overset{\ell_1}{\edge}t'_1\overset{\ell_2}{\edge}\,\cdots\,\overset{\ell_q}{\edge}t'_q=t'$.
Let $\tilde\lambda'=\eta^{\BB_0^T}_\ll\!(\tilde\lambda)$.
  Lemma~\ref{shift extended} says $\eta^{\BB_0^T}_\ll\!(\P^{\tB_0}_{\tilde\lambda})=\P^{\tB'_q}_{\tilde\lambda'}$.
Thus it is enough to prove that $\P^{\tB'_q}_{\tilde\lambda'}=\set{\tilde\lambda'}$.
Since $\eta_\ll^{\BB_0^T}\bigl(\Cone_t^{\tB_0;t_0}\bigr)=\Cone_t^{\tB'_q;t'}$, we may set $t_0$ to be $t'$ instead and assume there is a maximal red sequence for $B_r$ starting from $t$ and ending at $t_0$.
(It is well known that the linear independence of the columns is preserved under mutation.
See \cite[Lemma~1.2]{GSV03} or \cite[Lemma~3.2]{ca3}.)

Let $\kk=k_r\cdots k_1$ be the reverse of the maximal red sequence and define seeds $t_0\overset{k_1}{\edge}t_1\overset{k_2}{\edge}\,\cdots\,\overset{k_r}{\edge}t_r=t$.
Then Theorem~\ref{P in B0C extended} says that $\P^{\tB_0}_{\tilde\lambda,\kk}$ is contained in $\set{\tilde\lambda+\tB_0C_t^{B_0;t_0}\alpha:\alpha\in\reals^n,\alpha\ge0}$.

Since $\kk^{-1}$ is a maximal red sequence for $B_r$, or in other words a maximal green sequence for $-B_r$, every column of $C_{t_0}^{-B_r;t}$ has negative sign, so $\Cone_{t_0}^{B_r^T;t}=\set{x\in\reals^n:x^TC_{t_0}^{-B_r;t}\ge0}$ consists of vectors with nonpositive entries.
Since $\reals_{\le0}^n$ is a cone in the mutation fan $\F_{-B_r}$ (for example, combining \mbox{\cite[Proposition~7.1]{universal}}, \mbox{\cite[Proposition~8.9]{universal}}, and sign-coherence of $C$-vectors) and also $\Cone_{t_0}^{B_r^T;t}$ is a cone in $\F_{-B_r}$, we see that $\Cone_{t_0}^{B_r^T;t}=\reals_{\le0}^n$.
Thus, up to permuting columns, $C_{t_0}^{-B_r;t}$ is the negative of the identity matrix.
We see that $\P^{\tB_0}_{\tilde\lambda,\kk}\subseteq\set{\tilde\lambda-\tB_0\alpha:\alpha\in\reals^n,\alpha\ge0}$.

  Since also $\P^{\tB_0}_{\tilde\lambda,\varnothing}=\set{\tilde\lambda+\tB_0\alpha:\alpha\in\reals^n,\alpha\ge0}$, and since $\tB_0$ is has linearly independent columns, we conclude that $\P^{\tB_0}_{\tilde\lambda}=\set{\tilde\lambda}$.
\end{proof}

When $B$ is of finite type, every $\tilde\lambda\in\reals^m$ is in some $\Cone^{\tB_0;t_0}_t$ for some seed $t$.
It is also well known, and easy to prove, that every exchange matrix $B$ of finite type admits a maximal green sequence and a maximal red sequence.
Thus the following theorem is a consequence of Theorem~\ref{P point}.

\begin{theorem}\label{finite P point indep}  
Suppose $\tB$ is an $m\times n$ extended exchange matrix with linearly independent columns such that $B$ is of finite type.
  Then $\P^\tB_{\tilde\lambda}=\set{\tilde\lambda}$ for any $\tilde\lambda\in\reals^m$.
\end{theorem}

%Since $B$ is of finite type, every $\tilde\lambda\in\reals^m$ is in some $\Cone^{\tB_0;t_0}_t$ for some seed $t$.
%Thus Theorem~\ref{finite P point indep} follows from Theorem~\ref{P point} and the following proposition.
%
%\begin{proposition}\label{fin red}
%If $B$ is an exchange matrix of finite type, then $B$ admits a maximal green sequence and a maximal red sequence.
%\end{proposition}
%\begin{proof}
%The matrix $B$ is of finite type if and only if $-B$ is of finite type, so it is enough to prove that every $B$ of finite type admits a maximal green sequence.
%We construct a maximal green sequence as follows:
%any line parallel to the all-ones vector eventually stays in the positive cone in one direction and eventually stays in the negative cone in the other direction.
%We can choose a line that is disjoint from all codimension-$2$ faces of the $\g$-vector fan.
%The $\g$-vector cones visited by this line, from the positive cone to the negative cone, define a maximal green sequence.
%\end{proof}

\section{Dominance regions in finite type}\label{rescue sec}
In this section, we strengthen Theorem~\ref{finite P point indep} by removing the hypothesis that~$\tB$ has linearly independent columns.

\begin{theorem}\label{finite P point}  
Suppose $\tB$ is an $m\times n$ extended exchange matrix such that $B$ is of finite type.
  Then $\P^\tB_{\tilde\lambda}=\set{\tilde\lambda}$ for any $\tilde\lambda\in\reals^m$.
\end{theorem}

We had expected the proof of Theorem~\ref{finite P point} to be easy and uniform, but we did not find an easy or uniform proof.  
Instead we rely on the classification of finite crystallographic root systems and the corresponding classification of cluster algebras of finite type~\cite{ca2} and use the surfaces model, folding, and computer checks.

\subsection{Finite type, coefficient free}
We begin by proving the coefficient-free version of Theorem~\ref{finite P point}.

\begin{theorem}\label{finite P point no}  
Suppose $B$ is an $n\times n$ exchange matrix of finite type.
Then ${\P^B_\lambda=\set{\lambda}}$ for all $\lambda\in\reals^n$.
\end{theorem}

We say that an exchange matrix $B=[b_{ij}]$ is \newword{bipartite} if there is a bipartition $\set{1,\ldots,n}=P\cup N$ such that $b_{ij}>0$ implies $i\in P$ and $j\in N$.
A small part of the classification of cluster algebras of finite type~\cite{ca2} says that every exchange matrix of finite type is mutation equivalent to a bipartite exchange matrix of finite type.
Thus in light of Lemma~\ref{shift}, Theorem~\ref{finite P point no} is an immediate consequence of the following more specific proposition.

\begin{proposition}\label{finite P point bip}
  Suppose $B$ is a bipartite $n\times n$ exchange matrix of finite type with bipartition $\set{1,\ldots,n}=P\cup N$ as above.
  Let $\kk_P$ and $\kk_N$ be any shuffle of $P$ and $N$ respectively and consider the concatenation $\kk=\kk_P \kk_N$.
  Then $\bigcap_{\ell\in\integers}\P^B_{\lambda,\kk^\ell}=\set{\lambda}$ for all $\lambda\in\reals^n$.
\end{proposition}

The intersection in Proposition~\ref{finite P point bip} is written as an infinite intersection but in fact is the intersection of finitely many different sets.

\begin{proof}%[Proof of Proposition~\ref{finite P point bip}]
In light of Lemma~\ref{block dom}, we may as well assume that $B$ is block-indecomposable.
Thus the Cartan matrix associated to $B$ is indecomposable of finite type.
We verify the proposition type-by-type using the marked surfaces model, folding, and computations.
We assume the most basic facts about the marked surfaces model \cite{cats1,cats2}, including the signed adjacency matrix of a triangulation, mutation by flipping arcs, the definition of shear coordinates, and the fact that shear coordinates define a bijection between laminations and vectors.
(In general, one must be careful of the difference between rational laminations and real laminations, but we will only consider finite type, where one may as well allow laminations to have real weights.)

\medskip

\noindent
\textbf{Type A.}
In the following, we assume $n\ge2$ to avoid separate explanations for type~$A_1$, which can instead be handled by direct inspection.
%To avoid corner cases we deal with type $A_1$ by direct inspection as for the exceptional types below and in the following assume $n\ge2$.
Let $T$ be a triangulation of the regular $(n+3)$-gon whose signed adjacency matrix is the bipartite exchange matrix~$B$.
Then up to switching the orientation of the plane, $T$ is as exemplified in the left picture of Figure~\ref{T fig}.
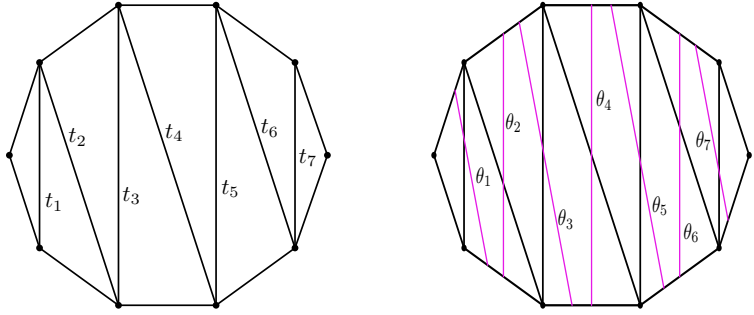
\begin{figure}
\scalebox{0.85}{\includegraphics{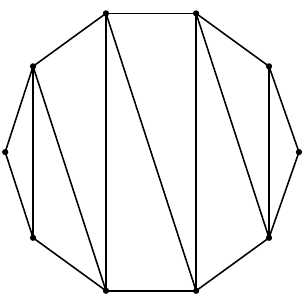}
\begin{picture}(0,0)(73,-73)
\put(-58.5,-23){$t_1$}
\put(-48,6){$t_2$}
\put(-24,-20){$t_3$}
\put(-5,8){$t_4$}
\put(20,-17){$t_5$}
\put(37.5,10){$t_6$}
\put(54.5,-3){$t_7$}
\end{picture}}
\qquad\quad
\scalebox{0.66}[0.85]{\includegraphics{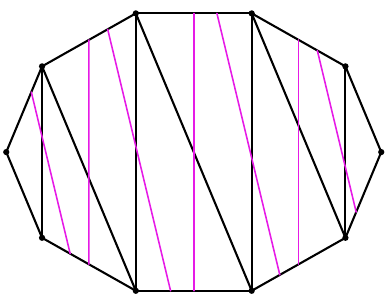}
\begin{picture}(0,0)(93,-73)
\put(-70,-12){$\theta_1$}
\put(-53,12){$\theta_2$}
\put(-23,-31){$\theta_3$}
\put(-1,22){$\theta_4$}
\put(31,-25){$\theta_5$}
\put(49,-37){$\theta_6$}
\put(56,4){$\theta_7$}
\end{picture}}
\caption{A bipartite triangulation $T$ and the lamination $\Theta^{(\ell)}$}
\label{T fig}
\end{figure}
We refer to the arcs of~$T$ as $t_1,\ldots,t_n$. 
One easily verifies that flipping the arcs of $T$ in the order specified by $\kk$ corresponds to an $(n+3)$-fold rotation of~$T$.
The direction of the rotation depends on the exact choice of~$B$ and will not matter in the argument.
For definiteness in what follows, we will call that direction ``counterclockwise'', and for convenience, we will refer to this $(n+3)$-fold rotation as a rotation of ``one unit'' counterclockwise.

A lamination consists of a set of \newword{compatible} (i.e.\ noncrossing) curves each connecting an edge of the $(n+3)$-gon to a different and non-adjacent edge and each with an assigned real weight.
Any collection of weighted curves joining the edges of the polygon has an associated vector of shear coordinates and there is a unique lamination having the same shear coordinates.
Maximal sets of compatible curves have cardinality~$n$.
If a lamination has fewer than $n$ curves, we always extend it to a set of $n$ compatible curves by adjoining additional curves with weight~$0$.
The way in which this is done will not affect the arguments below.
Thus the combinatorics of laminations is the combinatorics of triangulations of the $(n+3)$-gon, except that each curve connects two non-adjacent \emph{sides} of the polygon rather than two non-adjacent \emph{vertices}.

To compute dominance regions, we need to find, for every vector that is a nonnegative linear combination of the columns of $B$, the lamination whose shear coordinates, with respect to~$T$, are that vector.
Those laminations all have the same underlying set $\set{\theta_1,\ldots,\theta_n}$ of curves, which we now describe.
Since $B$ is bipartite, each arc $t_i$ except for $t_1$ and $t_n$ lies in a quadrilateral containing exactly two boundary edges and the curve $\theta_i$ connects these two edges as exemplified in the right picture of Figure~\ref{T fig}.
The quadrilaterals enclosing $t_1$ and $t_n$ have three boundary edges and the curves $\theta_1$ and $\theta_n$ connect the two that are opposite to each other. 
This collection has several remarkable properties that are crucial to the proof.  
In fact, most of the argument will seem clear without explicit appeals to these properties, but we highlight these properties in order to make the later type-D argument more clear.

\begin{itemize}
\item
\newword{First remarkable property}:
Every curve that can appear in a lamination can be rotated to become one of the curves in $\set{\theta_1,\ldots,\theta_n}$.
%\item
%\newword{First remarkable property'}:
%    For any lamination $\Lambda$ and any curve $\theta_i$, we can find a rotation of $\Lambda$ that contains $\theta_i$ or we can find a  rotation of $\Lambda$ so that $\theta_i$ is incompatible with exactly one curve and this curve can be rotated to be some $\theta_{i'}$.
\item
\newword{Second remarkable property}:
The shear coordinates of $\theta_i$ with respect to~$T$ is $\ep_i=\pm1$ in position $t_i$ and $0$ elsewhere, where the sign $\ep_i$ is determined by the orientation of $t_i$ in its quadrilateral.
\item
\newword{Third remarkable property}:
Any curve $\gamma\neq\theta_i$ that can appear in a lamination has the property that the shear coordinates of $\gamma$ and $\theta_i$ weakly disagree in sign in position $t_i$.
They strictly disagree in sign if and only if~$\gamma$ and~$\theta_i$ are not compatible.
\end{itemize}
What is remarkable about these properties is not that we were able to find a set of curves satisfying them, but rather that the laminations that we are forced to consider consist of a set of curves with these properties.

In order to compute $\P^B_\lambda$, we consider the weighted lamination $\Lambda$ such that $\lambda$ is the negation of the shear coordinates of $\Lambda$ with respect to the triangulation~$T$.
We can compute $\eta^{B^T}_{\kk^\ell}(\lambda)$ by reading the negative shear coordinates of $\Lambda$ with respect to a rotation of $T$, specifically the rotation $\ell$ units counterclockwise that corresponds to flipping the arcs of $T$ in the order specified by $\kk^\ell$.
(See \cite[Theorem~12.6]{cats2}, \cite[Theorem~4.3]{unisurface}, and \cite[Section~5]{unisurface}.)
If $\Lambda$ is supported on fewer than $n$ curves, we extend it by adjoining compatible curves with weight~$0$.
(There may be more than one way to do this, but we fix some choice.)
We label the curves of the lamination~$\Lambda$ in some order as $L_1,L_2,\ldots,L_n$ with weights $w_1,\ldots,w_n$ respectively.

For each $\ell\in\integers$, a point $x\in \P^B_{\lambda,\kk^\ell}$ is of the form $\bigl(\eta_{\kk^\ell}^{B^T}\bigr)^{-1}\bigl(\eta_{\kk^\ell}^{B^T}(\lambda)+\mu_{\kk^\ell}(B)\alpha^{(\ell)}\bigr)$ for some $\alpha^{(\ell)}$ with nonnegative entries $a_1^{(\ell)},\ldots,a_n^{(\ell)}$.
For convenience, we also define $a_0^{(\ell)}=a_{n+1}^{(\ell)}=0$.
In the $(n+3)$-gon model, such a point $x$ is obtained as follows.
Rotate $T$ by $\ell$ units counterclockwise to obtain a triangulation~$T'$.
Reading the shear coordinates of $\Lambda$ with respect to the rotated triangulation $T'$ is the same as applying the mutation map $\eta^{B^T}_{\kk^\ell}$ to $\lambda$.
Thus we find the lamination whose shear coordinates with respect to $T'$ are the sum of the shear coordinates of $\Lambda$ with respect to~$T'$ and $-\mu_{\kk^\ell}(B)\alpha^{(\ell)}=-B\alpha^{(\ell)}$.
Then $x$ is the negative of the shear coordinates of that lamination with respect to the unrotated triangulation~$T$.

The description above describes ${x=\bigl(\eta_{\kk^\ell}^{B^T}\bigr)^{-1}\bigl(\eta_{\kk^\ell}^{B^T}(\lambda)+\mu_{\kk^\ell}(B)\alpha^{(\ell)}}\bigr)$ directly.
However, it will be convenient to take advantage of the fact that rotations of the $(n+3)$-gon commute with the other operations that produce a point $x$, and thus construct such an $x\in \P^B_{\lambda,\kk^\ell}$ without ever rotating the triangulation.
This will allow us to make use of the remarkable properties of the collection $\set{\theta_1,\ldots,\theta_n}$.

So instead we construct a point $x\in \P^B_{\lambda,\kk^\ell}$ as follows.
First rotate $\Lambda$ clockwise $\ell$ units to obtain a lamination $\Lambda^{(\ell)}$
with curves $(L_1)^{(\ell)},(L_2)^{(\ell)},\ldots,(L_n)^{(\ell)}$ weighted $w_1,\ldots,w_n$ respectively.
Then find the unique lamination $\Theta^{(\ell)}$ whose shear coordinates with respect to~$T$ are~$-B\alpha^{(\ell)}$.
Crucially, $\Theta^{(\ell)}$ consists of the curves $\set{\theta_1,\ldots,\theta_n}$ with $\theta_i$ having weight $z_i^{(\ell)}=a_{i-1}^{(\ell)}+a_{i+1}^{(\ell)}$. 
%\begin{figure}
%\scalebox{0.83}[1]{\includegraphics{Theta.pdf}
%\begin{picture}(0,0)(93,-73)
%\put(-70,-12){$\theta_1$}
%\put(-53,12){$\theta_2$}
%\put(-23,-31){$\theta_3$}
%\put(-1,22){$\theta_4$}
%\put(31,-25){$\theta_5$}
%\put(49,-37){$\theta_6$}
%\put(56,4){$\theta_7$}
%\end{picture}}
%%\scalebox{1}{\includegraphics{Theta.pdf}}
%%\begin{picture}(0,0)(113,-83)
%%\put(-88,-8){$b_1^{(\ell)}$}
%%\put(-63,8){$b_2^{(\ell)}$}
%%\put(-28,-27){$b_3^{(\ell)}$}
%%\put(-1,18){$b_4^{(\ell)}$}
%%\put(35,-17){$b_5^{(\ell)}$}
%%\put(60,-37){$b_6^{(\ell)}$}
%%\put(67.5,0){$b_7^{(\ell)}$}
%%\end{picture}
%\caption{The lamination $\Theta^{(\ell)}$}
%\label{Theta fig}
%\end{figure}
Let $X^{(\ell)}$ be the lamination whose shear coordinates with respect to~$T$ are the same as the shear coordinates of $\Lambda^{(\ell)}\cup\Theta^{(\ell)}$.
(Both $\Theta^{(\ell)}$ and $X^{(\ell)}$ depend on $a_1^{(\ell)},\ldots,a_n^{(\ell)}$, but this dependence is not explicitly in the notation.)
Rotate $X^{(\ell)}$ counterclockwise $\ell$ units and read shear coordinates with respect to~$T$.
The negative of these shear coordinates is~$x$.

Now suppose $x\in\bigcap_{\ell\in\integers}\P^B_{\lambda,\kk^\ell}$.
By the construction above, for each $\ell$, the vector~$x$ is the shear coordinates with respect to $T$ of a rotation of a lamination $X^{(\ell)}$ as above.
Since there is a unique lamination associated to each vector, we obtain the following fact, which we refer to below as the \newword{main observation}:  all of the laminations $X^{(\ell)}$ for $\ell\in\integers$ are related by rotations.
Specifically, $X^{(\ell)}$ is obtained by rotating $X^{(0)}$ clockwise $\ell$ units.

We will show that $\alpha^{(\ell)}=0$ for all $\ell$, or in other words that $a_1^{(\ell)}=\cdots=a_n^{(\ell)}=0$.
%This implies that $x=\lambda$.
Since $a_i^{(\ell)}\ge 0$ for all $i$ and $\ell$, it is enough to show that each $z_i^{(\ell)}$ is~$0$.

By the first remarkable property, for each $L_p$ there exists an $\ell_p\in\integers$ so that $(L_p)^{(\ell_p)}=\theta_{i_p}$ for some $i_p$.
Then the sum of the shear coordinates of $\Lambda^{(\ell_p)}$ and~$\Theta^{(\ell_p)}$ is $\ep_{i_p}(w_p+z_{i_p}^{(\ell_p)})$ in position~$t_{i_p}$ because the second remarkable property assures that $(L_p)^{(\ell_p)}$ makes this contribution and the third remarkable property says that no other curves of $\Lambda^{(\ell_p)}$ or $\Theta^{(\ell_p)}$ contribute to this coordinate.
(See the left picture of Figure~\ref{theti fig}.)
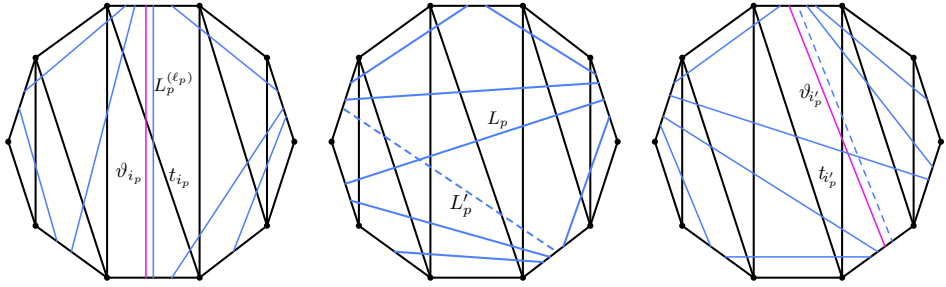
\begin{figure}
\scalebox{0.77}{\includegraphics{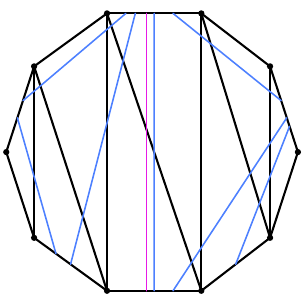}
\begin{picture}(0,0)(73,-73)
\put(-21,-17){$\thet_{i_p}$}
\put(-2.5,25){\small$L_p^{(\ell_p)}$}
%\put(25,-13){$L'_j$}
\put(5,-19){$t_{i_p}$}
\end{picture}}
\hspace{0.5pt}
\scalebox{0.77}{\includegraphics{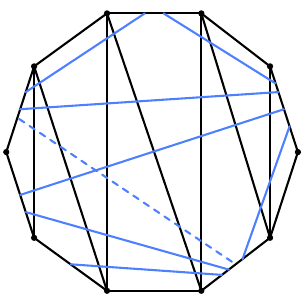}
\begin{picture}(0,0)(73,-73)
\put(2.5,9){\small$L_p$}
\put(-16,-33){$L'_p$}
%\put(5,-19){$t_{i_p}$}
\end{picture}}
\hspace{0.5pt}
\scalebox{0.77}{\includegraphics{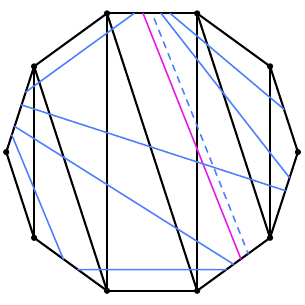}
\begin{picture}(0,0)(73,-73)
%\put(-25,25){\small$(L_p)^{(\ell'_p)}$}
\put(-3,22){$\thet_{\!i'_p}$}
\put(6.5,-18){$t_{\!i'_p}$}
\end{picture}}
%\quad
%\scalebox{0.77}{\includegraphics{theti3.pdf}}
%\begin{picture}(0,0)(73,-73)
%\put(-3,-33){$\thet_i$}
%\put(15,-30){$t_i$}
%\end{picture}
\caption{An illustration of the proof of the Type-A case}
\label{theti fig}
\end{figure}
Therefore, by the third remarkable property again, $X^{(\ell_p)}$ has the curve $\theta_{i_p}$ with weight $w_p+z_{i_p}^{(\ell_p)}$.
(Or, if ${w_p+z_{i_p}^{(\ell_p)}=0}$, then $X^{(\ell_p)}$ consists of curves that are compatible with $\theta_{i_p}$.)
By the main observation, this implies $X^{(0)}$ has the curve $L_p$ with weight $w_p+z_{i_p}^{(\ell_p)}$.
It follows that $X^{(0)}$ contains the same curves as $\Lambda$ 
(after possibly adjoining some curves with weight $0$ that are compatible with the curves in $X^{(0)}$).

Let $L'_p$ be the unique curve that is compatible with every curve in $\Lambda$ except~$L_p$.
(See the middle picture of Figure~\ref{theti fig}.)
By the first remarkable property, there exists $\ell'_p\in\integers$ so that $(L'_p)^{(\ell'_p)}=\theta_{i'_p}$ for some $i'_p$.
Since every curve in $\Lambda^{(\ell'_p)}$ or $\Theta^{(\ell'_p)}$, besides $(L_p)^{(\ell'_p)}$, is compatible with $\theta_{i'_p}$, the third remarkable property says that the shear coordinate of $X^{(\ell'_p)}$ in position~$t_{i'_p}$ is ${\ep_{i'_p}(-w_p+z_{i'_p}^{(\ell'_p)})}$.
(See the right picture of Figure~\ref{theti fig}.)
But the main observation says that $X^{(\ell'_p)}$ contains $(L_p)^{(\ell'_p)}$ with weight $w_p+z_{i_p}^{(\ell_p)}$ and otherwise contains curves that are compatible with the rotation of $L_p$.
Thus, by the third remarkable property, the shear coordinate of $X^{(\ell'_p)}$ in position~$t_{i'_p}$ is $-\ep_{i'_p}(w_p+z_{i_p}^{(\ell_p)})$.
Therefore $w_p-z_{i'_p}^{(\ell'_p)}=w_p+z_{i_p}^{(\ell_p)}$, and we conclude that $z_{i_p}^{(\ell_p)}=z_{i'_p}^{(\ell'_p)}=0$ since both are nonnegative.

It follows that $X^{(0)}=\Lambda$ and thus $x=\lambda$.

\medskip

\noindent
\textbf{Type D.}
The outline of the proof in type D is essentially the same as in type~A.
The major difference in type D is the appearance of tagged triangulations.
Let~$T$ be a tagged triangulation of the once-punctured regular $n$-gon whose signed adjacency matrix is the bipartite exchange matrix~$B$, exemplified in the left picture of Figure~\ref{d T fig}.
\begin{figure}
\scalebox{0.85}{\includegraphics{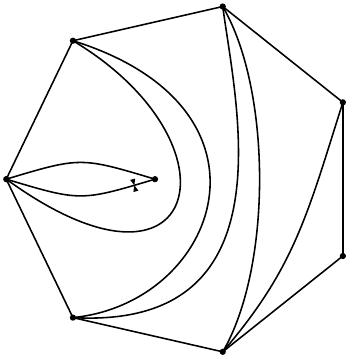}
\begin{picture}(0,0)(82,-86)
\put(-62,12.5){$t_1$}
\put(-36.5,-14){$t_2$}
\put(-25,-33.5){$t_3$}
\put(-6.5,39.5){$t_4$}
\put(6.5,-53.5){$t_5$}
\put(35.5,-5){$t_6$}
\put(55.5,-26){$t_7$}
\end{picture}}
\qquad
\scalebox{0.85}{\includegraphics{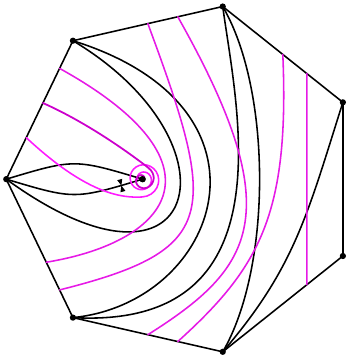}
\begin{picture}(0,0)(82,-86)
\put(-40,-15.5){$\theta_1$}
\put(-64,22){$\theta_2$}
\put(-68,-35){$\theta_3$}
\put(-27.5,61){$\theta_4$}
\put(0,68){$\theta_5$}
\put(34,45){$\theta_6$}
\put(57.5,-35){$\theta_7$}
\end{picture}}
\caption{A bipartite tagged triangulation $T$ and lamination~$\Theta^{(\ell)}$}
\label{d T fig}
\end{figure}
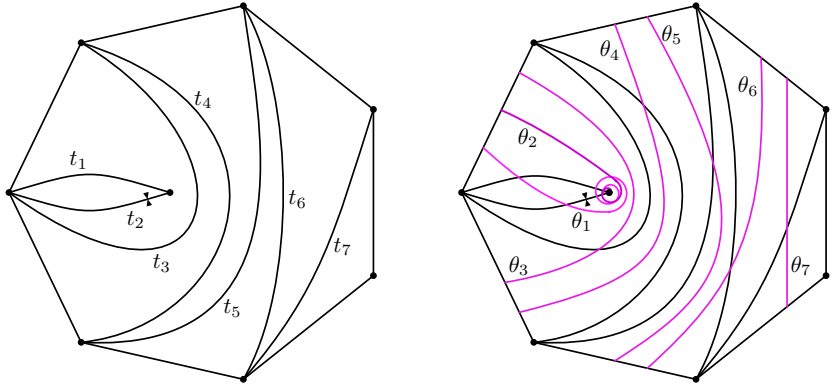
The arcs of $T$ are $t_1,\ldots,t_n$.
The mutation $\mu_\kk$ on $B$ corresponds to an $n$-fold rotation (rotation by ``one unit'') of~$T$ in a direction that we will call ``counterclockwise'', together with swapping which arc is notched at the puncture.

Let $\Lambda$ be the weighted quasi-lamination such that $\lambda$ is the negation of the shear coordinates of $\Lambda$.
Quasi-laminations are a variant of laminations with the property that the curves allowed in quasi-laminations are in bijection with the tagged arcs.
(See \cite[Section~4]{unisurface}.)
A quasi-lamination $\Lambda$ in the once-punctured $n$-gon is a pairwise compatible set of curves, each with a real weight.
Some of the curves connect an edge of the $n$-gon to a \emph{different} edge.
The two edges must be non-adjacent, unless the curve and the two edges enclose the puncture.
The remaining curves have an endpoint on an edge and, at the other end, spiral into the puncture in either the clockwise or counterclockwise direction.
In the once-punctured $n$-gon, there are two differences between laminations and quasi-laminations.
First, a curve with both endpoints on the same edge of the $n$-gon, going around the puncture, is \emph{not} allowed in a quasi-lamination.
Second, while the notion of compatibility is ``non-intersecting'' for most pairs of curves, there is one exception:  
for each side, the two curves starting on that side and spiraling around the puncture in opposite directions \emph{are} compatible, even though they intersect.

The key properties of laminations also hold for quasi-laminations, specifically the bijection (shear coordinates) between quasi-laminations and vectors \cite[Theorem~4.4]{unisurface} and the fact that applying $\eta^{B^T}$ to vectors corresponds to flipping a sequence of tagged arcs in the initial triangulation \cite[Theorem~4.3]{unisurface}.
In particular, the negative of the shear coordinate vector of $\Lambda$ with respect to a rotation of $T$ by $\ell$ units, with taggings swapped when $\ell$ is odd, is $\eta^{B^T}_{\kk^\ell}(\lambda)$.

Maximal sets of pairwise compatible curves, in the sense of quasi-laminations, have cardinality $n$, and behave combinatorially like tagged triangulations of the punctured $n$-gon.
If $\Lambda$ has fewer than $n$ curves, we extend it to a maximal set of compatible curves by adjoining curves with weight~$0$.
The curves of $\Lambda$ are $L_1,L_2,\ldots,L_n$, with weights $w_1,\ldots,w_n$.

As in type A, we can describe a point $x\in\P^B_{\lambda,\kk^\ell}$ directly using the definition,  % (rotating $T$, constructing the lamination whose shear coordinates are the shear coordinates of $\Lambda$ plus $-\mu_{\kk^\ell}(B)\alpha^{(\ell)}$, then rotating $T$ back).
but here we skip directly to the more convenient description:
to construct a point $x\in\P^B_{\lambda,\kk^\ell}$, we first rotate $\Lambda$ clockwise $\ell$ units, swapping all spiral directions if $\ell$ is odd to obtain a lamination $\Lambda^{(\ell)}$.

We define $\Theta^{(\ell)}$ to be the lamination whose shear coordinates are $-\mu_{\kk^\ell}(B)\alpha^{(\ell)}=-B\alpha^{(\ell)}$.
The lamination $\Theta^{(\ell)}$ has $n$ curves, described as follows and exemplified in the right picture of  Figure~\ref{d T fig}.
The curve $\theta_1$ has an endpoint on the edge enclosed by $t_3$ and spirals counterclockwise into the puncture.
The curve $\theta_2$ also has an endpoint on that edge but spirals clockwise into the puncture.
The weights on these two curves are $z_1^{(\ell)}=z_2^{(\ell)}=a_3$.
Each remaining curve intersects exactly one arc of $T$.
The curve $\theta_3$ intersects $t_3$ and has weight $z_3^{(\ell)}=a_1+a_2+a_4$.
For $i=4,\ldots,n$, the curve $\theta_i$ has weight $z_i^{(\ell)}=a_{i-1}^{(\ell)}+a_{i+1}^{(\ell)}$ and intersects~$t_i$ (taking $a_{n+1}^{(\ell)}=0$ as before).
%\begin{figure}
%\scalebox{1}{\includegraphics{Thetad.pdf}
%\begin{picture}(0,0)(82,-86)
%\put(-40,-15.5){$\theta_1$}
%\put(-64,22){$\theta_2$}
%\put(-68,-35){$\theta_3$}
%\put(-27.5,61){$\theta_4$}
%\put(0,68){$\theta_5$}
%\put(34,45){$\theta_6$}
%\put(57.5,-35){$\theta_7$}
%\end{picture}}
%\caption{The lamination $\Theta^{(\ell)}$ in type D}
%\label{d Theta fig}
%\end{figure}
As in type A, the lamination $X^{(\ell)}$ has shear coordinates with respect to~$T$ equal to the sum of the shear coordinates of $\Lambda^{(\ell)}$ and the shear coordinates of $\Theta^{(\ell)}$.
The vector $x$ is obtained by rotating $X^{(\ell)}$ counterclockwise $\ell$ units, swapping all spiral directions if $\ell$ is odd, reading shear coordinates with respect to~$T$, and negating.

Suppose $x=\lambda+B\alpha^{(0)}$ is in $\bigcap_{\ell\in\integers}\P^B_{\lambda,\kk^\ell}$, so that $x$ determines a lamination $X^{(\ell)}$ for all $\ell\in\integers$.
As in type A, we have the \newword{main observation}:
for any $\ell\in\integers$, the lamination $X^{(\ell)}$ is obtained by rotating~$X^{(0)}$ clockwise by $\ell$ units, swapping spiral directions if $\ell$ is odd.

The crux of the proof is that the curves $\set{\theta_1,\ldots,\theta_n}$ have the three \newword{remarkable properties} described in the type-A case, substituting ``rotations that appropriately swap spirals'' for ``rotations'' and ``quasi-laminations'' for ``laminations''.
When the three properties are verified, the proof is then completed precisely as in the type-A case, with the same substitutions.
The verification is straightforward, using the definitions of compatibility and shear coordinates, and we omit the details.

\medskip

\noindent
\textbf{Types C and B.} 
The result in these types follows from the results in types A and~D by folding.
(See Section~\ref{fold sec}.)
One readily verifies that (as is well known) a bipartite exchange matrix $B'$ of type~$C_n$ is obtained from a bipartite exchange matrix $B$ of type $A_{2n-1}$ by folding.
The stable automorphism $\sigma$ has orbits $\set{n}$ and $\set{i,2n-i}$ for $i=1,\ldots,n-1$.
Since $i$ and $2n-i$ are in the same orbit for all $i=1,\ldots,n-1$, we can ask for $\kk$ to be an orbit sequence, and therefore every sequence appearing in Proposition~\ref{finite P point bip} is an orbit sequence.
Thus the type-C case of the theorem follows from the type-A case and Theorem~\ref{fold dom reg}.

Similarly, a bipartite exchange matrix of type $B_n$ is obtained from a bipartite exchange matrix of type $D_{n+1}$ by folding with respect to a stable automorphism whose only nontrivial orbit consists of the two ``antennae'' of the Dynkin diagram.
The two antennae are in the same part of the bipartition, so the type-B case of the theorem similarly follows from the type-D case.

\medskip

\noindent
\textbf{The exceptional types.}
The remaining finite types are checked computationally.
A priori, the computation is infinite for each exchange matrix because the proposition must hold for each $\lambda\in\reals^n$, but we reduce it to a finite computation that must be done for one $\lambda$ in the interior of each maximal cone of the mutation fan.

The set $\P^B_{\lambda,\kk}$ can be computed, for each $\kk$ as a union of polyhedra, as follows.
One computes $\sett{\eta_\kk^{B^T}(\lambda)+\mu_\kk(B)\alpha:\alpha\ge0}$, a cone at the point $\eta_\kk^{B^T}(\lambda)$, and then applies $\bigl(\eta_{\kk}^{B^T}\bigr)^{-1}$ as a sequence of piecewise linear maps with two domains of linearity.
At the first step, if the relative interior of the cone $\sett{\eta_\kk^{B^T}(\lambda)+\mu_\kk(B)\alpha:\alpha\ge0}$ intersects both domains of linearity, the piecewise-linear map breaks the cone into a union of two polyhedra.
Our computation considers, instead of this union, the smallest cone at the image of $\lambda$ containing the union.
Then at the next step, instead of passing to a union of two polyhedra, we again take the smallest cone containing the union.  
Continuing in this manner, we obtain a cone $\overline{\P}^B_{\lambda,\kk}$ at $\lambda$.

If $\lambda$ is in the interior of a maximal cone $C$ in the mutation fan, and $\lambda'$ is another vector in $C$, but not necessarily in the interior, we see that $\overline{\P}^B_{\lambda,\kk}$ is at least as big as $(\lambda-\lambda')+\overline{\P}^B_{\lambda',\kk}$.
(At every step, the cone $C$ is contained in one of the two domains of linearity.
But at some step, $\lambda'$ might be on the boundary of the two domains of linearity, and if so, one of the two polyhedra might be smaller in the construction of $\overline{\P}^B_{\lambda',\kk}$ than in the translation of the construction of $\overline{\P}^B_{\lambda,\kk}$.)
Therefore, we need only verify for one $\lambda$ in the interior of each maximal cone of the mutation fan, that $\bigcap_{\ell\in\integers}\overline{\P}^B_{\lambda,\kk^\ell}=\set{\lambda}$.
Our computation is assisted by the technology of $c$-sortable elements, which index the maximal cones.
\end{proof}

\subsection{Finite type, arbitrary coefficients}
Having proved Theorem~\ref{finite P point no}, we now prepare to complete the proof of Theorem~\ref{finite P point}.
Call an extended exchange matrix~$\tB$ \newword{salient} if the nonnegative linear span of the columns of $\tB$ contains no line.
(This is a reference to convex geometry, where a cone is called salient if and only if there is no nonzero vector $x$ such that $x$ and $-x$ are contained in the cone.)
Requiring that $\tB$ is salient is strictly weaker than requiring that the columns of $\tB$ are linearly independent.
One can show that $\tB$ is salient by exhibiting a vector $x\in\reals^m$ whose dot product is strictly positive with every nonzero column of $\tB$.
%For us, the crucial fact about a salient extended exchange matrix $\tB$ is that the intersection of the nonnegative linear span of the columns of $\tB$ with the nonpositive linear span of the columns of $\tB$ is $\set{0}$.
Furthermore, if $\tB$ is salient and $\alpha\in\reals^n$ has nonnegative entries and $\tB\alpha=0$, then $\alpha=0$.

\begin{lemma}\label{bip sal}
If $B$ is a bipartite exchange matrix, then $B$ is salient.
\end{lemma}
\begin{proof}
Suppose $\set{1,\ldots,n}=P\cup N$ is a bipartition such that $b_{ij}>0$ implies $i\in P$ and $j\in N$.
Choose a vector $x\in\reals^n$ that is positive in positions $P$ and negative in positions~$N$.
Then $x$ has strictly positive dot product with every nonzero column of $B$, so $B$ is salient.
\end{proof}

\begin{proof}[Proof of Theorem~\ref{finite P point}]
By Lemma~\ref{shift extended}, it is enough to prove the theorem for any exchange matrix mutation equivalent to~$\tB$.
Thus, since every exchange matrix of finite type is mutation equivalent to a bipartite exchange matrix of finite type, we may as well assume that $B$ is bipartite, and thus salient by Lemma~\ref{bip sal}.

Suppose $\tB$ is $m\times n$ and $B$ is bipartite of finite type and suppose $\tilde\lambda\in\reals^m$ and $\lambda=\Proj_n(\tilde\lambda)$.
Combining Theorem~\ref{finite P point no} with Proposition~\ref{contains proj}, we see that $\P_{\tilde\lambda}^\tB\subseteq\Proj_n^{-1}(\set{\lambda})$.
But also $\P_{\tilde\lambda}^\tB\subseteq\P_{\tilde\lambda,\varnothing}^\tB=\set{\tilde\lambda+\tB\alpha:\alpha\ge0}$, and therefore ${\P_{\tilde\lambda}^\tB\subseteq\set{\tilde\lambda+\tB\alpha:\alpha\ge0,B\alpha=0}}$.
Since $B$ is salient, if $\alpha\ge0$ and $B\alpha=0$, then $\alpha=0$.
Thus $\P_{\tilde\lambda}^\tB=\set{\tilde\lambda}$.
\end{proof}

\section{Dominance regions in affine type}\label{aff sec}
An exchange matrix is of \newword{affine type} if it is mutation-equivalent to an acyclic exchange matrix whose underlying Cartan matrix is of affine type.
We will explain and prove the following theorem, as well as a more general theorem (Theorem~\ref{affine main extended}) describing $\P^\tB_{\tilde\lambda}$ for arbitrary extensions $\tB$ of~$B$.

\begin{theorem}\label{affine main} 
Suppose $B$ is an exchange matrix of affine type.
If $\lambda$ is in the relative interior of the imaginary wall~$\d^B_\infty$, then the dominance region $\P^B_\lambda$ is the line segment ${\set{\lambda+aB\delta^B:a\ge0}\cap\d^B_\infty}$ parallel to the imaginary ray, with one endpoint at $\lambda$ and the other endpoint on the relative boundary of $\d^B_\infty$.
\end{theorem}

Until the theorem is explained fully, we remark that the imaginary wall $\d^B_\infty$ is a codimension-1 closed convex cone, the closure of the complement of the union of all full-dimensional cones of the mutation fan.
The imaginary wall is a union of cones of the mutation fan.
There is a unique ray of the mutation fan in the relative interior of the imaginary wall, and that ray has direction $-B\delta^B$ for some vector~$\delta^B$.
When $B$ is acyclic, $\delta^B$ coincides with the imaginary root usually denoted~$\delta$.
%The cone $\d_\infty^\tB$ is defined to be $\Proj_n^{-1}\d_\infty^B$.
%It is the closure of the complement of the union of all cones $\Cone^{\tB;t_0}_t$ for seeds $t$.

\begin{remark}\label{2x2}
Theorem~\ref{affine main} is true for all exchange matrices of affine type, including the case where $B$ is $2\times2$.
The latter case of the theorem is \cite[Theorem 1.2]{RSDom} when $bc=4$.  
Some of the considerations in our proof of Theorem~\ref{affine main} only make sense when $B$ is larger than $2\times2$ so we silently omit this case in what follows.
\end{remark}

\begin{remark}\label{indecomposable}
The definition of a Cartan matrix $A$ of affine type implies in particular that $A$ is block-indecomposable.
Thus an exchange matrix of affine type is also block-indecomposable.
But Theorem~\ref{affine main} can easily be applied to exchange matrices that decompose into matrices of finite and affine type.
For example, suppose $B$ has a diagonal block decomposition $\begin{bsmallmatrix}B_1&0\\0&B_2\end{bsmallmatrix}$ where $B_1$ and $B_2$ are of affine type and $\lambda=\begin{bsmallmatrix}\lambda_1\\\lambda_2\end{bsmallmatrix}$ with each $\lambda_i$ in the relative interior of the imaginary wall~$\d^{B_i}_\infty$, then Theorem~\ref{affine main} and Lemma~\ref{block dom} imply that $\P_\lambda^B$ is a product of line segments.

\end{remark}

To explain and prove Theorem~\ref{affine main}, we quote and develop the relevant theory for different classes of affine-type seeds, namely acyclic seeds, arbitrary seeds, and neighboring seeds.

\subsection{Acyclic seeds of affine type}\label{acyc sec}
Most of what we know about cluster algebras of affine type rests on the fact that every cluster algebra of affine type has a seed whose exchange matrix is an acyclic orientation of a Cartan matrix of affine type.
(Indeed, the analogous observation is true for cluster algebras of finite type.)
The associated Cartan matrix of affine type allows us to model $\dd$-vectors, $\g$-vectors, and the mutation fan by almost-positive Schur roots, sortable elements and (doubled) Cambrian fans.
We now quote some useful results in this setting and prove some additional facts.
More details can be found in \cite{affdenom,framework,afframe,affscat}.

Let $B_0$ be an acyclic $n\times n$ exchange matrix in a cluster algebra of affine type, indexed so that $b_{ij}\ge0$ whenever $i<j$.
%Let $\tB_0$ be an extension with linearly independent columns.
Let $A$ be the Cartan matrix underlying $B_0$, so that $A$ is of affine type.
Let $\delta$ be the shortest positive imaginary root in the root system $\RS$ associated to $A$.
Let $\RSpos$ be the set of positive roots in $\RS$ and let $\Simples=\set{\alpha_1,\ldots,\alpha_n}$ be the simple roots of $\RS$.
The \newword{support} of a root is the set of simple roots appearing with nonzero coefficient in its expression as a linear combination of simple roots.

As before, write $d_1,\ldots,d_n$ for the skew-symmetrizing constants of $B_0$ (meaning that the entries $b_{ij}$ of $B_0$ satisfy $d_i b_{ij}=-d_j b_{ji}$ for all $i,j$).
These constants also symmetrize~$A$ (meaning that the entries $a_{ij}$ of $A$ satisfy $d_i a_{ij}=d_j a_{ji}$ for all $i,j$).
The \newword{simple co-roots} are $\alpha_i\ck= d_i^{-1} \alpha_i$. 
Thus the diagonal matrix $\symmetrizer$ with diagonal entries $d_1,\ldots,d_n$, applied to the simple-root coordinates of a vector, gives the simple-co-root coordinates of that vector.
%Furthermore, the matrix $\symmetrizer B_0\symmetrizer^{-1}$ has entries $d_ib_{ij}d_j^{-1}=-b_{ji}$, so $\symmetrizer B_0\symmetrizer^{-1}=-B_0^T$.
%Also, $(E_{\ep,k}^B)^T=F_{\ep,k}^{-B_0^T}=\symmetrizer F_{\ep,k}^B\symmetrizer^{-1}$.

Let $V$ be the real vector space spanned by the roots, let $V^*$ be the dual space, and let $\br{\,\cdot\,,\,\cdot\,}:V^*\times V$ be the natural pairing between them.
The simple co-roots~$\Simples\ck$ are also in $V$ (each being a positive scaling of the corresponding simple root).
We take as a basis for $V^*$ the \newword{fundamental weights}, defined as the dual basis to~$\Simples\ck$.
Given $\beta\in V$, the notation $\beta^\perp$ indicates the set of vectors in $V^*$ that pair to $0$ with~$\beta$.

\begin{remark}\label{danger: co-roots}
The purpose of using $\Simples$ as the basis for $V$ but using the dual basis to $\Simples\ck$ as the basis for $V^*$ is to incorporate the symmetrizability of $A$ (and therefore the skew-symmetrizability of $B$) seamlessly into the constructions.
However, some care must be taken when we use these constructions to understand the vectors in~$\reals^n$ that appear elsewhere in the paper.
We identify $V^*$ with the vector space $\reals^n$ by identifying the \emph{fundamental weights} with the standard unit basis vectors of $\reals^n$ and also identify $V$ with the vector space $\reals^n$ by identifying the \emph{simple roots} with the standard unit basis vectors.
With these identifications, we view multiplying a vector in~$V$ by a matrix as taking simple-root coordinates to fundamental-weight coordinates.
In the same way, we view a matrix as defining a bilinear form on $V$ in terms of simple-co-root coordinates on the left and simple-root coordinates on the right.
Thus the expressions~$\beta^\perp$ below must be treated with some care when we have passed to vectors in $\reals^n$:
integer vectors in $\beta^\perp$ do not necessarily have zero dot product with the simple-root coordinates of~$\beta$, but, rather, they have zero dot product with the simple-co-root coordinates of $\beta$.
\end{remark}

The Cartan matrix $A$ is the matrix of a bilinear form $K$ on $V$, in the simple-roots basis on the right and the simple-co-roots basis on the left.
Each real root defines a reflection that negates the root and fixes the hyperplane that is orthogonal to the root (with orthogonality defined in terms of $K$).
Let $W$ be the Weyl group generated by these reflections (an affine Coxeter  group with simple reflections $s_1,\ldots,s_n$ corresponding to the simple roots $\alpha_1,\ldots,\alpha_n$).
The action of $W$ on $V$ fixes~$\delta$.

Let $c$ be the Coxeter element that can be read from $B_0$ by multiplying the simple reflections with $s_i$ preceding $s_j$ whenever $b_{ij}>0$.
Because we have indexed $B_0$ so that its $ij$-entry is nonnegative whenever $i<j$, we have $c=s_1\cdots s_n$. 
Define
\begin{align}
\label{rep->}
\TravProj{c}&:=\set{\alpha_1,s_1\alpha_2,\ldots,s_1\cdots s_{n-1}\alpha_n};\\
\label{rep<-}
\TravInj{c}&:=\set{\alpha_n,s_n\alpha_{n-1},\ldots,s_n\cdots s_2\alpha_1}.
\end{align}%
Continuing under the assumption that $A$ is of affine type, the sets $\TravProj{c}$ and $\TravInj{c}$ are disjoint, and the union $\TravProj{c}\cup\TravInj{c}$ contains one representative of each of the $2n$ infinite $c$-orbits of roots.
(See \cite[Chapter~1]{Dlab-Ringel} or \cite[Theorem~1.2]{afforb}.)
The following lemma is \cite[Lemma~1.6]{Dlab-Ringel} or \cite[Lemma~4.1]{afforb}.

\begin{lemma}\label{c to pos}
Suppose $\RS$ is a root system of affine type and $\beta$ is a positive root in~$\RS$. 
Then $c\beta$ is negative if and only if $\beta\in\TravInj{c}$, in which case $-c\beta\in\TravProj{c}$.
Also $c^{-1}\beta$ is negative if and only if $\beta\in\TravProj{c}$, in which case $-c^{-1}\beta\in\TravInj{c}$.
\end{lemma}

Lemma~\ref{c to pos} leads immediately to the following result.
\begin{proposition}\label{who is pos}
Suppose $\RS$ is a root system of affine type and $\beta\in\RS$ is in an infinite $c$-orbit.
Then $c^i\beta$ is a positive root for all $i\ge0$ if and only if $\beta=c^p\gamma$ for some $\gamma\in\TravProj{c}$ and $p\ge0$.
\end{proposition}

The $1$-eigenspace of the action of $c$ on $V$ is spanned by $\delta$.
The eigenvalue $1$ of~$c$ has algebraic multiplicity $2$, so there is a generalized $1$-eigenvector $\gamma_c$ associated to~$\delta$.
This means that $c\gamma_c=\gamma_c+\delta$.
(In fact there are many choices of generalized eigenvector; see \cite[Proposition~3.1]{afforb} for what distinguishes $\gamma_c$.)

We define a bilinear form $\omega_c$ on $V$ whose matrix, in simple-co-root coordinates on the left and simple-root coordinates on the right is~$B_0$.
This bilinear form is skew-symmetric because $B_0$ is skew-symmetrizable and the skew-symmetrizing constants are also the scaling factors between simple roots and simple co-roots.
The following lemma is a version of \cite[Lemma~3.8]{typefree}.

\begin{lemma}\label{omega s}
$\omega_{s_1cs_1}(s_1x,s_1y)=\omega_c(x,y)=\omega_{s_ncs_n}(s_nx,s_ny)$ for any $x$ and $y$ in~$V$.
\end{lemma}

Applying Lemma~\ref{omega s} $n$ times, we obtain the following lemma.

\begin{lemma}\label{omega c}
$\omega_c(x,y)=\omega_c(cx,cy)$ for all $x$ and $y$ in $V$.
\end{lemma}

\begin{proposition}\label{om del}
Suppose $\RS$ is of affine type.
If $\beta=c^p\gamma$ for some $\gamma\in\TravProj{c}$ and $p\ge0$, then $\omega_c(\beta,\delta)\ge0$. 
\end{proposition}
\begin{proof}
The hypothesis is that $\beta=(s_1\cdots s_n)^ps_1\cdots s_{k-1}\alpha_k$ for some $k\in\set{1,\ldots,n}$ and $p\ge0$.
We argue by induction on $np+k$.
If $p=0$ and $k=1$, then $\beta=\alpha_1$ and $\omega_c(\alpha_1,\delta)\ge0$ because every entry of $B_0$ in row $1$ is nonnegative with at least one positive entry and $\delta$ has all simple-root coordinates positive.
Otherwise $s_1\beta=(s_2\cdots s_ns_1)^ps_2\cdots s_{k-1}\alpha_k$ if $k>1$ or $s_1\beta=(s_2\cdots s_ns_1)^{p-1}s_2\cdots s_n\alpha_1$ if $k=1$.
In either case, $s_1\beta$ is of the form $(s_1cs_1)^{p'}\gamma'$ for $\gamma'\in\TravProj{s_1cs_1}$.
By induction, $\omega_{s_1cs_1}(s_1\beta,\delta)\ge0$, and since $s_1\delta=\delta$, Lemma~\ref{omega s} completes the proof.
\end{proof}

%The following proposition is an immediate consequence of \cite[Proposition~2.16]{affdenom}, because the bilinear form $E_c$ in \cite[Proposition~2.16]{affdenom} satisfies $\omega_c(\beta,\beta')=E_c(\beta,\beta')-E_c(\beta',\beta)$.

\begin{proposition}\label{om del fin}
Suppose $\RS$ is of affine type and suppose $x\in V$.
Then $x$ is in a finite $c$-orbit if and only if $\omega_c(x,\delta)=0$.
\end{proposition}
\begin{proof}
Suppose the orbit of $x$ has $k$ elements.
The sum $x+cx+\cdots+c^{k-1}x$ of all vectors in the $c$-orbit of $x$ is fixed by the action of $c$ and thus is a scalar multiple of $\delta$, since $\delta$ spans the $1$-eigenspace of $c$.
Because $\omega_c$ is skew-symmetric, $\omega_c(x+cx+\cdots+c^{k-1}x,\delta)=0$.
Now Lemma~\ref{omega c} says that $k\omega_c(x,\delta)=0$.
The subspace consisting of vectors $x$ such that $\omega_c(x,\delta)=0$ has dimension at most $n-1$.
(For example, one can check that $\omega_c(\alpha_1,\delta)>0$.)
Also, the subspace consisting of vectors in finite $c$-orbits has dimension $n-1$.
(It contains $n-1$ linearly independent roots.)
Since every vector $x$ in a finite $c$-orbit has $\omega_c(x,\delta)=0$, these subspaces must coincide.
\end{proof}

Every finite $c$-orbit of roots consists entirely of positive roots or entirely of negative roots.
(See \cite[Chapter~1]{Dlab-Ringel} or \cite[Theorem~1.2(5)]{afforb}.)
Let $\RST{c}$ be the set of roots in~$\RS$ that are in finite $c$-orbits.
These form a subsystem of $\RS$, in the sense that the set of real roots in $\RST{c}$ is closed under the reflections they define, and all imaginary roots are in $\RST{c}$.
But $\RST{c}$ might not be a root system in the usual (Kac-Moody) sense.
Rather, it is the product of $1$, $2$, or $3$ affine root systems of type $\afftype{A}$, living in a vector space that may be smaller than the sum of the ranks of the irreducible factor ($0$, $1$, or $2$ dimensions smaller):
each factor has a shortest positive imaginary root, and all of these are identified with $\delta$.
(Thus $\RST{c}$ is a Kac-Moody root system only in the case where it has exactly one of these components.)
The set of simple roots of $\RST{c}$ is written as $\SimplesT{c}$.
The total number of roots in $\SimplesT{c}$ is $n-2$ plus the number of factors of $\RST{c}$.
If $\beta_0,\ldots,\beta_k$ are the simple roots in one of the factors of $\RST{c}$, then $\beta_0+\cdots+\beta_k=\delta$.
These simple roots can be indexed so that $c\beta_{i-1}=\beta_i$ for $i=1,\ldots,k$ and $c\beta_k=\beta_0$.

Let $\APTre{c}$ be the set of positive real roots $\beta$ in $\RST{c}$ that can be written uniquely as a linear combination of elements of $\SimplesT{c}$ and denote by $\SuppT(\beta)$ their support with respect to $\SimplesT{c}$.
Let $\APT{c}=\APTre{c}\cup\set\delta$. 
The set $\AP{c}$ of \newword{almost positive Schur roots} is ${-\Simples\cup(\RSpos\setminus\RST{c})\cup\APT{c}}$.
The set $\APre{c}=\AP{c}\setminus\set\delta$ of real roots in $\AP{c}$ is precisely the set of denominator vectors of cluster variables \cite[Theorem~1.2]{affdenom}, with respect to the initial exchange matrix $B_0$.
There is a notion \cite[Definition~4.3]{affdenom} of \newword{$c$-compatibility} of almost positive Schur roots (similar in spirit to the compatibility of almost positive roots in~\cite{ga}) such that a set of almost positive real Schur roots is pairwise $c$-compatible if and only if the corresponding cluster variables are contained in a common cluster.
Thus almost positive real Schur roots and $c$-compatibility model the $\dd$-vector fan (the fan obtained by replacing each cluster with the nonnegative linear span of its $\dd$-vectors).
In addition to these $\dd$-vector cones, there are cones spanned by pairwise $c$-compatible sets that contain~$\delta$.
The set of all these cones is a complete fan $\Fan_c(\RS)$ in $V$ called the \newword{affine generalized associahedron fan}.
%Just as in finite type \cite{ga}, the notion of $c$-compatibility derives from a more detailed notion of $c$-compatibility degree, which assigns an integer to any pair of almost-positive Schur roots.
%(Two roots are $c$-compatible if and only if their $c$-compatibility degree is~$0$.)
%We will not need the definition of $c$-compatibility, but we give some of its crucial properties, particularly as it relates to roots in $\APT{c}$.
The following proposition is \cite[Proposition~5.6]{affdenom}.

\begin{proposition}\label{delta c compat}
A root $\alpha\in\APre{c}$ is $c$-compatible with~$\delta$ if and only if $\alpha\in\APTre{c}$.
\end{proposition}

Two roots $\alpha,\beta\in\APTre{c}$ are called \newword{nested} if $\SuppT(\alpha)\subseteq\SuppT(\beta)$ or if $\SuppT(\beta)\subseteq\SuppT(\alpha)$; they are \newword{spaced} if $[\SuppT(c^{-1}\alpha)\cup\SuppT(\alpha)\cup\SuppT(c\alpha)]\cap\SuppT(\beta)=\varnothing$.
The following proposition is \cite[Proposition~5.12]{affdenom}.

\begin{proposition}\label{compatible in tubes}
Two distinct roots $\alpha$ and $\beta$ in $\APTre{c}$ are $c$-compatible if and only if they are nested or spaced.
\end{proposition}

Thus roots $\alpha,\beta\in\APTre{c}$ will always be $c$-compatible if they are in different factors of $\RST{c}$, while
$c$-compatibility of roots in the same factor can be understood by picturing their supports on a cycle.
(The cycle consists of the simple roots in that factor with edges connecting each $\beta$ with $c\beta$.)

\begin{remark}
  The roots appear as ``tubes'' in the cycle (in the sense of graph associahedra), and $c$-compatibility is the usual compatibility of tubes:
  either the tubes are nested or there is at least one vertex between them on both sides.
  %However, the representation theory literature uses the term ``tubes'' for the components
  However, this terminology conflicts with the representation theory literature use of the term ``tubes'', so we will avoid the term altogether.
\end{remark}

There is a bijection $\tau_c:\AP{c}\to\AP{c}$ that induces an automorphism of $\Fan_c(\RS)$.
The definition of $\tau_c$ and some of its key properties are found in \cite[Section~3]{affdenom}.
Here we quote and rephrase some of those properties from \cite[Proposition~3.12]{affdenom}.
The negative simple roots $-\Simples$ are a set of representatives for the $n$ infinite $\tau_c$-orbits in $\AP{c}$.
The finite $\tau_c$-orbits are exactly the finite $c$-orbits and the actions of $c$ and $\tau_c$ agree on these orbits.
We reorganize some other assertions of \cite[Proposition~3.12]{affdenom} as follows.

\begin{proposition}\label{delta limit}
Suppose $\beta\in\AP{c}\setminus\APT{c}$ and let $I$ be the unique integer such that $\tau_c^I(\beta)\in-\Simples$.
Let $\ell$ be the least common multiple of the sizes of finite $c$-orbits.
\begin{enumerate}[label=\bf\arabic*., ref=\arabic*]
\item \label{c pos}
For $i\not\in\set{I-1,I}$, the action of $\tau_c$ on $\tau_c^i(\beta)$ is the same as the action of~$c$.
\item \label{c neg}
For $i\not\in\set{I,I+1}$, the action of $\tau_c^{-1}$ on $\tau_c^i(\beta)$ is the same as the action of~$c^{-1}$.
\item \label{K pos}
$K(\gamma_c,\tau_c^i(\beta))>0$ for $i>I$.
\item \label{K neg}
$K(\gamma_c,\tau_c^i(\beta))<0$ for $i<I$.
\item \label{exists}
There exists a positive real number $a$ such that $\tau_c^{i+\ell}(\beta)=\tau_c^i(\beta)+a\delta$ when $i>I$ and $\tau_c^{i-\ell}(\beta)=\tau_c^i(\beta)+a\delta$ when $i<I$.
\item \label{lim}
$\displaystyle\lim_{i\to\infty}\frac{\tau_c^i\beta}{|\tau_c^i\beta|}=\lim_{i\to-\infty}\frac{\tau_c^i\beta}{|\tau_c^i\beta|}=\frac{\delta}{|\delta|}$.
\end{enumerate}
\end{proposition}

\begin{proof}
The first two assertions are \cite[Proposition~3.12(2)]{affdenom}.
The next two assertions are \cite[Proposition~3.12(8)]{affdenom}.
The hyperplane $\set{v\in V:K(\gamma_c,v)=0}$ is the direct sum of all eigenspaces of $c$, so $c^\ell$ fixes this subspace pointwise.
Write $\tau_c^i(\beta)$ as $q\gamma_c+v$ with $K(\gamma_c,v)=0$. 
(This can be done because $\set{v\in V:K(\gamma_c,v)=0}$ is a hyperplane and $\gamma_c$ is not in the hyperplane.  
Indeed, $K(\gamma_c,\gamma_c)>0$.)
Since $K$ is positive semidefinite and $\gamma_c$ is not a multiple of $\delta$, we have $q>0$ for $i>I$ by Assertion~\ref{K pos} and $q<0$ for $i<I$ by Assertion~\ref{K neg}.
Applying Assertion~\ref{c pos} and using the fact that $c\gamma_c=\gamma_c+\delta$, we have $\tau_c^\ell(\tau_c^i(\beta))=q\ell\delta+q\gamma_c+v$ and $\tau_c^{-\ell}(\tau_c^i(\beta))=-q\ell\delta+q\gamma_c+v$, and Assertions~\ref{exists} and~\ref{lim} follow.
\end{proof}

A maximal set of pairwise $c$-compatible roots in $\AP{c}$ is called a \newword{$c$-cluster}.
A \newword{real $c$-cluster} is a $c$-cluster that does not contain $\delta$.
Real $c$-clusters have $n$ roots and span maximal cones of the $\dd$-vector fan.
An \newword{imaginary $c$-cluster} is a $c$-cluster that contains $\delta$.
Imaginary $c$-clusters have $n-1$ roots and span cones outside the $\dd$-vector fan.

There is a piecewise linear homeomorphism $\nu_c$ from $V$ to $V^*$ that is linear on every cone of $\Fan_c(\RS)$ and thus defines a fan $\nu_c(\Fan_c(\RS))$ in $V^*$.
The fan $\nu_c(\Fan_c(\RS))$ is the mutation fan for $B_0^T$ \cite[Theorem~2.9]{affscat}, and the map $\nu_c$ restricts to an isomorphism from the $\dd$-vector fan of $B_0$ to the $\g$-vector fan \cite[Theorem~3.4]{affscat}.
The cones of $\Fan_c(\RS)$ spanned by $\APT{c}$ map to the \newword{imaginary cones} of $\nu_c(\Fan_c(\RS))$.
These cover the closure of the complement of the $\g$-vector fan, a codimension-$1$ cone that we call the \newword{imaginary wall}~$\d_\infty$.
The imaginary wall is contained in $\delta^\perp$.
The vector $\nu_c(\delta)$ spans a ray called the \newword{imaginary ray} that is in every imaginary cone.
This is the only ray of the fan $\nu_c(\Fan_c(\RS))$ that is not spanned by the $\g$-vector of a cluster variable.
The following lemma is \cite[Lemma~5.9]{affscat}, written in different notation.
(The statement in \cite{affscat} is phrased in terms of $\omega_c(\,\cdot\,,\delta)$.)

\begin{lemma}\label{nu delta}  
Suppose $B_0$ is an acyclic exchange matrix of affine type.
The vector~$\nu_c(\delta)$ spanning the imaginary ray is $-\frac12B_0\delta$.
\end{lemma}

%Since $\mu_{12\cdots n}(B)=B$, the map $\eta^{B_0^T}_{12\cdots n}$ is an automorphism of the fan and separately of the $\g$-vector fan.
The following proposition is part of \cite[Proposition~7.33]{affscat}.

\begin{proposition}\label{eta nice} 
Suppose $B_0$ is an acyclic exchange matrix of affine type and the Coxeter element associated to $B_0$ is $c=s_1\cdots s_n$.
\begin{enumerate}[label=\bf\arabic*., ref=\arabic*]
\item \label{nu tau}   
$\eta^{B^T}_{12\cdots n}\circ\nu_c=\nu_c\circ\tau_c$ as maps on $\AP{c}$.
\item\label{eta aut mut}
$\eta^{B_0^T}_{12\cdots n}$ is an automorphism of $\F_{B^T}$.
\item\label{eta is c}
$\eta^{B_0^T}_{12\cdots n}$ fixes $\d_\infty$ as a set, agrees with $c$ on $\d_\infty$, and has finite order on $\d_\infty$.
\end{enumerate}
\end{proposition}

%This is the version of the proposition before Nathan made some changes 26 Nov 2025
%\begin{proposition}\label{eta nice} 
%Suppose $B_0$ is an acyclic exchange matrix of affine type and the Coxeter element associated to $B_0$ is $c=s_1\cdots s_n$.
%Then the piecewise linear map $\eta^{B_0^T}_{12\cdots n}$
%\begin{enumerate}[label=\bf\arabic*., ref=\arabic*]
%\item \label{nu tau}   
%has $\eta^{B^T}_{12\cdots n}\circ\nu_c=\nu_c\circ\tau_c$ as maps on $\AP{c}$, 
%\item\label{eta aut mut}
%is an automorphism of $\F_{B^T}$,
%\item\label{eta is c}
%fixes $\d_\infty$ as a set, agrees with $c$ on $\d_\infty$, and has finite order on $\d_\infty$.
%\end{enumerate}
%\end{proposition}
%
\begin{proposition}\label{delta limit eta}
Suppose $B_0$ is an acyclic exchange matrix of affine type and let~$\ell$ be the least common multiple of the sizes of finite $c$-orbits.
Suppose $x\in V^*\setminus\d_\infty$.
Then there exist integers $I_+$ and $I_-$ such that the following assertions hold.
\begin{enumerate}[label=\bf\arabic*., ref=\arabic*]
\item \label{in big dinf}
For $i>I_+$ or $i<I_-$, $\br{\bigl(\eta_{12\cdots n}^{B^T}\bigr)^i(x),\beta}<0$ for all $\beta\in\RS$ such that $\omega_c(\beta,\delta)>0$.
\item \label{c agree}
%\item \label{c agree +}
For $i>I_+$ or $i<I_-$, $\bigl(\eta_{12\cdots n}^{B^T}\bigr)^i(x)$ and $\nu_c(\delta)$ are in the same domain of definition of $\eta_{12\cdots n}^{B^T}$ and $\eta_{12\cdots n}^{B^T}$ agrees with $c$ on $\bigl(\eta_{12\cdots n}^{B^T}\bigr)^i(x)$.
%\item \label{c agree -}
%For $i<I_-$, $\bigl(\eta_{12\cdots n}^{B^T}\bigr)^i(x)$ and $\nu_c(\delta)$ are in the same domain of definition of $\bigl(\eta_{12\cdots n}^{B^T}\bigr)^{-1}$ and $\bigl(\eta_{12\cdots n}^{B^T}\bigr)^{-1}$ agrees with $c^{-1}$ on $\bigl(\eta_{12\cdots n}^{B^T}\bigr)^i(x)$,
\item \label{correct side}
$\brr{\bigl(\eta_{12\cdots n}^{B^T}\bigr)^i(x),\delta}>0$ for $i>I_+$ and $\brr{\bigl(\eta_{12\cdots n}^{B^T}\bigr)^i(x),\delta}<0$ for $i<I_-$.
\item \label{pos a}
There exists positive real $a$ such that $\bigl(\eta_{12\cdots n}^{B^T}\bigr)^{i+\ell}(x)=\bigl(\eta_{12\cdots n}^{B^T}\bigr)^i(x)+a\nu_c(\delta)$ when $i>I_+$ and $\bigl(\eta_{12\cdots n}^{B^T}\bigr)^{i-\ell}(x)=\bigl(\eta_{12\cdots n}^{B^T}\bigr)^i(x)+a\nu_c(\delta)$ when $i<I_-$.
\item \label{eta lim}
$\displaystyle\lim_{i\to\infty}\frac{\bigl(\eta_{12\cdots n}^{B_0^T}\bigr)^i(x)}{\bigl|\bigl(\eta_{12\cdots n}^{B_0^T}\bigr)^i(x)\bigr|}=\lim_{i\to-\infty}\frac{\bigl(\eta_{12\cdots n}^{B^T}\bigr)^i(x)}{\bigl|\bigl(\eta_{12\cdots n}^{B^T}\bigr)^i(x)\bigr|}=\frac{\nu_c(\delta)}{|\nu_c(\delta)|}$.
\end{enumerate}
\end{proposition}

\begin{proof}
We will show that each assertion holds for large enough $i$ or small enough~$i$, as appropriate.
Then $I_+$ and $I_-$ are implicitly defined to realize ``large enough'' or ``small enough'' for all the assertions simultaneously.

Since $x$ is not in $\d_\infty$, it is in a full-dimensional cone of $\nu_c(\Fan_c(\RS))$.
Thus $(\nu_c)^{-1}(x)$ is in a maximal cone of $\Fan_c(\RS)$ spanned by some $\beta_1,\ldots,\beta_n$, indexed with $\beta_1,\ldots,\beta_p\in\AP{c}\setminus\APT{c}$ and $\beta_{p+1},\ldots,\beta_n\in\APT{c}$ with $p\ge2$.
There exist $q_1,\ldots,q_n\ge0$ such that $(\nu_c)^{-1}(x)=\sum_{j=1}^nq_j\beta_j$ with $q_j>0$ for at least one $j\in\set{1,\ldots,p}$.
By Proposition~\ref{delta limit}, there exist integers $I_1,\ldots,I_p$ and positive real $a_1,\ldots,a_p$ such that $\tau_c^{i+\ell}\beta_j=(\tau_c^i\beta_j)+a_j\delta$ whenever $i>I_j$ and $\tau_c^{i-\ell}\beta_j=(\tau_c^i\beta_j)+a_j\delta$ whenever $i<I_j$, for all $j=1,\ldots,p$.
Set $a=\sum_{j=1}^p q_ja_j>0$.
Then $\tau_c^{i+\ell}((\nu_c)^{-1}(x))=\tau_c^i((\nu_c)^{-1}(x))+a\delta$ for $i>\max(I_1,\ldots,I_p)$ and $\tau_c^{i-\ell}((\nu_c)^{-1}(x))=\tau_c^i((\nu_c)^{-1}(x))+a\delta$ for $i<\min(I_1,\ldots,I_p)$.
Applying $\nu_c$ and appealing to Proposition~\ref{eta nice}.\ref{nu tau}, we see that $\bigl(\eta_{12\cdots n}^{B^T}\bigr)^{i+\ell}(x)=\bigl(\eta_{12\cdots n}^{B^T}\bigr)^i(x)+a\nu_c(\delta)$ for $i>\max(I_1,\ldots,I_p)$ and $\bigl(\eta_{12\cdots n}^{B^T}\bigr)^{i-\ell}(x)=\bigl(\eta_{12\cdots n}^{B^T}\bigr)^i(x)+a\nu_c(\delta)$ for $i>\min(I_1,\ldots,I_p)$.
This is Assertion~\ref{pos a}, and Assertion~\ref{eta lim} follows immediately.

The proof of Proposition~\ref{eta nice}.\ref{eta is c} in \cite[Proposition~7.33]{affscat} includes checking, for each of the $n$ mutation steps of applying $\eta^{B^T}_{12\cdots n}$ to $\nu_c(\delta)$, which of the cases of the mutation map applies.
At each step, the condition yields strict inequality (i.e. the vector in question is not on the hyperplane separating the two conditions), so $\nu_c(\delta)$ is in the \emph{interior} of the domain of definition of $\eta^{B^T}_{12\cdots n}$.
Thus Assertion~\ref{eta lim} implies that for large enough or small enough $i$, the $\bigl(\eta_{12\cdots n}^{B_0^T}\bigr)^i(x)$ are in the same domain of definition of $\eta^{B^T}_{12\cdots n}$ and $\eta^{B^T}_{12\cdots n}$agrees with $c$ on $\bigl(\eta_{12\cdots n}^{B_0^T}\bigr)^i(x)$.
This is Assertion~\ref{c agree}.

Lemma~\ref{nu delta} says that $\nu_c(\delta)\in V^*$ is $-\frac12\omega_c(\,\cdot\,,\delta)$, so for any $\beta\in\RS$, the pairing $\br{\nu_c(\delta),\beta}$ is $-\frac12\omega(\beta,\delta)$.
Thus Assertion~\ref{pos a} implies that for large enough or small enough $i$, we have $\brr{\bigl(\eta_{12\cdots n}^{B^T}\bigr)^i(x),\beta}<0$ for all $\beta\in\RS$ such that $\omega_c(\beta,\delta)>0$.
This is Assertion~\ref{in big dinf}.  

Now $2\nu_c(\delta)=-\omega_c(\,\cdot\,,\delta)$ is a $1$-eigenvector for $c$.
(This is \cite[Lemma~3.5]{afforb} and follows immediately from Lemma~\ref{omega c} and the fact that $c\delta=\delta$.)
Also $-\omega_c(\,\cdot\,,\gamma_c)$ is a generalized $1$-eigenvector for $c$, associated to the $1$-eigenvector $2\nu_c(\delta)$.
(One can check this directly, or, for example, in light of Lemma~\ref{nu delta} this is \cite[Lemma~2.8]{affncA}.)
Using the fact \cite[Lemma~3.5]{afforb} that $\omega_c(\delta,\,\cdot\,)$ is a negative scaling of $K(\gamma_c,\,\cdot\,)$, we compute $\brr{-\omega_c(\,\cdot\,,\gamma_c),\delta}=-\omega_c(\delta,\gamma_c)=K(\gamma_c,\gamma_c)$, which is positive because $K$ is positive semidefinite and $\gamma_c$ is not a multiple of $\delta$.
Since $c$ fixes $\d_\infty$ as a set and has finite order $\ell$ on $\d_\infty$ (Proposition~\ref{eta nice}.\ref{eta is c}), it fixes the entire hyperplane $\delta^\perp$ as a set.
Thus we can write $\bigl(\eta_{12\cdots n}^{B^T}\bigr)^i(x)$ as $q\omega_c(\,\cdot\,,\gamma_c)+v$ for $v\in\delta^\perp$ fixed by $c^\ell$.
Since $c^\ell\omega_c(\,\cdot\,,\gamma_c)=\omega_c(\,\cdot\,,\gamma_c)+2\ell\nu_c(\delta)$, Assertion~\ref{pos a} implies that $q>0$ for large enough $i$.
Thus $\brr{\bigl(\eta_{12\cdots n}^{B^T}\bigr)^i(x),\delta}>0$ for large enough~$i$.
  Similarly, $c^{-\ell}\omega_c(\,\cdot\,,\gamma_c)=\omega_c(\,\cdot\,,\gamma_c)-2\ell\nu_c(\delta)$, so $q<0$ and thus $\brr{\bigl(\eta_{12\cdots n}^{B^T}\bigr)^i(x),\delta}<0$ for small enough~$i$.
This is Assertion~\ref{correct side}.
\end{proof}

%%No, this used the assertion in Proposition~\ref{eta nice} that we lost.
%We also point out an implication between some of the conclusions of Proposition~\ref{eta nice} for large or small enough~$i$.
%
%\begin{lemma}\label{dinf+- add}
%Suppose $x\in\set{x\in V^*:\br{x,\beta}<0,\,\forall \beta\in\RSfin\text{ s.t. }\omega_c(\beta,\delta)>0}$ and let $m$ be the least common multiple of the sizes of finite $c$-orbits.
%\begin{enumerate}[label=\bf\arabic*., ref=\arabic*]
%\item \label{dinf+ add}
%If $\br{x,\delta}>0$, then $\bigl(\eta_{12\cdots n}^{B^T}\bigr)^{m}(x)=x+a\nu_c(\delta)$ for some positive~$a$.
%\item \label{dinf- add}
%If $\br{x,\delta}<0$, then $\bigl(\eta_{12\cdots n}^{B^T}\bigr)^{-m}(x)=x+a\nu_c(\delta)$ for some positive~$a$.
%\end{enumerate}
%\end{lemma}
%\begin{proof}
%As in the proof of Proposition~\ref{eta nice}, write $x$ as $q\omega_c(\,\cdot\,,\gamma_c)+v$ for $v\in\delta^\perp$ fixed by $c^m$.
%Since $\brr{-\omega_c(\,\cdot\,,\gamma_c),\delta}=K(\gamma_c,\gamma_c)>0$, $q$ has the same sign as $\br{x,\delta}>0$.
%Proposition~\ref{eta nice}.\ref{eta is c bigger} says that $\bigl(\eta_{12\cdots n}^{B^T}\bigr)^{m}(x)=c^mx=x+qm\nu_c(\delta)$.
%\end{proof}

We now prove our first preliminary result about dominance regions in the affine case.

\begin{lemma}\label{P in dinf}
Suppose $B_0$ is acyclic of affine type.
If $\lambda\in\d_\infty$, then $\P^{B_0}_\lambda\subseteq\d_\infty$.
\end{lemma}
\begin{proof}
By definition, $\P^{B_0}_\lambda\subseteq\P^{B_0}_{\lambda,\varnothing}=\set{\lambda+B_0\alpha:\alpha\ge0}$.
Write the entries of $\lambda$ as $(\lambda_1,\ldots,\lambda_n)$.
Since $B_0$ is acyclic and its first row has nonnegative entries, every point in $\P^{B_0}_\lambda$ has first coordinate greater than or equal to $\lambda_1$.
Also, since $\delta$ has all entries positive and the first row of $B_0$ is nonzero, $B_0\delta$ has first coordinate strictly positive.

Let $\ell$ be the order of $\eta^{B_0^T}_{12\cdots n}$ on $\d_\infty$ (see Proposition~\ref{eta nice}.\ref{eta is c}).
Then $\bigl(\eta^{B_0^T}_{12\cdots n}\bigr)^\ell$ fixes $\d_\infty$ pointwise.
Also, $\mu_{12\cdots n}(B_0)=B_0$, so Lemma~\ref{shift}.\ref{shift all} says ${\bigl(\eta^{B_0^T}_{12\cdots n}\bigr)^\ell(\P_\lambda^{B_0})=P_\lambda^{B_0}}$.

Now suppose $\P^{B_0}_\lambda$ contains a point $x\not\in\d_\infty$.
Applying Proposition~\ref{delta limit eta}.\ref{eta lim}, we find a sequence of points in $\P^{B_0}_\lambda$ that approach the direction of $\nu_c(\delta)$ with the component in the direction of $\nu_c(\delta)$ increasing without bound.
By Lemma~\ref{nu delta}, and the fact that $B_0\delta$ has first coordinate strictly positive, we conclude that $\P^{B_0}_\lambda$ contains points with first coordinate strictly \emph{negative}.
This contradiction shows that $\P^{B_0}_\lambda\subseteq\d_\infty$.
\end{proof}

We conclude this section with a brief discussion of $c$-sortable elements.  
Recalling that $c=s_1\cdots s_n$, write $c^\infty$ for the infinite word $s_1\cdots s_n|s_1\cdots s_n|s_1\cdots s_n|\cdots$.
(The symbols ``$|$'' are not considered as letters of the word, but are \newword{dividers} placed in the word for convenience.)
Every element $w$ of the group $W$ can be written (in many ways) as a subword of $c^\infty$.
Each subword is specified by a finite sequence of positions in $c^\infty$.
Out of all those subwords, the \newword{$c$-sorting word} for $w$ is the subword with the \emph{lexicographically leftmost} sequence of positions.
The $c$-sorting word for $w$ is also specified by a finite sequence of nonempty subsets of $\set{s_1,\ldots,s_n}$, namely the set of letters of the $c$-sorting word that are before the first divider, the set of letters between the first and second dividers, between the second and third dividers, etc.
The element $w$ is \newword{$c$-sortable} if this sequence of subsets is weakly decreasing.
(In other words, working from left to right, once a letter $s_i$ of $c^\infty$ is skipped in the $c$-sorting word, it never appears afterwards.)

The following lemma is immediate from the definition.
\begin{lemma}\label{any prefix}
If $a_1\ldots a_r$ is the $c$-sorting word for a $c$-sortable element, then for any $p\in\set{0,1,\ldots,r}$, the prefix $a_1\cdots a_p$ is the $c$-sorting word for a $c$-sortable element.
\end{lemma}

The $c$-sortable elements and their $c$-sorting words contain a lot of combinatorial data about the mutation fan.
%Each $c$-sortable element determines an $n$-dimensional cone $\Cone_c(v)$ in $V^*$ whose inward-facing normals are certain roots or co-roots in $\RS\subset V$ determined by the $c$-sorting word as follows.
Given a $c$-sortable element $v$ and an index $i\in\set{1,\ldots,n}$, let $a_1\cdots a_p$ be the prefix of the $c$-sorting word for $v$ consisting of all the letters before the first place where $s_i$ is skipped in the $c$-sorting word.
Define $C_c^i(v)$ to be the root $a_1\cdots a_p\alpha_i$ and define $C_c^i(v)\ck$ to be the corresponding co-root $a_1\cdots a_p\alpha\ck_i$.
The skip of $s_i$ in the $c$-sorting word is \newword{unforced} if $a_1\cdots a_ps_i$ is a reduced word.
Otherwise, it is \newword{forced}.
The root $C_c^i(v)$ is positive if and only if the skip is unforced.

For each $c$-sortable element $v$, define a cone  
\[\Cone_c(v)=\bigcap_{\beta \in C_c(v)}\set{x\in V^*:\br{x,\beta} \geq 0},\]
where $C_c(v):=\{C_c^i(v):1\le i\le n\}$.
%(Since each $C_c^i(v)\ck$ is a scaling of $C_c^i(v)$, we could equally well write $\bigcap_{\beta\ck\in C_c(v)\ck}\set{x\in V^*:\br{x,\beta\ck} \geq 0}$).
In general acyclic (but not necessarily affine) type, the union of these cones $\Cone_c(v)$ and their faces, over all $c$-sortable elements $v$, is a fan called the \newword{$c$-Cambrian fan}, which is a subfan of the $\g$-vector fan in $V^*$.
In particular, each $c$-sortable element~$v$ specifies a unique seed.
The $C$-vectors of this seed are encoded in the $c$-sorting word for~$v$.
Specifically, the $C$-vector in position $i$ is the integer vector given by the simple-root coordinates of $C_c^i(v)$.
(See \cite[Therem~1.1]{framework} and \cite[Theorem~5.12]{framework}.)
The simple-co-root coordinates of $C_c^i(v)\ck$ give the $C$-vector of the seed specified by~$v$ in the cluster algebra with initial exchange matrix $-B_0^T$.
Thus by \cite[Theorem~1.1]{framework}, the exchange matrix at the seed associated to $v$ has $ij$-entry $\omega_c(C_c^i(v)\ck,C_c^j(v))$.

When $B_0$ is of finite type (equivalently, when $W$ is finite), the $c$-Cambrian fan coincides with the $\g$-vector fan, but otherwise (when $W$ is infinite), the $c$-Cambrian fan is not complete and is a proper subfan of the $\g$-vector fan.
When $B_0$ is of affine type (equivalently, when $W$ is of affine type), the $\g$-vector fan is the \newword{doubled $c$-Cambrian fan}: the union of the \mbox{$c$-Cambrian} fan and the image of the \mbox{$c^{-1}$-Cambrian} fan under the antipodal map.  % (more briefly, the Cambrian fan and the \newword{opposite Cambrian fan}).
In particular, the $c$-Cambrian fan covers the open halfspace $\set{x\in V^*:\br{x,\delta}>0}$.

\begin{proposition}\label{tack on c}
Suppose $W$ is of affine type and suppose $v$ is a $c$-sortable element whose $c$-sorting word starts with $s_1\cdots s_n$.
Then $cv$ is $c$-sortable and its $c$-sorting word is $s_1\cdots s_n$ followed by the $c$-sorting word for~$v$.
\end{proposition}
\begin{proof}
Let $a_1\cdots a_r$ be the $c$-sorting word for $v$.
We first show that $s_na_1\cdots a_r$ is reduced.
This means showing that $\alpha_n\not\in\set{a_1\cdots a_{p-1}\alpha(a_p):p=1,\ldots,r}$, where $\alpha(a_p)$ is the simple root associated to the simple reflection~$a_p$.
Suppose to the contrary that $\alpha_n=a_1\cdots a_{p-1}\alpha(a_p)$ for some~$p$.
Since $\set{a_1\cdots a_{p-1}\alpha(a_p):p=1,\ldots,n}$ equals $\TravProj{c}$ and $\alpha_n\in \TravInj{c}$, we must have $p>n$.
But then Lemma~\ref{c to pos} says that $a_{n+1}\cdots a_{p-1}\alpha(a_p)$ is negative, contradicting the fact that $a_{n+1}\cdots a_r$ is reduced.

We conclude that $s_na_1\cdots a_r$ is reduced.
Therefore, we see that $s_na_1\cdots a_r$ is the $(s_ncs_n)$-sorting word for $s_nv$ and that $s_nv$ is $(s_ncs_n)$-sortable.
Applying that fact $n$ times implies the proposition.
\end{proof}

\begin{proposition}\label{sort green}
Suppose $B_0$ is an acyclic $n\times n$ exchange matrix of affine type.
Let $v$ be a $c$-sortable element and let $\kk=k_r\cdots k_1$ be the \emph{reverse} of the sequence of indices in the $c$-sorting word for $v$.
Then $\kk$ is a green sequence for $B_0$.
\end{proposition}
\begin{proof}
Write $t_0\overset{k_1}{\edge}t_1\overset{k_2}{\edge}\,\cdots\,\overset{k_r}{\edge}t_r=t$.
Lemma~\ref{any prefix} implies that each $t_p$ is associated to a $c$-sortable element.
Each $t_{p+1}$ is obtained from $t_p$ by mutating in position $k_{p+1}$, the last letter of the $c$-sorting word associated to $t_{p+1}$.
Since the $c$-sorting word associated to $t_{p+1}$ is reduced, the skip of $s_{k_{p+1}}$ is unforced, and thus the $C$-vector at $t_p$ indexed by $k_{p+1}$ is positive.
\end{proof}

\begin{proposition}\label{any B is sort}
Suppose $B_0$ is an acyclic exchange matrix of affine type and suppose $B$ is mutation-equivalent to $B_0$.
Let $t_0$ be a seed with exchange matrix $B_0$.
Then there exists a $c$-sortable element~$v$ such that, writing $\kk$ for the \emph{reverse} of the sequence of indices in the $c$-sorting word for~$v$ and defining seeds ${t_0\overset{k_1}{\edge}t_1\overset{k_2}{\edge}\,\cdots\,\overset{k_r}{\edge}t_r=t}$, the exhange matrix at $t$ is $B$.
These $v$ and $t$ can be chosen so that $\bigl(G_{t_p}^{-B_0^T;t_0}\bigr)^T\delta$ has positive entries for all ${p=0,\ldots,r}$.
\end{proposition}
\begin{proof}
Since $B$ is mutation-equivalent to $B_0$, taking $B_0$ to be the exchange matrix at $t_0$, there exists a seed $t'$ with $B$ as exchange matrix.
Taking $x$ in the interior of $\Cone^{B_0;t_0}_{t'}$, Proposition~\ref{delta limit eta}.\ref{correct side} says that $\br{\bigl(\eta_{12\cdots n}^{B^T}\bigr)^i(x),\delta}>0$  for large enough~$i$.
Thus $\bigl(\eta_{12\cdots n}^{B^T}\bigr)^i(x)$ is in the $c$-Cambrian fan, and so $\bigl(\eta_{12\cdots n}^{B^T}\bigr)^i\bigl(\Cone^{B_0;t_0}_{t'}\bigr)$ is a cone $\Cone^{B_0;t_0}_t$ for some seed $t$ associated to a $c$-sortable element $v$.
The automorphism $\bigl(\eta_{12\cdots n}^{B^T}\bigr)^i$ of $\nu_c(\Fan_c(\RS))$ maps $\Cone^{B_0;t_0}_t$ to $\Cone^{B_0;t_0}_{t'}$.
The map $\bigl(\eta_{12\cdots n}^{B^T}\bigr)^i$ acts as an initial seed mutation $\mu_{(12\cdots n)^i}$, which fixes $B_0$.
If $\ll$ is the sequence of indices that mutates $t_0$ to $t'$, then $\ll(12\cdots n)^i$ mutates $t_0$ to $t$.
We see that the exchange matrix at $t$ is $\mu_{\ll(12\cdots n)^i}(B_0)=\mu_\ll(B_0)=B$.

Increasing $i$ in the previous paragraph means using Proposition~\ref{tack on c} to replace~$v$ by a $c$-sortable element~$c^jv$.
It is known \cite[Proposition~4.6]{afframe} that there are only finitely many $c$-sortable elements $v$ such that the nonnegative span of $G_{t_p}^{-B_0^T;t_0}$ intersects the halfspace $\set{x\in V^*:\br{x,\delta}<0}$. 
Thus we choose $i$ large enough (i.e. choose $j$ large enough) so that $\bigl(G_{t_p}^{-B_0^T;t_0}\bigr)^T\delta$ has nonnegative entries for all $p\ge jn$.
But for $0\le p\le jn$, the prefix $a_1\cdots a_p$ of the $c$-sorting word for $v$ is a prefix of $c^\infty$.
By the characterization of $C$-vectors in terms of skips, we see that the columns of $C_{t_p}^{B_0;t_0}$ are the inward-facing normals of the cone obtained by applying $a_1\cdots a_p$ to the positive cone.
The resulting cone $\Cone^{B_0;t_0}_{t_p}$ is therefore contained in the Tits cone,
so that every nonzero vector in $\Cone^{B_0;t_0}_{t_p}$ pairs strictly positively with~$\delta$.
The cone $\Cone^{B_0;t_0}_{t_p}$ is the nonnegative span of the rows of $\bigl(C_t^{B_0;t_0}\bigr)^{-1}$, which equals $\bigl(G_t^{-B_0^T;t_0}\bigr)^T$ by \cite[Theorem~1.2]{NZ}.
We see that the columns of $G_{t_p}^{-B_0^T;t_0}$ pair strictly positively with $\delta$.
\end{proof}

\subsection{General seeds of affine type}
We now prove some facts about exchange matrices of affine type, without the requirement of acyclicity.

Let $B$ be an exchange matrix of affine type.
Then $B$ is mutation-equivalent to some acyclic exchange matrix $B_0$ of affine type.
Mutation maps $\eta_\kk^{B^T}$ are homeomorphisms of $V^*$ and isomorphisms of mutation fans and also constitute initial seed mutations of $\g$-vectors.
Thus the following structure carries over from the acyclic seed:
there is an \newword{imaginary ray}, the unique ray of the mutation fan that is not spanned by the $\g$-vectors of cluster variables.
An \newword{imaginary cone} of $\F_{B^T}$ is a cone containing the imaginary ray.
The maximal imaginary cones have codimension~$1$.
The mutation fan $\F_{B^T}$ consists of the $\g$-vector fan of $B$ and the imaginary cones.
Write $\d^B_\infty$ for the union of the imaginary cones.
This is a finite union of cones of codimension $1$, and we will show that it is contained in a hyperplane.
Thus we will again call $\d^B_\infty$ the \newword{imaginary wall}.
Since $\d^B_\infty$ is the image of $\d_\infty$ under the fan isomorphism $\eta_\kk^{B^T}$, the following lemma is immediate from Lemmas~\ref{shift} and~\ref{P in dinf}.

\begin{lemma}\label{P in dBinf}
Suppose $B$ is of affine type.
If $\lambda\in\d^B_\infty$, then $\P^B_\lambda\subseteq\d^B_\infty$.
\end{lemma}

Let $B_0$ be an acyclic exchange matrix that is mutation-equivalent to $B$ and consider the cluster pattern with $B_0$ the seed at $t_0$.
We will construct a particular seed $t$ with exchange matrix $B$ and a green sequence $\kk=k_r\cdots k_1$ for $B_0$ that mutates $t_0$ to $t$ and then establish some additional properties of $\kk$.

Proposition~\ref{any B is sort} says there exists a $c$-sortable element $v$ whose associated seed has exchange matrix $B$.
Let $\kk$ be the reverse of the sequence of indices in the $c$-sorting word for $v$ and write $t_0\overset{k_1}{\edge}t_1\overset{k_2}{\edge}\,\cdots\,\overset{k_r}{\edge}t_r=t$.
Proposition~\ref{any B is sort} says further that $v$ can be chosen so that $\bigl(G_{t_p}^{-B_0^T;t_0}\bigr)^T\delta$ has nonnegative entries for all $p=0,\ldots,r$.
Proposition~\ref{sort green} says that $\kk$ is a green sequence for $B_0$.
For such a choice of $v$ and $t$, define $\delta^B$ to be $\bigl(G_t^{-B_0^T;t_0}\bigr)^T\delta$.  The following proposition in particular implies that $\delta^B$ is well defined, in the sense that it depends only on $B$ and not on the specific choices we made.

\begin{proposition}\label{delta is the man}
Suppose $B$ is an exchange matrix of affine type.
\begin{enumerate}[label=\bf\arabic*., ref=\arabic*]
\item \label{im ray pos}
$\delta^B$ has nonnegative entries and is the shortest integer vector in the ray it~spans.  \item \label{im ray}
$-\frac12B\delta^B$ is the shortest integer vector in the imaginary ray of $\F_{B^T}$.
\item \label{im hyp}
All imaginary cones in $\F_{B^T}$ are contained in $(\delta^B)^\perp$.
\end{enumerate}
\end{proposition}

We refer to Remark~\ref{danger: co-roots} for clarification on how to interpret Proposition~\ref{delta is the man}.\ref{im hyp} in specific examples.
See also the last part of the following proof.

\begin{proof}%[Proof of Proposition~\ref{delta is the man}]
  The first claim in Assertion~\ref{im ray pos} is true because we chose $v$ and $t$ so that $\bigl(G_t^{-B_0^T;t_0}\bigr)^T\delta$ has nonnegative entries.
Recall from Section~\ref{def sec} that the inverse of $G_t^{-B^T;t_0}$ is $\bigl(C_t^{B;t_0}\bigr)^T$.
Since $\delta$ is the shortest integer vector in the ray that it spans, and since $\delta^B$ is obtained from $\delta$ by the action of an integer matrix whose inverse is an integer matrix,~$\delta^B$ is also the shortest integer vector in the ray that it spans.

Suppose $a_1\cdots a_r$ is the $c$-sorting word for $v$.
Lemma~\ref{any prefix} says that for any ${p\in\set{0,1,\ldots,r}}$, the prefix $a_1\cdots a_p$ is the $c$-sorting word for another $c$-sortable element $v_p$.
The corresponding postfix of $\kk$ and the seed $t_p$ have the same properties we used to define $\delta^B$, so writing $B_p$ for the exchange matrix at $t_p$, we can use $t_p$ to define $\delta^{B_p}$.
Thus we argue by induction on $r$ for the remaining assertions.

Assertion~\ref{im ray} is true by Lemma~\ref{nu delta} when $r=0$ (so that $t=t_0$ and $B=B_0$) because then $\delta^B=\delta$.
By induction, $-\frac12B_{r-1}\delta^{B_{r-1}}$ is the shortest integer vector in the imaginary ray of $\F_{B_{r-1}^T}$.
The mutation map $\eta_{k_r}^{B_{r-1}^T}$ takes the shortest integer vector in the imaginary ray of $\F_{B_{r-1}^T}$ to the shortest integer vector in the imaginary ray of $\F_{B^T}$. 
Thus we check that $\eta_{k_r}^{B_{r-1}^T}(-\frac12B_{r-1}\delta^{B_{r-1}})=-\frac12B\delta^{B}$.

We first appeal to \cite[Proposition~1.3]{NZ} to say that 
\[\delta^B=\bigl(G_t^{-B_0^T;t_0}\bigr)^T\delta=\bigl(G_{t_{r-1}}^{-B_0^T;t_0}E_{+,k_r}^{B_{r-1}^T}\bigr)^T\delta=F_{-,k_r}^{B_{r-1}}\bigl(G_{t_{r-1}}^{-B_0^T;t_0}\bigr)^T\delta,\]
because $\kk$ is a green sequence, so that column $k_r$ of $C_{t_{r-1}}^{-B_0^T;t_0}$ is positive.
Now Lemma~\ref{EBF trick} says that $B=E_{-,k_r}^{B_{r-1}}B_{r-1}F_{-,k_r}^{B_{r-1}}$.
Thus since $F_{-,k_r}^{B_{r-1}}$ is its own inverse, 
\[-\frac12B\delta^B=-\frac12E_{-,k_r}^{B_{r-1}}B_{r-1}\bigl(G_{t_{r-1}}^{-B_0^T;t_0}\bigr)^T\delta=E_{-,k_r}^{B_{r-1}}(-\frac12B_{r-1}\delta^{B_{r-1}}).\]
Then showing that $\eta_{k_r}^{B_{r-1}^T}(-\frac12B_{r-1}\delta^{B_{r-1}})=E_{-,k_r}^{B_{r-1}}(-\frac12B_{r-1}\delta^{B_{r-1}})$ amounts to showing that the $k_r$-entry of $-\frac12B_{r-1}\delta^{B_{r-1}}$ is nonpositive.

We use Proposition~\ref{BGCB} to rewrite $-\frac12B_{r-1}\delta^{B_{r-1}}=-\frac12B_{r-1}\bigl(G_{t_{r-1}}^{-B_0^T;t_0}\bigr)^T\delta$ as $\bigl(C_{t_{r-1}}^{-B_0^T;t_0}\bigr)^T(-\frac12B_0\delta)$.
Thus we must show that $\omega_c(C^{k_r}_c(v_{r-1})\ck,\delta)\ge0$. 
Since $C^{k_r}_c(v_{r-1})\ck$ is a positive scalar multiple of $C^{k_r}_c(v_{r-1})$, it is equivalent to check that $\omega_c(C^{k_r}_c(v_{r-1}),\delta)\ge0$. 

In the definition of $\delta^B$, we replaced a $c$-sortable element $v$ by the $c$-sortable element~$c^iv$ for large enough $i$.
Since we can increase $i$ arbitrarily without disturbing the properties we have established for $\kk$, we know that $c^jC^{k_r}_c(v_{r-1})=C^{k_r}_c(c^jv_{r-1})$ is a positive root for any $j\ge0$.
Thus Proposition~\ref{who is pos} says that either $C^{k_r}_c(v_{r-1})$ is $c^p\gamma$ for some $\gamma\in\TravProj{c}$ and $p\ge0$ or $C^{k_r}_c(v_{r-1})$ is in a finite $c$-orbit.
Thus Proposition~\ref{om del} or Proposition~\ref{om del fin} says that $\omega_c(C^{k_r}_c(v_{r-1}),\delta)\ge0$. 
We have proved Assertion~\ref{im ray}.

Assertion~\ref{im hyp} is true when $r=0$.
If $r>0$, then by induction, all imaginary cones in $\F_{B_{r-1}^T}$ are contained in $(\delta^{B_{r-1}})^\perp$.
The imaginary cones in $\F_{B^T}$ are the images, under $\eta^{B^T_{r-1}}_{k_r}$, of the imaginary cones in $\F_{B_{r-1}^T}$.
Every imaginary cone in $\F_{B_{r-1}^T}$ contains the imaginary ray, spanned by $-\frac12B_{r-1}\delta^B_{r-1}$.
We have previously shown that the $k_r$ entry of $-\frac12B_{r-1}\delta^B_{r-1}$ is nonpositive.
%More specifically, we showed that it is strictly negative unless $C^{k_r}_c(v_{r-1})$ is in a finite $c$-orbit.
Let $\lambda$ be any vector in an imaginary cone of $\F_{B_{r-1}^T}$.

If the $k_r$ entry of $-\frac12B_{r-1}\delta^B_{r-1}$ is strictly negative, then the $k_r$ entry of $\lambda$ is nonpositive, and thus $\eta^{B^T_{r-1}}_{k_r}$ acts on $\lambda$ (just as it acts on $-\frac12B_{r-1}\delta^B_{r-1}$) by the map given in the basis of fundamental weights by the matrix $E_{-,k_r}^{B_{r-1}}$.
The map taking $\delta^{B_{r-1}}$ to $\delta^B$ is given, in the basis of simple roots, by the matrix $F_{-,k_r}^{B_{r-1}}$.
The fundamental weights are dual to the simple \emph{co-roots} and the diagonal matrix $\symmetrizer$ of symmetrizing constants, applied to the simple-root coordinates of a vector, gives the simple-co-root coordinates of that vector.
Recalling the identity $(E_{\ep,k}^B)^T\symmetrizer=\symmetrizer F_{\ep,k}^B$ from Section~\ref{acyc sec}, we compute
\begin{align*}
\br{\eta^{B^T_{r-1}}_{k_r}(\lambda),\delta^B}
&=\br{E_{-,k_r}^{B_{r-1}}\lambda,F_{-,k_r}^{B_{r-1}}\delta^{B_{r-1}}}\\
&=(E_{-,k_r}^{B_{r-1}}\lambda)^T\symmetrizer F_{-,k_r}^{B_{r-1}}\delta^{B_{r-1}}\\
&=\lambda^T(E_{-,k_r}^{B_{r-1}})^T\symmetrizer F_{-,k_r}^{B_{r-1}}\delta^{B_{r-1}}\\
&=\lambda^T\symmetrizer F_{-,k_r}^{B_{r-1}}F_{-,k_r}^{B_{r-1}}\delta^{B_{r-1}}=\lambda^T\symmetrizer\delta^{B_{r-1}}=\br{\lambda,\delta^{B_{r-1}}}=0.
\end{align*}

If instead the $k_r$ entry of $-\frac12B_{r-1}\delta^B_{r-1}$ is zero then either linear map associated to $\eta^{B^T_{r-1}}_{k_r}$ takes $-\frac12B_{r-1}\delta^B_{r-1}$ to $-\frac12B\delta^B$, namely $E_{\ep,k_r}^{B_{r-1}}$ for $\ep\in\set{+,-}$.
Taking $\ep$ to be the sign of the $k_r$ entry of $\lambda$, we compute as above with the sign $\ep$ replacing the sign~$-$, to obtain $\br{\eta^{B^T_{r-1}}_{k_r}(\lambda),\delta^B}=0$.
\end{proof}

Proposition~\ref{delta is the man} allows us to prove one containment in Theorem~\ref{affine main}.

\begin{proposition}\label{affine main partial}
Suppose $B$ is an exchange matrix of affine type.
If $\lambda$ is contained in the imaginary wall~$\d^B_\infty$, then $\P^B_\lambda\supseteq\set{\lambda+aB\delta^B:a\ge0}\cap\d^B_\infty$.  % \emph{contains} the line segment parallel to the imaginary ray, with one endpoint at $\lambda$ and the other endpoint on the relative boundary of $\d^B_\infty$.
\end{proposition}
\begin{proof}
Suppose $\kk$ is a sequence of indices in $\set{1,\ldots,n}$ and let $B'=\mu_\kk(B)$.
Then the mutation map $\eta^{B^T}_\kk$ takes the imaginary ray in~$\F_{B^T}$ to the imaginary ray in~$\F_{(B')^T}$ and takes the imaginary cone of $\F_{B^T}$ containing $\lambda$ to an imaginary cone in~$\F_{(B')^T}$.
Since $\delta^{B'}$ has nonnegative entries, the ray $\set{\eta^{B^T}_\kk(\lambda)+aB'\delta^{B'}:a\ge0}$ is contained in $\set{\eta^{B^T}_\kk(\lambda)+B'\alpha:\alpha\ge0}$.
Since $(\eta_\kk^{B^T})^{-1}$ is linear on each imaginary cone and takes the direction $-B'\delta^{B'}$ of the imaginary ray in $\F_{(B')^T}$ to the direction $-B\delta^B$ of the imaginary ray in $\F_{B^T}$, we see that 
\begin{multline*}
\P^B_{\lambda,\kk}=(\eta_\kk^{B^T})^{-1}\set{\eta^{B^T}_\kk(\lambda)+B'\alpha:\alpha\ge0}\\
\supseteq(\eta_\kk^{B^T})^{-1}\left(\set{\eta^{B^T}_\kk(\lambda)+aB'\delta^{B'}:a\ge0}\cap\d^{B'}_\infty\right)\\
=\set{\lambda+aB\delta^B:a\ge0}\cap\d^B_\infty. 
\end{multline*}
This is true for all $\kk$, so $\P^B_\lambda\supseteq\set{\lambda+aB\delta^B:a\ge0}\cap\d^B_\infty$, as desired.
\end{proof}

We will also use Proposition~\ref{delta is the man} to prove the following theorem.

\begin{theorem}\label{affine P point indep}
Suppose $\tB$ is an extended exchange matrix with linearly independent columns such that $B$ is of affine type.
  If $\tilde\lambda\in\Cone^{\tB;t_0}_t$ for some $t$, then $\P^\tB_{\tilde\lambda}=\set{\tilde\lambda}$.
\end{theorem}

Theorem~\ref{affine P point indep} follows from Theorem~\ref{P point} and the following proposition.

\begin{proposition}\label{aff red}
If $B$ is an exchange matrix of affine type, then $B$ admits a maximal green sequence and a maximal red sequence.
\end{proposition}
\begin{proof}  
Because $B$ is of affine type if and only if $-B$ is of affine type, it is enough to prove that every $B$ of affine type admits a maximal green sequence.
The imaginary wall $\d^B_\infty$ is a union of convex cones.
By Proposition~\ref{delta is the man}.\ref{im hyp}, it is contained in the hyperplane $(\delta^B)^\perp$, but it is strictly smaller than that hyperplane.
(If the imaginary wall were the whole hyperplane, the $\g$-vector fan would be contained in a halfspace, but in the case where $B$ is acyclic, the $\g$-vector fan is dense in the whole space, and that property is preserved by mutation maps.) 
For that reason, we can construct a maximal green sequence as follows.
For any point $x$ in the interior of the positive cone, we can move the line segment from $x$ to~$-x$ an arbitrarily small amount to obtain a line segment from the positive cone to the negative cone that is contained in the interior of the $\g$-vector fan and doesn't pass through any codimension-$2$ faces of the $\g$-vector fan.
The $\g$-vector cones visited by this define a maximal green sequence.
\end{proof}

\begin{remark}\label{affine P point remark}
  Eventually, putting together Theorems~\ref{finite P point}, \ref{affine P point indep}, and~\ref{affine main extended}, we will have characterized all dominance regions in finite type, all dominance regions in affine type with the condition that $\tB$ has linearly independent columns, and additionally, with no condition on linear independence, all dominance regions $\P_{\tilde\lambda}^\tB$ in affine type such that the projection of $\tilde\lambda$ to its first $n$ coordinates is in the imaginary wall.
The piece that is missing for a characterization of all dominance regions in affine type is the following statement that is likely to be true.

\begin{probable}\label{affine P point} 
Suppose $\tB$ is an $m\times n$ extended exchange matrix such that $B$ is of affine type.
  If $\tilde\lambda\in\Cone^{\tB;t_0}_t$ for some $t$, then $\P^\tB_{\tilde\lambda}=\set{\tilde\lambda}$.
\end{probable}

We have not pursued this result because the results that we do prove are sufficient for our purposes (related to pointed bases and theta functions, as explained in the introduction).
For the same reason, and because we don't know any other uses for the result, we have refrained from giving it the status of a conjecture. 
If a need for this result becomes apparent in the future, we suspect that it can be proved using some of the same ideas as the proof of Theorem~\ref{finite P point}:
starting with the coefficient-free case, using a bipartite initial seed, folding and computations (reduced to a finite computation using symmetry).
However, some variation of the proof would be needed in affine type $\afftype A$, because there exist exchange matrices of affine type $\afftype A$ that are not mutation-equivalent to bipartite exchange matrices.
These cases can presumably be handled in some way, perhaps using another canonical choice of triangulation in the surfaces model.
\end{remark}

\subsection{Some tools}
In this section, we gather some tools that we will use in Section~\ref{neigh sec}.

\subsubsection{Mutation-finite exchange matrices}\label{mut fin sec}
An exchange matrix $B$ is \newword{mutation-finite} if only finitely many different matrices can be obtained from $B$ by arbitrary sequences of mutations.
The following theorem is \cite[Theorem~2.8]{FeShTu}.  

\begin{theorem}\label{mut fin 2x2}
An $n\times n$ exchange matrix $B$ with $n\ge3$ is mutation-finite if and only if, for every sequence $\kk$ of indices in $\set{1,\ldots,n}$, the exchange matrix $B'=\mu_\kk(B)$ satisfies $b'_{ij}b'_{ji}\ge-4$ for all indices $i$ and $j$.
\end{theorem}

This theorem is useful to us because all cluster algebras of affine type are mutation-finite.
In fact, the following stronger statement holds \cite[Theorem~3.5]{Seven} 

\begin{theorem}\label{acyc mut fin}
An acyclic $n\times n$ exchange matrix with $n\ge3$ is mutation-finite if and only if its underlying Cartan matrix is of finite or affine type.
\end{theorem}

We also point out the following obvious and well known fact, which holds because, for any subset $I$ of $\set{1,\ldots,n}$, and any sequence $\kk$ of indices in $I$, if $B'=[b'_{ij}]=\mu_\kk(B)$, then $\mu_\kk([b_{ij}]_{i,j\in I})=[b'_{ij}]_{i,j\in I}$.
%(The analogous fact for 

\begin{proposition}\label{mut fin sub}
Suppose $B$ is an $n\times n$ exchange matrix.
If $B$ is mutation-finite then, for every subset $I$ of $\set{1,\ldots,n}$, the submatrix $[b_{ij}]_{i,j\in I}$ is mutation-finite.
\end{proposition}
%\begin{proof}
%Suppose $B'=[b_{ij}]_{i,j\in I}$ is not mutation-finite.
%Then Theorem~\ref{mut fin 2x2} says that there is some sequence $\kk$ of indices in $I$ such that $\mu_\kk(B')$ has a $2\times2$ submatrix failing the condition of Theorem~\ref{mut fin 2x2}.
%Then the same submatrix occurs in $\mu_\kk(B)$, and thus $B$ is not mutation-finite.
%\end{proof}

\subsubsection{Growth of cluster algebras}\label{growth sec}
The \newword{growth} of a cluster algebra is the asymptotic behavior of the function that counts the number of seeds within a given mutation distance from some initial seed.
Some facts about growth will be useful.

First, beyond rank $2$, the notion of affine type has the following intrinsic characterization.
Combining \cite[Theorem~3.5]{Seven} and \cite[Theorem~1.1]{FeShThTu12}, we see that a cluster algebra of rank at least $3$ is of affine type if and only if it is not of finite type but has linear growth.
%Growth rate of cluster algebras
%Anna Felikson, Michael Shapiro, Hugh Thomas and Pavel Tumarkin

Second, given an $n\times n$ exchange matrix $B=[b_{ij}]$ and a subset $I\subseteq\set{1,\ldots,n}$, if the submatrix $[b_{ij}]_{i,j\in I}$ has exponential growth, it is immediate that that $B$ has exponential growth as well.

Finally, a non-acyclic skew-symmetrizable $3\times3$ exchange matrix with $b_{12}b_{21}=b_{13}b_{31}=b_{23}b_{32}=-4$ has exponential growth.
To see why, one can check that such matrices have the property that every single-step mutation coincides with negation of the matrix, and thereby apply \cite[Theorem~1.1]{FeShThTu12}.

\subsubsection{Exchange relations in the principal coefficients case}\label{exch rel sec}
%We pause to prove a necessary fact about exchange relations, assuming basic background from \cite[Sections~2 \&~3]{ca4}.
%For a moment we work in the most general setting of (normalized) cluster algebras in the sense of \cite[Definition~2.3]{ca4}.
%(See also \cite[Remark~2.7]{ca4}.)
Two cluster variables~$x$ and $x'$ in a cluster pattern of rank $n$  are \newword{exchangeable} if there is no cluster containing both $x$ and $x'$ but there exists a set $\Gamma$ of $n-1$ cluster variables such that $\Gamma\cup\set{x}$ is a cluster and $\Gamma\cup\set{x'}$ is a cluster.
The exchange relation relating $x$ and $x'$ might, in principle, depend on $\Gamma$, but we show that, in the case of principal coefficients, it only depends on $x$ and $x'$.

\begin{lemma}\label{exch ind}  
Suppose $x$ and $x'$ are exchangeable cluster variables in a cluster pattern with principal coefficients.
Then every exchange relation for $x$ and $x'$ is the same:
it involves the same two cluster monomials with the same coefficients.
\end{lemma}

\begin{proof}
Let $\g(x)$ and $\g(x')$ be the $\g$-vectors of $x$ and $x'$.
Sign-coherence of $C$-vectors (explained earlier in Section~\ref{def sec}) implies that every exchange relation for $x$ and $x'$ has a monomial $M$, involving cluster variables but not coefficient variables, whose $\g$-vector is $\g(x)+\g(x')$.
This is a cluster monomial, and it is known that different cluster monomials have different $\g$-vectors (as conjectured in \cite[Conjecture~7.10]{ca4} and proved in the skew-symmetric case in \cite{CIKLFP} and in general as \cite[Theorem~2.13]{GHKK}).
We see that $M$ is the same cluster monomial in every exchange relation for $x$ and~$x'$.
Any exchange relation for $x$ and $x'$ writes $xx'$ as $M$ plus another monomial $N$.
But then $N=xx'-M$, so every exchange relation for $x$ and $x'$ also has the same~$N$.
Now $N$ is a monomial in the coefficient variables times a cluster monomial, and again, it is the same cluster monomial in any exchange relation and the same coefficient monomial.
\end{proof}

\begin{remark}\label{4.3}
It appears that, to extend Lemma~\ref{exch ind} to arbitrary coefficients, one needs \cite[Conjecture~4.3]{ca4}, which says that the exchange graph does not depend on the choice of coefficients.
\end{remark}

\subsubsection{Salient exchange matrices of affine type}
\begin{proposition}\label{affine salient}
If $B$ is an exchange matrix of affine type, then there exists an exchange matrix $B_0$ that is mutation-equivalent to $B$ such that $B_0$ is salient.
% and any extension of $B_0$ is salient.
\end{proposition}
\begin{proof}
Since $B$ is of affine type, it is mutation-equivalent to an exchange matrix whose graph is an orientation of a Dynkin diagram of affine type.
When this graph is an oriented tree, source-sink moves make a bipartite exchange matrix $B_0$ mutation-equivalent to $B$.
Since Cartan matrices of affine type have connected Dynkin diagrams $B_0$ has no zero columns.
%Lemma~\ref{bip sal ext} says that any extension of $B_0$ is salient.

Otherwise, the graph is a cycle, and source-sink moves lead to an exchange matrix $B'_0$ whose oriented graph has exactly one source and exactly one sink.
For convenience, we can reindex $B'_0$ so that its source is $1$, the sink is $n$, and the directed graph for $B'_0$ has ${1\to2\to\cdots\to k\to n}$ and $1\to k+1\to k+2\to\cdots\to n$ for some $k$ with $1\le k<n$.
The matrix $B'_0$ is salient except when $n$ and $k$ are both even, but we perform one mutation on $B'_0$ to obtain a matrix that is salient for all $n$ and~$k$.

Let $B_0$ be obtained from $B_0'$ by mutating at the sink.
Thus the directed graph for $B_0$ has ${1\to2\to\cdots\to k\leftarrow n}$ and $1\to k+1\to k+2\to\cdots\to n-1\leftarrow n$ for some $k$ with $1\le k<n$.
(If $k=1$ or $k=n-1$, then the  arrows are $1\leftarrow n$ and $1\to2\to\cdots\to n-1\leftarrow n$.)

To show that $B_0$ is salient, we wish to find a vector $x\in\reals^n$ whose dot product is strictly positive with each column of $B_0$.
Each column of $B_0$ contains exactly two nonzero entries, so each column constrains $x$ by a strict inequality relating two of its entries.
%Thus the desired vector $x$ exists if and only if these strict inequalities do contradict each other.
The inequalities that appear depend on the parity of $k$ and of $n-k$.
\begin{align*}
k\text{ odd:}\qquad
&x_1>x_3>\cdots>x_k>-x_{n-1}\\
&-x_{k+1}>x_2>x_4>\cdots>x_{k-1}>-x_n\\
k\text{ even:}\qquad
&x_1>x_3>\cdots>x_{k-1}>-x_n\\
&-x_{k+1}>x_2>x_4>\cdots>x_k>-x_{n-1}\\
n-k\text{ odd:}\qquad
&x_1>x_{k+2}>x_{k+4}\cdots>x_{n-1}>-x_k\\
&-x_2>x_{k+1}>x_{k+3}>\cdots>x_{n-2}>-x_n\\
n-k\text{ even:}\qquad
&x_1>x_{k+2}>x_{k+4}\cdots>x_{n-2}>-x_n\\
&-x_2>x_{k+1}>x_{k+3}>\cdots>x_{n-1}>-x_k
\end{align*}
We see that for any combination of the parities of $k$ and $n-k$, the inequalities can all be satisfied, and we conclude that $B_0$ is salient for every choice of $n$ and~$k$.
%
%Since $B_0$ is skew-symmetric, the equations describing its kernel are exactly the same as the inequalities above, but with equality replacing strict inequality everywhere.
%By inspection of those equations, we see that every nonzero element of the kernel of $B_0$ has entries with opposite signs.
%Thus no nonzero element of the kernel has nonnegative entries, so Lemma~\ref{extend salient} says that every extension of $B_0$ is salient.
\end{proof}

%\subsection{Neighboring exchange matrices}
\subsection{Neighboring seeds}\label{neigh sec}
In a cluster pattern of affine type, a seed \newword{neighboring the imaginary wall} (or simply a \newword{neighboring seed}) is a seed that has $n-2$ of its $\g$-vectors contained in the imaginary wall.
It is easy to see that neighboring seeds exist in every cluster pattern of affine type using facts about the affine almost-positive roots model.
Recall that, when $B$ is acyclic of affine type, $\nu_c(\Fan_c(\RS))$ is the mutation fan for $B^T$, and that $\nu_c$ restricts to an isomorphism from the $\dd$-vector fan of $B$ to the $\g$-vector fan.
By \cite[Proposition~5.14]{affdenom}, every maximal cone of the mutation fan having a ray spanned by $\nu_c(\delta)$, has exactly $n-2$ other rays, which are spanned by vectors $\nu_c(\beta)$ for roots $\beta\in\APre{c}$.
Then \cite[Proposition~5.14]{affdenom} further implies that these $n-2$ rays are contained in some full-dimensional cone together with $2$ additional rays, also spanned by vectors $\nu_c(\beta)$ for roots $\beta\in\APre{c}$.
The seed associated to this full-dimensional cone is neighboring.

\begin{lemma}\label{neigh B only}
The property that a seed $t$ is neighboring depends only on the exchange matrix at $t$, not on the specific seed $t$ or on the choice of initial seed $t_0$.
\end{lemma}
\begin{proof}
Suppose $t_0$ and $t$ are seeds in a cluster pattern of affine type, choose $\kk$ such that $t_0\overset{k_1}{\edge}t_1\overset{k_2}{\edge}\,\cdots\,\overset{k_r}{\edge}t_r=t$.
Let $B_0$ be the exchange matrix at $t_0$ and let $B$ be the exchange matrix at $t$.
Then $\eta^{B^T}_\kk$ is an isomorphism from $\F_{B_0^T}$ to $\F_{B^T}$ sending $\Cone_t^{B_0;t_0}$ to $\Cone_t^{B;t}$, which is the positive cone.
Thus $t$ is neighboring if and only if the positive cone in $\F_{B^T}$ shares $n-2$ rays with the imaginary wall in $\F_{B^T}$.
This is a property that depends only on~$B$.
\end{proof}

In this section, we characterize \newword{neighboring exchange matrices}, the exchange matrices of neighboring seeds.

\begin{lemma}\label{neigh neg and T}
The property of an exchange matrix being of affine type and being neighboring is preserved by transpose and preserved by negating the matrix.
\end{lemma}
\begin{proof}
An exchange matrix $B$ of affine type has $-B$ and $B^T$ of affine type because mutation commutes with transpose and negation, because $A$ is of affine type if and only if $A^T$ is, and because $B$ and $-B$ have the same underlying Cartan matrix.

We have already seen that $-B^T$ is a rescaling of $B$, so \cite[Proposition~7.8(3)]{universal} says that $\F_B$ and $\F_{B^T}$ are related by a linear map.
Also, \cite[Proposition~7.1]{universal} says that $\F_B$ and $\F_{-B}$ are related by the antipodal map.
These linear maps preserve the property that the positive cone shares $n-2$ rays with the imaginary wall.
\end{proof}

We call column $i$ of an exchange matrix $B$ a \newword{quasi-leaf} if column $i$ has at most two nonzero entries and, if $b_{ji}$ and $b_{ki}$ are both nonzero, then the restriction of $B$ to rows and columns $i$, $j$, and $k$ is $\pm\begin{bsmallmatrix*}[r]0&1&-1\\-1&0&1\\1&-1&0\end{bsmallmatrix*}$.
In describing block decompositions of matrices, we allow \newword{empty blocks}, meaning blocks with $0$ columns and/or $0$ rows.
Thus, for example, a ``$0\times 2$ block'' or a ``$0\times 0$ block''.

\begin{theorem}\label{neigh B}
Suppose $B$ is an exchange matrix of affine type.
Then the following conditions are equivalent.
\begin{enumerate}[label=\rm(\roman*), ref=(\roman*)]
\item \label{neigh}
$B$ is a neighboring exchange matrix.
\item \label{aff 2}
There exist indices $i$ and $j$ such that $\begin{bsmallmatrix}0&b_{ij}\\b_{ji}&0\end{bsmallmatrix}$ is of affine type.
\item \label{neigh detailed}
Up to relabeling, $B$ is
    $
      \begin{bsmallmatrix}
        B_{11} 	& 0 		& 0 		& B_{14} \\ 
        0 		& B_{22} 	& 0 		& B_{24} \\
        0 		& 0 		& B_{33} 	& B_{34} \\
        B_{41} 	& B_{42} 	& B_{43} 	& B_{44}
      \end{bsmallmatrix},
    $ 
where $B_{44}$ is a rank-$2$ exchange matrix of affine type and the following conditions hold for ${\ell\in\set{1,2,3}}$.
\begin{itemize}
\item
$B_{\ell\ell}$ is either a $0\times0$ block or an exchange matrix of finite type A. % (indexed with $B_{11}$, $B_{22}$, and $B_{33}$ in order of increasing size),
\item
If $B_{\ell\ell}$ is nonempty, then its last column is a quasi-leaf of $B_{\ell\ell}$.
\item
If $B_{\ell4}$ is nonempty, then it is nonzero only in its last row.
\item
If $B_{4\ell}$ is nonempty, then it is nonzero only in its last column.
\item
If $B_{\ell4}$ is nonempty (equivalently if $B_{4\ell}$ is nonempty), then the nonzero rows and columns in $\begin{bsmallmatrix} 0 & B_{\ell4} \\ B_{4\ell} & B_{44} \end{bsmallmatrix}$ form one of the matrices in Table~\ref{submat tab}.
\end{itemize}
The blocks $B_{11}$, $B_{22}$, and $B_{33}$ are in order of increasing size. 
If a submatrix in Table~\ref{submat tab} of type $A_{4}^{(2)}$, $G_{2}^{(1)}$, or $D_{4}^{(3)}$ appears, then $B_{11}$ and $B_{22}$ are empty.
\end{enumerate}
\end{theorem}
	\begin{table}
	\caption{Possible submatrices}
	\label{submat tab}	
%	\begin{center}
	\begin{tabular}{|cc|cc|}
	Type & matrix & Type & matrix \\
	\hline & & & \\[-1ex]
	$A_{2}^{(1)}$ & $\begin{bsmallmatrix*}[r]
	0 & 1 & -1 \\
	-1 & 0 & 2 \\
	1 & -2 & 0
	\end{bsmallmatrix*}$ & & \\[4ex]
	$C_{2}^{(1)}$ & $\begin{bsmallmatrix*}[r]
	0 & 2 & -2 \\
	-1 & 0 & 2 \\
	1 & -2 & 0
	\end{bsmallmatrix*}$ &
	$D_{3}^{(2)}$ & $\begin{bsmallmatrix*}[r]
	0 & 1 & -1 \\
	-2 & 0 & 2 \\
	2 & -2 & 0
	\end{bsmallmatrix*}$ \\[4ex]
	$G_2^{(1)}$ & $\begin{bsmallmatrix*}[r]
	0 & 3 & -3 \\
	-1 & 0 & 2 \\
	1 & -2 & 0
	\end{bsmallmatrix*}$ &
	$D_4^{(3)}$ & $\begin{bsmallmatrix*}[r]
	0 & 1 & -1 \\
	-3 & 0 & 2 \\
	3 & -2 & 0
	\end{bsmallmatrix*}$ \\[4ex]
	$A_{4}^{(2)}$ & $\begin{bsmallmatrix*}[r]
	0 & 1 & -2 \\
	-2 & 0 & 4 \\
	1 & -1 & 0
	\end{bsmallmatrix*}$ &
	$A_{4}^{(2)}$ & $\begin{bsmallmatrix*}[r]
	0 & 2 & -1 \\
	-1 & 0 & 1 \\
	2 & -4 & 0
	\end{bsmallmatrix*}$ 
	\end{tabular}
%	\end{center}	
	\end{table}

\begin{remark}\label{related remark}
The existence, in each mutation class of affine type, of a matrix satisfying Condition~\ref{neigh detailed} 
has been pointed out elsewhere, namely by Felikson and Tumarkin \cite{FeTuCoeff} in the context of a classification of extended exchange matrices of finite mutation type and by Greenberg and Kaufman \cite{GreenbergKaufman} in the context of a uniform description of the cluster modular groups of affine and extended affine types.  
\end{remark}

%We now prepare to prove Theorem~\ref{neigh B} as a series of propositions.  
In what follows, fix an acyclic $n\times n$ exchange matrix $B_0$ of affine type that is mutation-equivalent to $B$ and let $c$ be the corresponding Coxeter element.
It is apparent that condition~\ref{neigh detailed} implies condition~\ref{aff 2} in Theorem~\ref{neigh B}, and the following proposition proves that condition~\ref{aff 2} implies condition~\ref{neigh}.
\begin{proposition}\label{aff 2 block}
If $B$ is of affine type and there are indices $i$ and $j$ such that $\begin{bsmallmatrix}0&b_{ij}\\b_{ji}&0\end{bsmallmatrix}$ is of affine type, then $B$ is a neighboring exchange matrix.
\end{proposition}

\begin{proof}
We use \cite[Corollary~4.18]{afframe}, which applies to the acyclic affine exchange matrix $B_0$.
The notation $\DF_c$ in \cite{afframe} refers to the doubled Cambrian fan, which coincides with the $\g$-vector fan \cite[Corollary~1.3]{afframe}.
As part of \cite[Corollary~4.18]{afframe}, if an $(n-2)$-dimensional cone $F$ of the $\g$-vector fan of $B_0$ is contained in infinitely many maximal cones of the $\g$-vector fan, then $F$ is in the boundary of the support of the $\g$-vector fan.
Thus $F$ is also in the imaginary wall $\d_\infty$, and since $F$ is a cone in the $\g$-vector fan, it is in the boundary of~$\d_\infty$.

Now, suppose some $\begin{bsmallmatrix}0&b_{ij}\\b_{ji}&0\end{bsmallmatrix}$ is of affine type (and in particular of infinite type).
Then, in the $\g$-vector fan of $B$, there are infinitely many maximal cones of the $\g$-vector fan that contain the cone $F$ spanned by the $\g$-vectors $\set{\e_k:k\not\in\set{i,j}}$.
There is a mutation map $\eta$ that is an isomorphism from the $\g$-vector fan for $B$ to the $\g$-vector fan for $B_0$.
The image of the positive cone under $\eta$ is a maximal cone in the $\g$-vector fan for $B_0$ whose associated seed $t$ has exchange matrix $B$.
The image of $F$ under $\eta$ is an $(n-2)$-dimensional cone in the $\g$-vector fan for $B_0$ that is contained in infinitely many maximal cones of the $\g$-vector fan.
Thus the image of $F$ is in the boundary of the imaginary wall.
Since the image of the positive cone contains the image of $F$, we conclude that $t$ is a neighboring seed, so that $B$ is a neighboring exchange matrix.
\end{proof}

It remains to show that condition~\ref{neigh} implies condition~\ref{neigh detailed}.
Assume that ${B=[b_{ij}]}$ is neighboring.
We will establish some notation and then prove the various parts of condition~\ref{neigh detailed} as a series of propositions.

The fact that $B$ is neighboring means there is a seed $t$ with exchange matrix~$B$ such that $n-2$ of the columns of $G^{B_0,t_0}_t$ are in the boundary of $\d_\infty$.
We use Proposition~\ref{any B is sort} to choose a seed $t$ that corresponds to a $c$-sortable element~$v$.
We reindex the rows and columns of $B_0$ so that columns $1,\ldots,n-2$ of $G^{B_0,t_0}_t$ are in the boundary of $\d_\infty$ and columns $n-1,n$ are not.
If necessary, we swap indices $n-1$ and~$n$ so that $b_{(n-1)n}\ge0$.
With this indexing, we write $\xi_1,\ldots,\xi_n$ for the columns of $G^{B_0,t_0}_t$.
The vectors $\set{\xi_i:i\in[1,n-2]}\cup\set{-B_0\delta}$ span an imaginary cone in $\F_{B_0^T}$.

\begin{proposition}\label{B44 aff}
The submatrix of $B$ with rows and columns indexed by $n-1,n$ is a rank-$2$ exchange matrix of affine type.
\end{proposition}
\begin{proof}
Consider the set of full-dimensional cones of $\F_{B_0^T}$ that have rays spanned by the vectors $\xi_1,\ldots,\xi_{n-2}$. 
This set is infinite.
(If the set is finite, the union of these full-dimensional cones intersects the relative interior of an imaginary cone, contradicting the fact that $\F_{-B_0^T}$ is a fan.)
Thus ${b_{(n-1)n}b_{n(n-1)}\le-4}$.
But also, $B$ is mutation-finite, so $b_{(n-1)n}b_{n(n-1)}\ge-4$ by Theorem~\ref{mut fin 2x2}, so ${b_{(n-1)n}b_{n(n-1)}=-4}$.
\end{proof}

There are vectors $\gamma_1,\ldots,\gamma_n\in\APre{c}$ such that $\nu_c(\gamma_i)=\xi_i$ for $i=1,\ldots,n$.
These roots are pairwise $c$-compatible and $\set{\gamma_1,\ldots,\gamma_{n-2}}\cup\set{\delta}$ is an imaginary $c$-cluster.
In particular,  $\gamma_1,\ldots,\gamma_{n-2}$ are all in $\APTre{c}$ and thus in~$\RST{c}$.
Choose integers ${0\leq p_1\leq p_2 \leq p_3}$ such that for each nonzero $p_\ell$ there is a component of $\RST{c}$ of rank $p_\ell+1$.
As a consequence of Proposition~\ref{compatible in tubes}, we see that $\set{\gamma_1,\ldots,\gamma_{n-2}}$ consists of $p_\ell$ roots from the $\ell\th$ component for each $\ell=1,2,3$.
Thus ${p_1+p_2+p_3=n-2}$.
Define $I_1=\set{1,\ldots,p_1}$, $I_2=\set{p_1+1,\ldots,p_1+p_2}$, and $I_3=\set{p_1+p_2+1,\ldots,n-2}$ and further reindex the rows and columns of $B$ so that each subset $\set{\gamma_i:i\in I_\ell}$  consists of roots from the same component of $\RST{c}$, for $\ell=1,2,3$.
Furthermore, if $p_1>0$, then there is a unique root in $\set{\gamma_i:i\in I_1}$ that is a sum of $p_1$ elements of $\SimplesT{c}$, and we give this the index $p_1$.
Similarly, if $p_2>0$, we make $p_1+p_2$ the index of the unique root in $\set{\gamma_i:i\in I_2}$ that is a sum of $p_2$ elements of $\SimplesT{c}$.
In any case, we make $p_1+p_2+p_3=n-2$ the index of the root in $\set{\gamma_i:i\in I_3}$ that is a sum of $p_3$ elements of $\SimplesT{c}$.
After this reindexing, we consider the decomposition of $B$ into blocks $B_{ij}$ given by the composition $(p_1+p_2+p_3+2)$ of $n$.

A \newword{special index} is the index of the last column/row of $B_{\ell\ell}$ (if $B_{\ell\ell}$ is nonempty) for $\ell=1,2,3$.
There are $1$, $2$, or $3$ special indices, namely whichever of $p_1$, $p_1+p_2$ and $p_1+p_2+p_3=n-2$ are nonzero.
The two indices of $B_{44}$ are the \newword{affine indices} and $B_{44}$ is the \newword{affine submatrix} of~$B$.
%The \newword{tube components} of $\set{1,\ldots,n}$ are whichever of the sets $\set{1,\ldots,p_1}$, $\set{p_1+1,\ldots,p_1+p_2}$, and $\set{p_1+p_2+1,\ldots,n-2}$ are nonempty.

\begin{proposition}\label{A and q-l}
  For $\ell=1,2,3$, if $p_\ell>0$, then $B_{\ell\ell}$ is an exchange matrix of finite type $A_{p_\ell}$ and the special index in $B_{\ell\ell}$ is a quasi-leaf.
\end{proposition}
\begin{proof}
Write $\beta_0,\beta_1,\ldots,\beta_{p_\ell}$ for the simple roots of the component of $\RST{c}$ associated to $B_{\ell\ell}$, numbered so that $\beta_0$ is the unique one of these roots that does not occur in the support of $\gamma_i$ for $i\in I_\ell$.
For $1\le q\le r\le p_\ell$, write $\beta_{qr}$ for $\beta_q+\beta_{q+1}+\cdots+\beta_{r}$ so that each $\gamma_i$ is~$\beta_{q_ir_i}$ for some $q_i$ and $r_i$.
If $i$ is the special index, then $\gamma_i$ is~$\beta_{1p_\ell}$.
%Proposition~\ref{compatible in tubes} says that roots
We can represent each $\gamma_i$ pictorially as in the left picture of Figure~\ref{tube fig}, by circling the nodes $q_i$ through~$r_i$ on a cycle with nodes labeled $0,1,\ldots,p_\ell$.

%We write $x_{qr}$ for the principal-coefficients cluster variable with denominator vector $\beta_{qr}$.
%The denominator vector of $x_{qr}$ is $\beta_{qr}$, so the $\g$-vector of $x_{qr}$ is $\nu_c(\beta_{qr})$.
%The map $\nu_c$ is linear on the subspace spanned by roots in finite orbits.
%Proposition~\ref{compatible in tubes} says that two of these cluster variables can be in the same cluster if and only if the corresonding roots are nested or spaced.
%There is a similar characterization of which pairs of these cluster variables are exchangeable \cite[Theorem~7.2]{affdenom}:
%The cluster variables associated to non-special indices in $B_{\ell\ell}$ are exchangeable if and only if the corresponding roots are adjacent or have overlapping support but are not nested.
In the cluster algebra with principal coefficients at $B_0,t_0$ write $x_{qr}$ for the cluster variable with denominator vector~$\beta_{qr}$ and $\g$-vector $\nu_c(\beta_{qr})$.
Proposition~\ref{compatible in tubes} says cluster variables $x_{qr}$ and $x_{q'r'}$ can be in the same cluster if and only if the $\beta_{qr}$ and $\beta_{q'r'}$ are nested or spaced.
%The map $\nu_c$ is linear on the cone spanned by positive roots.
Similarly, \cite[Theorem~7.2]{affdenom} says that cluster variables $x_{qr}$ and $x_{q'r'}$ are exchangeable if and only if $\beta_{qr}$ and $\beta_{q'r'}$ either have disjoint support but are not spaced or have overlapping support but are not nested.
(In interpreting \cite[Theorem~7.2]{affdenom}, the key fact is that $q\ge1$ and $q'\ge 1$, so neither of these roots has $\beta_0$ in its support.)
If two cluster variables $x_{qr}$ and $x_{q'r'}$ are exchangeable, then without loss of generality, $q<q'$, $q<r<r'$, and $q'\le r+1$.
The solid lines in the left picture of Figure~\ref{tube fig} indicate~$\beta_{qr}$ and~$\beta_{q'r'}$ in the case where $q'<r$.
Lemma~\ref{exch ind} implies that the only cluster variables that can occur in the exchange relation for $x_{qr}$ and $x_{q'r'}$ are the cluster variables that are in \emph{every} set $\Gamma$ of $n-1$ cluster variables such that $\Gamma\cup\set{x_{qr}}$ is a cluster and $\Gamma\cup\set{x_{q'r'}}$ is a cluster.
These are the cluster variables $x_{qr'}$, $x_{q'r}$, $x_{q(q'-2)}$, and $x_{(r+2)r'}$, except that the last three only exist if, respectively, $q'\le r$, $q\le q'-2$, and $r'\ge r+2$.
(In what follows, if $x_{q'r}$ does not exist, the notation $x_{q'r}$ will mean $1$, and similarly for $x_{q(q'-2)}$ and $x_{(r+2)r'}$.)
The dashed lines in the left picture of Figure~\ref{tube fig} indicate $\beta_{qr'}$, $\beta_{q'r}$, $\beta_{q(q'-2)}$, and $\beta_{(r+2)r'}$.

\begin{figure}
\scalebox{1}{
\includegraphics{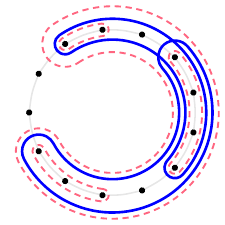}
\begin{picture}(0,0)(54,-54)
%\put(-25,13){\small$q$}
\put(-37,46){\small$q$}
\put(38,-38){\small$r$}
\put(36,37){\small$q'$}
\put(-56,-34){\small$r'$}
\end{picture}}
\qquad\qquad
\scalebox{1}{
\includegraphics{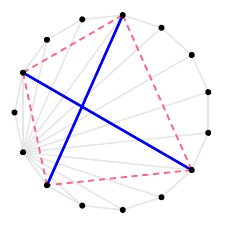}
\begin{picture}(0,0)(54,-54)
\put(-60,-3){\small$0$}
\put(-43,40){\small$q$}
\put(46,-12){\small$r$}
\put(24,44){\small$q'$}
\put(-26,-56){\small$r'$}
\end{picture}}
\caption{Illustrations of the proof of Proposition~\ref{A and q-l}}
\label{tube fig}
\end{figure}
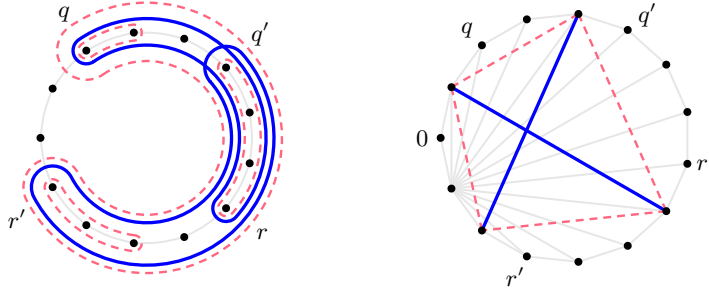

In the proof of Lemma~\ref{exch ind}, we saw that the exchange relation for $x_{qr}$ and $x_{q'r'}$ has a term with no coefficient variables whose $\g$-vector is the sum of the $\g$-vectors of $x_{qr}$ and $x_{q'r'}$.
Since $\nu_c$ is linear on the cone spanned by positive roots, that monomial must be $x_{qr'}x_{q'r}$.
The other term in the exchange relation is therefore $y^\phi x_{q(q'-2)}^ax_{(r+2)r'}^b$ for some $a\ge0$ and $b\ge0$, where $y^\phi$ is the monomial in the coefficient variables with exponent vector $\phi$.
The $\g$-vector of $y^\phi$ is $B_0\phi$.
The $\g$-vectors of all cluster variables in the exchange relation are in $\delta^\perp$ and the exchange relation is homogeneous in the $\g$-vector grading, so $B_0\phi$ is also in $\delta^\perp$.
In other words, $\omega_c(\delta,\phi)=0$, so Proposition~\ref{om del fin} says that $\phi$ is in a finite $c$-orbit.

Write $E_c$ for the $n\times n$ matrix $1-[-B_0]_+$ and $E_{c^{-1}}$ for the $n\times n$ matrix $1-[B_0]_+$.
Then $B_0=E_c-E_{c^{-1}}$.
The piecewise linear map~$\nu_c$, when applied to vectors with nonnegative entries, is defined by the matrix $E_c$ (taking simple-root coordinates to fundamental-weight coordinates).
Furthermore, a result of Howlett \cite[Theorem~2.1]{Howlett} (see also \cite[Theorem~2.6]{affdenom}) says that the matrix for $c$ in the basis of simple roots is~$-E_{c^{-1}}^{-1}E_c$.
Thus the $\g$-vector of $y^\phi$ is 
\[B_0\phi=(E_c-E_{c^{-1}})\phi=(E_c+E_cc^{-1})\phi=\nu_c(1+c^{-1})\phi.\]
Thus because the exchange relation is homogeneous, the vector $(1+c^{-1})\phi$ equals $\beta_{qr}+\beta_{q'r'}-a\beta_{q(q'-2)}-b\beta_{(r+2)r'}$.
We conclude that $a=b=1$ and that $\phi=\beta_{q'(r+1)}$.

We have written the exchange relation between for $x_{qr}$ and $x_{q'r'}$ as 
\begin{equation}\label{exch rel eq}
x_{qr}x_{q'r'}=x_{qr'}x_{q'r}+y^{\beta_{q'(r+1)}}x_{q(q'-2)}x_{(r+2)r'}.
\end{equation}
Up to a global change of sign, these exchange relations determine the entries $b_{ij}$ of~$B_{\ell\ell}$ for non-special~$j$.
Furthermore, appealing to Lemma~\ref{neigh neg and T}, we know that $B^T$ is also neighboring.
In particular, the argument to this point implies that the entries of $B_{\ell\ell}^T$ in non-special columns are $0$, $1$, or $-1$.
We see in particular that $B_{\ell\ell}$ has only entries $0$, $1$, or $-1$ and thus is skew-symmetric. 

On the other hand, consider a $(p_\ell+3)$-gon with vertices labeled $0,1,\ldots,p_\ell+2$ and as initial triangulation $T_0$, take all diagonals with an endpoint at $p_\ell+2$.
The roots~$\gamma_i$ specify a triangulation $T$ that doesn't contain any of these initial arcs.
The root $\gamma_i=\beta_{q_i r_i}$ corresponds to the diagonal from $q_i-1$ to $r_i+1$, i.e.\ the diagonal that crosses the initial diagonals with endpoints $q_i,q_i+1,\ldots,r_i$ and no other initial diagonals.  
The diagonal associated to $\gamma_i$ is contained in a quadrilateral in~$T$.
If the other diagonal is specified by $\gamma_{i'}$, then up to swapping primed and unprimed indices, the exchange relation is exactly what was given above.
The exchange relation is illustrated in the right picture of Figure~\ref{tube fig}, with solid lines representing the cluster variables being exchanged and dashed lines representing the cluster variables appearing in the exchange relation.
We conclude that, up to a global sign, the entries of  $B_{\ell\ell}$ in non-special columns agree with the entries of the signed adjacency matrix of $T$.
Since both matrices are skew-symmetric, they are therefore equal up to a sign.
We thus see that $B_{\ell\ell}$ is of type $A_{p_\ell}$.

In the triangulation $T$, there are two triangles containing the diagonal associated to $\beta_{1 p_\ell}$.
One of these triangles also contains two segments of the boundary (the segments connecting $0$ to $1$ and $0$ to $p_\ell+2$).
The other triangle contains zero, one, or two diagonals (non-boundary segments).
Thus the column associated to the special index has at most two nonzero entries, and if there are two, then those two, together with the special index, induce the submatrix of $B$ that appears in the definition of a quasi-leaf.
\end{proof}

The following proposition is immediate from the fact that the exchange relations \eqref{exch rel eq} only involve cluster variables associated to indices in $B_{\ell\ell}$.

\begin{proposition}\label{mostly zeros}
For  $k\neq\ell\in\set{1,2,3}$, the block $B_{k\ell}$ is zero, except possibly the entry $b_{ij}$ where $i$ is the special index in $I_k$ and $j$ is the special index in $I_\ell$.
For $\ell\in\set{1,2,3}$ and for $i$ the special index in $I_\ell$, the block $B_{\ell4}$ is zero except in row~$i$ and the block $B_{4\ell}$ is zero except in column~$i$.
\end{proposition}
%\begin{proof}
%Since the exchange relations to exchange out the cluster variable with non-special indices in $B_{\ell\ell}$ only involve cluster variables associated to indices in $B_{\ell\ell}$, we see that $b_{ij}=0$ whenever one of $i$ or $j$ is a non-special index in $B_{\ell\ell}$ and the other is not an index in $B_{\ell\ell}$.
%In other words, the blocks $B_{k\ell}$ with $k\neq\ell\in[1,2,3]$ can only have a non-zero entry in the bottom right corner (indexed by the two special indices).
%Similarly blocks $B_{\ell4}$ for $\ell\in[1,2,3]$ can only be non-zero in the last row and blocks $B_{4\ell}$ for $\ell\in\set{1,2,3}$ can only be non-zero in the last column.
%\end{proof}

\begin{proposition}\label{not all zeros}
For $\ell\in\set{1,2,3}$, if $B_{\ell4}$ is nonempty, then it is nonzero, and equivalently $B_{4\ell}$ is nonzero.
\end{proposition}
\begin{proof}
Let $k$ be the special index in $I_\ell$.
Continuing notation from the proof of Proposition~\ref{A and q-l}, $\xi_k$ is not exchangeable with any other cluster variable whose $\g$-vector is in the imaginary wall $\d^B_\infty$.
(This can be seen from the characterization of exchangeability in \cite[Theorem~7.2]{affdenom}.
In the language of that theorem, if $\gamma_k$ is $c$-exchangeable with any other root in $\APTre{c}$, then that root has $\beta_0$ in its support.
Since $\SuppT(\gamma_k)$ is $\set{\beta_1,\ldots,\beta_{p_\ell}}$, the union of the supports of the two roots is $\set{\beta_0,\ldots,\beta_{p_\ell}}$, so the two roots are not $c$-real-exchangeable.)
Therefore the exchange relation for $x_{1p_\ell}$ exchanges it with a cluster variable whose $\g$-vector is not in $\d^B_\infty$.  
In particular, the right side of that exchange relation can't involve only cluster variables indexed by non-affine indices, and therefore $B_{\ell4}$ has at least one nonzero entry in column $k$.
\end{proof}

\begin{proposition}\label{3x3 submat}
For $\ell\in\set{1,2,3}$, if $B_{\ell4}$ is nonempty, then the nonzero rows and columns in $\begin{bsmallmatrix} 0 & B_{\ell4} \\ B_{4\ell} & B_{44} \end{bsmallmatrix}$ form one of the matrices in Table~\ref{submat tab}.
\end{proposition}
\begin{proof}
Let $k$ be the special index in $I_\ell$.
To make the notation more compact, we use $p$ to stand for $n-1$.
The submatrix with rows and columns indexed by $k,p,n$ has a pair of opposite off-diagonal entries whose product is $-4$, so its underlying Cartan matrix is not of finite or affine type.
(There are no $3\times3$ Cartan matrices of finite or affine type with opposite off-diagonal entries whose product is $4$.)
Thus Theorem~\ref{acyc mut fin} says that the submatrix is not acyclic.
In other words, $b_{kp}>0$ and $b_{kn}<0$, and therefore $b_{pk}<0$ and $b_{nk}>0$.
By Theorem~\ref{mut fin 2x2}, each of these entries has absolute value $1$, $2$, $3$, or~$4$.

We first deal with the case where the affine submatrix has entries $\pm2$.
Below, we display some mutations of the submatrix.
%We will derive some inequalities from the fact that each of the mutations must also be non-acyclic and have opposite off-diagonal entries with products at least $-4$.
\begin{multline*}
\begin{bsmallmatrix*}
	0 & b_{kp} & b_{kn} \\
	b_{pk} & 0 & 2 \\
	b_{nk} & -2 & 0
\end{bsmallmatrix*}
\overset{p}{\longrightarrow}
\begin{bsmallmatrix*}
	0 & -b_{kp} & 2b_{kp}+b_{kn} \\
	-b_{pk} & 0 & -2 \\
	2b_{pk}+b_{nk}& 2 & 0
\end{bsmallmatrix*}
\overset{n}{\longrightarrow}
\begin{bsmallmatrix*}
	0 & 3b_{kp}+2b_{kn} & -2b_{kp}-b_{kn} \\
	3b_{pk}+2b_{nk} & 0 & 2 \\
	-2b_{pk}-b_{nk} & -2 & 0
\end{bsmallmatrix*}\\
\overset{k}{\longrightarrow}
\begin{bsmallmatrix*}
	0 & -3b_{kp}-2b_{kn} & 2b_{kp}+b_{kn} \\
	-3b_{pk}-2b_{nk} & 0 & 2+(2b_{kp}+b_{kn})(3b_{pk}+2b_{nk}) \\
	2b_{pk}+b_{nk} & -2-(2b_{pk}+b_{nk})(3b_{kp}+2b_{kn})& 0
\end{bsmallmatrix*}
\end{multline*}
\begin{multline*}
\begin{bsmallmatrix*}
	0 & b_{kp} & b_{kn} \\
	b_{pk} & 0 & 2 \\
	b_{nk} & -2 & 0
\end{bsmallmatrix*}\overset{n}{\longrightarrow}
\begin{bsmallmatrix*}
	0 & b_{kp}+2b_{kn} & -b_{kn} \\
	b_{pk}+2b_{nk} & 0 & -2 \\
	-b_{nk} & 2 & 0
\end{bsmallmatrix*}\overset{p}{\longrightarrow}
\begin{bsmallmatrix*}
	0 & -b_{kp}-2b_{kn} & 2b_{kp}+3b_{kn} \\
	-b_{pk}-2b_{nk} & 0 & 2 \\
	2b_{pk}+3b_{nk} & -2 & 0
\end{bsmallmatrix*}
\end{multline*}
The fact that each of these matrices is non-acyclic determines the sign of each entry.
Two of these inequalities are shown here, with an inequality that follows:
\[
3b_{kp}+2b_{kn}>0,\quad 2b_{kp}+3b_{kn}<0\quad \implies\quad b_{kn}>-\frac32b_{kp}>\frac94b_{kn}\\
\]
This rules out all pairs $(b_{kp},b_{kn})$ except $(1,-1)$, $(2,-2)$, $(3,-3)$, $(3,-4)$, and $(4,-4)$.
If $b_{kn}=-4$, then $b_{nk}=1$, and also $b_{kp}\ge3$ so also $b_{pk}=-1$.
Skew-symmetrizability of $B$ implies that $\frac{b_{pk}}{b_{kp}}=\frac{b_{nk}}{b_{kn}}$, which rules out the possibility that $(b_{kp},b_{kn})=(3,-4)$.
Similar considerations starting with the inequalities $3b_{pk}+2b_{nk}<0$ and $2b_{pk}+3b_{nk}>0$ rule out all pairs $(b_{pk},b_{nk})$ except $(-1,1)$, $(-2,2)$, $(-3,3)$, and $(-4,4)$.

We see that, in every case, $b_{kp}=-b_{kn}>0$ and $b_{pk}=-b_{nk}<0$.
To verify that all possibilities are listed in Table~\ref{submat tab}, it remains to rule out the cases where $(b_{kp},b_{kn},b_{pk},b_{nk})$ are $(1,-1,-4,4)$, $(2,-2,-2,2)$, or $(4,-4,-1,1)$.
None of these three cases are possible because, as explained in Section~\ref{growth sec}, in these cases, the submatrix would  have exponential growth and thus contradict the fact that $B$ has linear growth.

Next, we deal with the case where the affine submatrix has entries $-1$ and $4$.
\begin{multline*}
\begin{bsmallmatrix*}
	0 & b_{kp} & b_{kn} \\
	b_{pk} & 0 & 4	 \\
	b_{nk} & -1 & 0
\end{bsmallmatrix*}
\overset{p}{\longrightarrow}
\begin{bsmallmatrix*}
	0 & -b_{kp} & 4b_{kp}+b_{kn} \\
	-b_{pk} & 0 & -4 \\
	b_{pk}+b_{nk}& 1 & 0
\end{bsmallmatrix*}
\overset{n}{\longrightarrow}
\begin{bsmallmatrix*}
	0 & 3b_{kp}+b_{kn}& -4b_{kp}-b_{kn} \\
	3b_{pk}+4b_{nk} & 0 & 4 \\
	-b_{pk}-b_{nk}& -1 & 0
\end{bsmallmatrix*}\\
\overset{k}{\longrightarrow}
\begin{bsmallmatrix*}
	0 & -3b_{kp}-b_{kn}& 4b_{kp}+b_{kn} \\
	-3b_{pk}-4b_{nk} & 0 & 4+(3b_{pk}+4b_{nk})(4b_{kp}+b_{kn}) \\
	b_{pk}+b_{nk}& -1-(b_{pk}+b_{nk})(3b_{kp}+b_{kn}) & 0
\end{bsmallmatrix*}
\end{multline*}
\begin{multline*}
\begin{bsmallmatrix*}
	0 & b_{kp} & b_{kn} \\
	b_{pk} & 0 & 4	 \\
	b_{nk} & -1 & 0
\end{bsmallmatrix*}\overset{n}{\longrightarrow}
\begin{bsmallmatrix*}
	0 & b_{kp}+b_{kn} & -b_{kn} \\
	b_{pk}+4b_{nk} & 0 & -4	 \\
	-b_{nk} & 1 & 0
\end{bsmallmatrix*}\overset{p}{\longrightarrow}
\begin{bsmallmatrix*}
	0 & -b_{kp}-b_{kn} & 4b_{kp}+3b_{kn} \\
	-b_{pk}-4b_{nk} & 0 & 4	 \\
	b_{pk}+3b_{nk} & -1 & 0
\end{bsmallmatrix*}
\end{multline*}
These matrices determine inequalities, including
\[
3b_{kp}+b_{kn}>0,\quad 4b_{kp}+3b_{kn}<0\quad \implies\quad b_{kn}>-3b_{kp}>\frac94b_{kn},\\
\]
which rules out all pairs $(b_{kp},b_{kn})$ except $(1,-2)$, $(2,-3)$, and $(2,-4)$.
If $(b_{kp},b_{kn})=(2,-3)$, then $b_{pk}\in\set{-1,-2}$ and $b_{nk}=1$.
Skew-symmetrizability of $B$ implies that $\frac{b_{pk}}{b_{kp}}=4\frac{b_{nk}}{b_{kn}}$, which fails for either choice of $b_{pk}$, thus ruling out the possibility that $(b_{kp},b_{kn})=(2,-3)$.
Similar considerations starting with $b_{pk}+3b_{nk}>0$ and $3b_{pk}+4b_{nk}<0$ rule out all pairs $(b_{pk},b_{nk})$ except $(-2,1)$ and $(-4,2)$.
Thus $(b_{kp},b_{kn},b_{pk},b_{nk})$ is $(1,-2,-2,1)$, $(1,-2,-4,2)$, or $(2,-4,-2,1)$, but we rule out the second and third possibilities because they would imply exponential growth as explained in Section~\ref{growth sec}.

The last case, where the affine submatrix has entries $-4$ and $1$, is related to the previous case by passing to the negative transpose, so the same calculations lead to the desired result.
\end{proof}

\begin{proposition}\label{one component}
If the nonzero rows and columns in $\begin{bsmallmatrix} 0 & B_{\ell4} \\ B_{4\ell} & B_{44} \end{bsmallmatrix}$ are of type $A_{4}^{(2)}$, $G_{2}^{(1)}$, or $D_{4}^{(3)}$, then $\ell=3$ and $B_{11}$ and $B_{22}$ are empty.
\end{proposition}
\begin{proof}
If the submatrix is of type $A_{4}^{(2)}$, then the underlying Cartan matrix defines a root system with 3 root lengths, so $\RS$ is of type $A_{2l}^{(2)}$ in the notation of \cite[Chapter~4]{Kac}.
In this case, the root system can be rescaled to be of type $C^{(1)}$, and therefore $\RST{c}$ has only one component (see \cite[Table 1]{affdenom}).
If the submatrix is of type $G_{2}^{(1)}$ or $D_{4}^{(3)}$, then there are two root lengths, related by a factor of $\sqrt{3}$, so $\RS$ is of type $G_{2}^{(1)}$ or~$D_{4}^{(3)}$.
Again $\RST{c}$ has one component (and indeed the submatrix is all of~$B$).
\end{proof}

\begin{proposition}\label{really zero}
The matrices $B_{12}$, $B_{13}$, $B_{23}$, $B_{21}$, $B_{31}$, and $B_{32}$ are zero.
\end{proposition}
\begin{proof}
There is nothing to prove unless the affine submatrix is $\begin{bsmallmatrix*}[r]0&2\\-2&0\end{bsmallmatrix*}$.
(If not, Proposition~\ref{one component} says that $B_{12}$, $B_{13}$, $B_{23}$, $B_{21}$, $B_{31}$, and $B_{32}$ are all empty.)

By Proposition~\ref{mostly zeros}, it only remains to show that, for any two distinct special indices $j$ and $k$, the entry $b_{jk}$ is zero.
Suppose to the contrary that distinct special indices $j$ and $k$ have $b_{jk}\neq0$.
We may as well choose $j$ and $k$ such that $b_{jk}>0$.
Again writing $p$ for $n-1$, we consider the $4\times4$ submatrix of $B$ with rows and columns indexed by $j,k,p,n$.
We have $b_{jp}=-b_{jn}>0$, $b_{kp}=-b_{kn}>0$, $b_{pj}=-b_{nj}<0$, and $b_{pk}=-b_{nk}<0$.
%WON'T BE NECESSARY:
%Also, skew-symmetrizability of $B$ implies that $b_{kj}=-\frac{b_{nj}b_{kn}}{b_{jn}b_{nk}}b_{jk}<0$.
We compute the mutation in direction $j$.
\begin{align*}
\begin{bsmallmatrix*}
0&b_{jk}&-b_{jn}&b_{jn}\\
b_{kj}&0&-b_{kn}&b_{kn}\\
-b_{nj}&-b_{nk}&0&2\\
b_{nj}&b_{nk}&-2&0
\end{bsmallmatrix*}
&\overset{j}{\longrightarrow}
\begin{bsmallmatrix*}
0&-b_{jk}&b_{jn}&-b_{jn}\\
-b_{kj}&0&-b_{kn}&b_{kn}-b_{kj}b_{jn}\\
b_{nj}&-b_{nk}&0&2+b_{jn}b_{nj}\\
-b_{nj}&b_{nk}+b_{nj}b_{jk}&-2-b_{jn}b_{nj}&0
\end{bsmallmatrix*}
\end{align*}
Since $b_{jk}$, $b_{nj}$, and $b_{nk}$ are all strictly positive, $b_{nk}+b_{nj}b_{jk}\ge2$.
Similarly, $b_{kn}$, $b_{kj}$, and $b_{jn}$ are strictly negative, so $b_{kn}-b_{kj}b_{jn}\le-2$.
Both of these inequalities must be equality, and we see that $b_{jk}=b_{nj}=b_{nk}=1$ and $b_{kn}=b_{kj}=b_{jn}=-1$.
We have determined all entries of the $4\times 4$ submatrix.
One easily verifies that the submatrix has the property that every mutation agrees with some permutation of the row/column indices.  
Therefore, we see that this submatrix is not on the list exchange matrices of sub-exponential growth in \cite[Theorem~1.1]{FeShThTu12}.
Therefore $B$ has exponential growth, and by this contradiction, we conclude that $b_{jk}=0$.
\end{proof}

We have proved Theorem~\ref{neigh B}.
We conclude this section with some useful facts that can be proved using the theorem.

\begin{lemma}\label{affine mut}
Suppose $B$ is a neighboring exchange matrix and let $\kk$ be a $2$-element sequence consisting of the two affine indices. 
Then $\mu_\kk(B)=B$.
\end{lemma}
\begin{proof}
Theorem~\ref{neigh B} implies that most entries of $B$ are fixed under mutation in an affine position.
The only entries that might not be fixed are entries $b_{ij}$ where one of $i$ or $j$ is affine and the other is affine or special.
Thus we only need to check the submatrices shown in Table~\ref{submat tab}.
In each case, we check that mutation in one of the affine entries has the effect of negating the submatrix.
An additional mutation in the other affine index still fixes all entries of $B$ except those described above and again negates the submatrix, so that $\mu_\kk(B)=B$.  
\end{proof}

\begin{proposition}\label{neigh good stuff}
Suppose $B$ is a neighboring exchange matrix.
\begin{enumerate}[label=\bf\arabic*., ref=\arabic*]
\item\label{neigh delta}
The vector $\delta^B$ is zero in all non-affine indices.
Its affine entries are
\begin{itemize}
\item
$1,1$ if the affine submatrix of $B$ is $\begin{bsmallmatrix*}[r]0&2\\-2&0\end{bsmallmatrix*}$,
\item
$2,1$ if the affine submatrix of $B$ is $\begin{bsmallmatrix*}[r]0&4\\-1&0\end{bsmallmatrix*}$, or
\item
$1,2$ if the affine submatrix of $B$ is $\begin{bsmallmatrix*}[r]0&1\\-4&0\end{bsmallmatrix*}$.
\end{itemize}
%If the affine submatrix of $B$ is $\begin{bsmallmatrix*}[r]0&2\\-2&0\end{bsmallmatrix*}$, then $\delta^B$ has entries $1,1$ in the affine indices. 
%If the affine submatrix is $\begin{bsmallmatrix*}[r]0&4\\-1&0\end{bsmallmatrix*}$ or $\begin{bsmallmatrix*}[r]0&1\\-4&0\end{bsmallmatrix*}$, then $\delta^B$ has entries $2,1$ or $1,2$ respectively in the affine indices.
\item\label{neigh im ray}
The vector $-\frac12B\delta^B$ (the shortest integer vector that spans the imaginary ray) is zero in all non-affine entries.
Its affine entries are
\begin{itemize}
\item
$-1,1$ if the affine submatrix of $B$ is $\begin{bsmallmatrix*}[r]0&2\\-2&0\end{bsmallmatrix*}$,
\item
$-2,1$ if the affine submatrix of $B$ is $\begin{bsmallmatrix*}[r]0&4\\-1&0\end{bsmallmatrix*}$, or
\item
$-1,2$ if the affine submatrix of $B$ is $\begin{bsmallmatrix*}[r]0&1\\-4&0\end{bsmallmatrix*}$.
\end{itemize}
\end{enumerate}
\end{proposition}
\begin{proof}
We state these assertions about $\delta^B$ and $-\frac12B\delta^B$ together because the proofs of the assertions are intertwined.

Considering a cluster pattern with $B$ as the initial exchange matrix, the initial $\g$-vectors for non-affine indices are contained in the imaginary wall~$\d^B_\infty$.
Thus Proposition~\ref{delta is the man}.\ref{im hyp} implies that $\delta^B$ is zero in all non-affine positions.
  (More carefully, in light of~Remark~\ref{danger: co-roots}, the simple-co-root coordinates of $\delta^B$ are zero in non-affine positions, and therefore the same is true of the simple-root coordinates that we take as the ``entries'' of $\delta^B$.) 
%NO. Proposition~\ref{delta is the man}.\ref{im ray} DOES NOT give this minimality.
%Furthermore, Proposition~\ref{delta is the man}.\ref{im ray} says that $\delta^B$ has nonnegative entries and is minimal among integer vectors obtained from it by positive integer scaling.

Before determining the entries of $\delta^B$ in the two affine positions, we determine some entries of $-\frac12B\delta^B$.
Theorem~\ref{neigh B} and the fact that $\delta^B$ is zero in all non-affine positions imply that $-\frac12B\delta^B$ is zero in non-special non-affine positions.
Since~$-\frac12B\delta^B$ is the shortest integer vector in the imaginary ray, Lemma~\ref{affine mut} implies that it is fixed under $\eta^{B^T}_\kk$, where $\kk$ is a $2$-element sequence consisting of the two affine indices. 
We use the fact that $-\frac12B\delta^B$ is fixed under $\eta^{B^T}_\kk$ to determine its affine entries $r_{n-1}$, and $r_n$.
%We need only consider the submatrices shown in Table~\ref{submat tab}.
%The mutations are computed below, where the entries $r_k$ (for special~$k$), $r_{n-1}$, and $r_n$ denote entries in $-\frac12B\delta^B$.
We again use $p$ to stand for $n-1$.
We use the notation 
\[m(a,b)=\begin{cases}
ab&\text{if }\sgn(a)=\sgn(b)=+\\
-ab&\text{if }\sgn(a)=\sgn(b)=-\\
0&\text{otherwise}.
\end{cases}\]
To determine the affine entries, we need only the $2\times2$ affine submatrix of $B$.
In the most common case from Table~\ref{submat tab}, the mutations are as follows:
\[
\begin{bsmallmatrix*}
0 & 2 & r_p\\
-2 & 0 & r_n
\end{bsmallmatrix*}\overset{p}{\longrightarrow}
\begin{bsmallmatrix*}
0 & -2 & -r_p\\
2 & 0 & r_n+m(r_p,-2)
\end{bsmallmatrix*}\overset{n}{\longrightarrow}
\begin{bsmallmatrix*}
0 & 2 & -r_p+m(r_n+m(r_p,-2),-2)\\
-2 & 0 & -r_n-m(r_p,-2)
\end{bsmallmatrix*}
\]
We see that $r_n=-r_n-m(r_p,-2)$ and $r_p=-r_p+m(r_n+m(r_p,-2),-2)$, which can be rewritten $r_p=-r_p+m(-r_n,-2)$.
If $r_p$ is positive, then $r_n$ is zero, and thus also $r_p$ is zero.
By this contradiction, we see that $r_p$ is nonpositive.
Similarly, we see that $r_n$ is nonnegative.
%If $r_n$ is negative, then $r_p$ is zero, and thus also $r_n$ is zero, and we see that $r_n$ is nonnegative.
Thus both equations say $r_n=-r_p$.

Let $d_p$ and $d_n$ be the entries of $\delta^B$ in its affine positions.
Since $\delta^B$ is zero in its non-affine positions and since $r_p$ and $r_n$ are the affine entries of $-\frac12B\delta^B$, we have $r_p=-d_n$ and $r_n=d_p$.
Therefore $d_p=d_n$.
From there, we see that the special entries of $-\frac12B\delta^B$ are also zero, and since $-\frac12B\delta^B$ is the shortest integer vector in the ray it spans (Proposition~\ref{delta is the man}.\ref{im ray}), $r_n=1=-r_p$.
Therefore also $d_p=d_n=1$.

In the next case from Table~\ref{submat tab}, the mutations are
\[\begin{bsmallmatrix*}
0 & 4 & r_p\\
-1 & 0 &r_n
\end{bsmallmatrix*}\overset{p}{\longrightarrow}
\begin{bsmallmatrix*}
0 & -4 & -r_p\\
1 & 0 &r_n+m(r_p,-1)
\end{bsmallmatrix*}\overset{n}{\longrightarrow}
\begin{bsmallmatrix*}
0 & 4 &-r_p+m(r_n+m(r_p,-1),-4)\\
-1 & 0 &-r_n-m(r_p,-1)
\end{bsmallmatrix*}
\]
%We see that $r_n=-r_n-m(r_p,-1)$ and $r_p=-r_p+m(r_n+m(r_p,-1),-4)$, which can be rewritten $r_p=-r_p+m(-r_n,-4)$.
%If $r_p$ is positive, then $r_n$ is zero, and thus also $r_p$ is zero.
%By this contradiction, we see that $r_p$ is nonpositive.
%Similarly, we see that $r_n$ is nonnegative.
%If $r_n$ is negative, then $r_p$ is zero, and thus also $r_n$ is zero, and we see that $r_n$ is nonnegative.
%One equation says $2r_n=-r_p$ and the other says $2r_p=-4r_n$.
Arguing as before, we see that $r_n$ is nonnegative and $2r_n=-r_p$.
Again taking $d_p$ and $d_n$ to be the affine entries of $\delta^B$, we also find that $d_p=2r_n=-r_p=2d_n$, so that the special entries of $-\frac12B\delta^B$ are zero, and thus $r_p=-2$ and $r_n=1$.
Therefore,
$d_p=2$ and $d_n=1$.
%Since $\delta^B$ is zero in its non-affine positions and since $r_p$ and $r_n$ are the affine entries of $-\frac12B\delta^B$, we compute that $r_p=-2d_n$ and $r_n=\frac12d_p$.
%Therefore $d_p=2r_n=-r_p=2d_n$,...

In the final case, the mutations are
\[\begin{bsmallmatrix*}
0 & 1 & r_p\\
-4 & 0 &r_n
\end{bsmallmatrix*}\overset{p}{\longrightarrow} 
\begin{bsmallmatrix*}
0 & -1 & -r_p\\
4 & 0 &r_n+m(r_p,-4)
\end{bsmallmatrix*}\overset{n}{\longrightarrow} 
\begin{bsmallmatrix*}
0 & 1 &-r_p+m(r_n+m(r_p,-4),-1)\\
-4 & 0 &-r_n-m(r_p,-4)
\end{bsmallmatrix*}
\]
We see that $r_n=-2r_p\ge0$ and $2d_p=r_n=-2r_p=d_n$, so that special entries of~$-\frac12B\delta^B$ are zero, and thus $r_p=-1$, $r_n=2$, $d_p=1$, and $d_n=2$.
%We see that $r_n=-r_n-m(r_p,-4)$ and $r_p=-r_p+m(r_n+m(r_p,-4),-1)$, which can be rewritten $r_p=-r_p+m(-r_n,-1)$.
%If $r_p$ is positive, then $r_n$ is zero, and thus also $r_p$ is zero.
%By this contradiction, we see that $r_p$ is nonpositive.
%Similarly, we see that $r_n$ is nonnegative.
%If $r_n$ is negative, then $r_p$ is zero, and thus also $r_n$ is zero, and we see that $r_n$ is nonnegative.
%One equation says $2r_n=-4r_p$ and the other says $2r_p=-r_n$.
%Since $\delta^B$ is zero in its non-affine positions and since $r_p$ and $r_n$ are the affine entries of $-\frac12B\delta^B$, we compute that $r_p=-\frac12d_n$ and $r_n=2d_p$.
%Therefore $2d_p=r_n=-2r_p=d_n$, ...
\end{proof}

\subsection{Type-C companions of neighboring exchange matrices}\label{type C}
To each neighboring exchange matrix $B$, we now associate an $(n-2)\times(n-2)$ exchange matrix called the \newword{type-C companion} of $B$ and written $\Comp(B)$.
The type-C companion $\Comp(B)$ agrees with the \emph{non-affine} rows and columns of $B$, except that all entries in \emph{special columns} are multiplied by~$2$.
The type-C companion $\Comp(B)$ has a diagonal block decomposition with $1$, $2$, or $3$ diagonal blocks, having the same sizes as the blocks $B_{11}$, $B_{22}$, and/or $B_{33}$ of~$B$.
The following proposition is immediate from Theorem~\ref{neigh B} and \cite[Proposition~3.3]{NS14} and justifies the terminology.

\begin{proposition}\label{Comp is C}
Given a neighboring exchange matrix $B$, each diagonal block of $\Comp(B)$ is of finite type C.
\end{proposition}
%\begin{proof}
%This should just be Salvatore and Tomoki's characterizations of type C exchange matrices, together with the quasi-leaf part of Theorem~\ref{neigh B}\ref{neigh detailed}.
%
%\end{proof}

Let $\CompPlus(C)$ be the $n\times(n-2)$-matrix which agrees with the non-affine \emph{columns} of $B$, except that all entries in \emph{special columns} are multiplied by~$2$.
Although $\CompPlus(B)$ is an extension of $\Comp(C)$ in the sense of ``extended exchange matrices'', we don't want to think of it that way (and have chosen an overline in the notation rather than a tilde) because we will not always mutate it like an extended exchange matrix.
(See Proposition~\ref{special mut}, below.)

%\begin{proposition}\label{Comp span}
%Suppose $B$ is a neighboring exchange matrix.
%Then the intersection of the hyperplane $(\delta^B)^\perp$ with the nonnegative linear span of the columns of $B$ is contained in the nonnegative linear span of $\frac12B\delta^B$ and the columns of $\Comp(B)$.
%\end{proposition}
%\begin{proof}
%By inspection of Table~\ref{submat tab} and in light of Proposition~\ref{neigh good stuff}, we observe the following facts:
%First, the vector $-\frac12B\delta^B$ spanning the affine ray is a \emph{negative} linear combination of the last two columns of $B$.
%Second, the linear span of the last two columns of $B$ intersected with the hyperplane $(\delta^B)^\perp$ is the line spanned by $\frac12B\delta^B$.
%Third, $\frac12B\delta^B$ has nonzero entries only in the affine indices, and those entries are $1,-1$ unless the affine submatrix is $\begin{bsmallmatrix*}[r]0&4\\-1&0\end{bsmallmatrix*}$ or $\begin{bsmallmatrix*}[r]0&1\\-4&0\end{bsmallmatrix*}$, in which case, those entries are $2,-1$ or $1,-2$ respectively.
%Fourth, each non-affine column of $B$ is half the corresponding column of $\Comp(B)$ plus a nonnegative multiple of $-\frac12B\delta^B$. 
%\saySS{I am confused here: why half? non-affine non-special columns are unchanged, no?}
%(This multiple is zero unless the column is a special column.)
%The proposition follows.
%\end{proof}

\begin{proposition}\label{Comp span}
Suppose $B$ is a neighboring exchange matrix.
Then the intersection of the hyperplane $(\delta^B)^\perp$ with the nonnegative linear span of the columns of $B$ is contained in the nonnegative linear span of $\frac12B\delta^B$ and the columns of $\CompPlus(B)$.
\end{proposition}
\begin{proof}
Suppose $\alpha$ has nonnegative entries and $B\alpha$ is in $(\delta^B)^\perp$.
Write $\alpha$ as $\beta+\gamma$ such that $\beta$ is zero in its affine entries and $\gamma$ is zero in its non-affine entries.
Then $B\beta$ is a nonnegative linear combination of the non-affine columns of $B$, and equivalently a nonnegative linear combination of the columns of $\CompPlus(B)$.
Also $B\beta$ is in $(\delta^B)^\perp$ because every non-affine column of $B$ is.
By Proposition~\ref{neigh good stuff}, since $B\gamma\in(\delta^B)^\perp$, the vector $\gamma$ is a scalar multiple of $\delta^B$.
Thus $B\gamma$ is in the nonnegative span of $\frac12B\delta^B$.
\end{proof}

To avoid encumbering our notation with projection and inclusion maps, throughout this section we identify vectors in $\reals^n$ whose affine entries are zero with vectors in $\reals^{n-2}$.
In particular, we freely apply mutation maps $\eta^{\Comp(B)^T}$ to such vectors.

\begin{proposition}\label{nonspecial mut}
Suppose $B$ is a neighboring exchange matrix and $k$ is a non-special, non-affine index of~$B$.
\begin{enumerate}[label=\bf\arabic*., ref=\arabic*]
\item
$\mu_k(B)$ is neighboring.
\item
$\Comp(\mu_k(B))=\mu_k(\Comp(B))$.
\item
$\CompPlus(\mu_k(B))=\mu_k(\CompPlus(B))$, and these agree with $\CompPlus(B)$ in the affine rows.
\item
The mutation map $\eta^{B^T}_k$ fixes the imaginary ray pointwise.
\item
If $x$ is a vector whose affine entries are zero, then the affine entries of $\eta^{B^T}_k(x)$ are again zero and $\eta^{B^T}_k(x)=\eta^{\Comp(B)^T}_k(x)$.
\end{enumerate}
\end{proposition}
%Version of the prop before NR made changes 26 Nov 2025
%\begin{proposition}\label{nonspecial mut}
%Suppose $B$ is a neighboring exchange matrix and $k$ is a non-special, non-affine index of~$B$.
%Then the following hold:
%\begin{itemize}
%\item
%$\mu_k(B)$ is neighboring,
%\item
%$\Comp(\mu_k(B))=\mu_k(\Comp(B))$,
%\item
%$\CompPlus(\mu_k(B))=\mu_k(\CompPlus(B))$, and these agree with $\CompPlus(B)$ in the affine rows,
%\item
%the mutation map $\eta^{B^T}_k$ fixes the imaginary ray pointwise, and 
%\item
%if $x$ is a vector whose affine entries are zero, then $\eta^{B^T}_k(x)$ is again a vector whose affine entries are zero; writing $y$ for the projection of $x$ onto the subspace of non-affine indices, the non-affine entries of $\eta^{B^T}_k(x)$ coincide with the entries of the vector $\eta^{\Comp(B)^T}_k(y)$.
%\end{itemize}
%\end{proposition}
\begin{proof}
Mutating at $k$ means replacing the root $\gamma_k$ in the real $c$-cluster $\set{\gamma_1,\ldots,\gamma_n}$.
Since $k$ is not special and not affine, $\gamma_k$ is replaced by another root in $\APTre{c}$.
Thus the $\g$-vector $\xi_k$ is replaced by another vector in $\d_\infty$.
We see that $\mu_k(B)$ is neighboring.

Furthermore, since $k$ is not special and not affine, row $k$ and column $k$ only have nonzero entries within one of the blocks $B_{\ell\ell}$ for $\ell\in\set{1,2,3}$.
Thus mutating~$B$ at~$k$ amounts to replacing $B_{\ell\ell}$ by $\mu_k(B_{\ell\ell})$.
Mutating $B_{\ell\ell}$ at $k$ commutes with applying a positive scaling to a column of $B_{\ell\ell}$ other than~$k$.
Thus $\Comp(\mu_k(B))=\mu_k(\Comp(B))$ and ${\CompPlus(\mu_k(B))=\mu_k(\CompPlus(B))}$.
Since $B$ has zero entries in affine components of column $k$, the affine rows of $\CompPlus(B)$ are unchanged by mutation in position~$k$.

Proposition~\ref{neigh good stuff}.\ref{neigh im ray} says that $\frac12B\delta^B$ is zero in all non-affine entries.
Since $k$ is a non-affine entry, we see that $\eta^{B^T}_k$ fixes $\frac12B\delta^B$ and thus fixes the imaginary ray pointwise.
Since $k$ is non-special, the non-affine entries of column $k$ of $B$ agree with column $k$ of $\Comp(B)$.
Moreover the affine entries in column $k$ of $B$ are zero.
Thus if $x$ is a vector whose affine entries are zero, then the affine entries of $\eta^{B^T}_k(x)$ are zero and $\eta^{B^T}_k(x)=\eta^{\Comp(B)^T}_k(x)$.
\end{proof}

%\begin{proposition}\label{special mut}
%Suppose $B$ is a neighboring exchange matrix and $k$ is a special index of~$B$.
%Then there exists a sequence $\kk$ of indices in $\set{k,n-1,n}$ such that 
%\begin{itemize}
%\item
%$\mu_\kk(B)$ is neighboring,
%\item
%$\Comp(\mu_\kk(B))=\mu_k(\Comp(B))$,
%\item
%$\CompPlus(\mu_\kk(B))$ does not equal $\mu_k(\CompPlus(B))$, but $\CompPlus(\mu_\kk(B))$ agrees with $\CompPlus(B)$ in the affine rows,
%\item
%the mutation map $\eta^{B^T}_\kk$ fixes the imaginary ray pointwise, and 
%\item
%if $x$ is a vector whose affine entries are zero, then $\eta^{B^T}_\kk(x)=\eta^{\Comp(B)^T}_k(x)$.
%\end{itemize}
%\end{proposition}
\begin{proposition}\label{special mut}
Suppose $B$ is a neighboring exchange matrix and $k$ is a special index of~$B$.
Then there exists a sequence $\kk$ of indices in $\set{k,n-1,n}$ with the following properties.
\begin{enumerate}[label=\bf\arabic*., ref=\arabic*]
\item
$\mu_\kk(B)$ is neighboring.
\item
$\Comp(\mu_\kk(B))=\mu_k(\Comp(B))$.
\item
$\CompPlus(\mu_\kk(B))$ does not equal $\mu_k(\CompPlus(B))$, but $\CompPlus(\mu_\kk(B))$ agrees with $\CompPlus(B)$ in the affine rows.
\item
The mutation map $\eta^{B^T}_\kk$ fixes the imaginary ray pointwise.
\item
If $x$ is a vector whose affine entries are zero, then the affine entries of $\eta^{B^T}_\kk(x)$ are again zero and $\eta^{B^T}_\kk(x)=\eta^{\Comp(B)^T}_k(x)$.
\end{enumerate}
\end{proposition}
\begin{proof}
The special index $k$, together with the affine indices, determines a $3\times3$ submatrix of $B$ that agrees with one of the entries in Table~\ref{submat tab}.
For convenience in the rest of the proof, we call this $3\times3$ submatrix merely ``the submatrix'' and refer to the labels in the table as ``the type'' of the submatrix.
We argue separately for the different cases in the table, but are able to combine some cases that are related by scaling.

\medskip

\noindent
\textbf{Case 1.}
The submatrix is of type $A_2^{(1)}$.
The sequence $\kk$ is $k(n-1)knk$.
Recalling that mutations in the sequence are applied from right to left, mutation by this sequence acts on the submatrix as
%\[
%\begin{bsmallmatrix*}[r]
%0&1&-1\\
%-1&0&2\\
%1&-2&0
%\end{bsmallmatrix*}
%\overset{k}{\longrightarrow}
%\begin{bsmallmatrix*}[r]
%0&-1&\,\,\,1\\
%1&0&1\\
%-1&-1&0
%\end{bsmallmatrix*}
%\overset{n-1}{\longrightarrow}
%\begin{bsmallmatrix*}[r]
%0&\,\,\,1&1\\
%-1&0&-1\\
%-1&1&0
%\end{bsmallmatrix*}
%\overset{n}{\longrightarrow}
%\begin{bsmallmatrix*}[r]
%0&2&-1\\
%-2&0&1\\
%1&-1&0
%\end{bsmallmatrix*}.
%\]
\begin{multline*}
\begin{bsmallmatrix*}[r]
0&1&-1\\
-1&0&2\\
1&-2&0
\end{bsmallmatrix*}
\overset{k}{\longrightarrow}
\begin{bsmallmatrix*}[r]
0&-1&\,\,\,1\\
1&0&1\\
-1&-1&0
\end{bsmallmatrix*}
\overset{n}{\longrightarrow}
\begin{bsmallmatrix*}[r]
0&-1&\,\,\,-1\\
1&0&-1\\
1&1&0
\end{bsmallmatrix*}\\
\overset{k}{\longrightarrow}
\begin{bsmallmatrix*}[r]
0&1&\,\,\,1\\
-1&0&-1\\
-1&1&0
\end{bsmallmatrix*}
\overset{n-1}{\longrightarrow}
\begin{bsmallmatrix*}[r]
0&-1&\,\,\,1\\
1&0&1\\
-1&-1&0
\end{bsmallmatrix*}
\overset{k}{\longrightarrow}
\begin{bsmallmatrix*}[r]
0&1&\,\,\,-1\\
-1&0&2\\
1&-2&0
\end{bsmallmatrix*}.
\end{multline*}
In particular, $\mu_\kk(B)$ restricts, in its affine indices, to an affine exchange matrix of rank~$2$,
so Theorem~\ref{neigh B} implies that it is neighboring.

To compute the remaining entries of $\mu_\kk(B)$, we consider each $4\times4$ submatrix whose rows are indexed by $i,k,n-1,n$ and whose columns are indexed by $j,k,n-1,n$, for arbitrary $i,j\not\in\set{k,n-1,n}$.
%To make the notation more compact, we use $p$ to stand for $n-1$.
%We use the notation 
%\[m(a,b)=\begin{cases}
%ab&\text{if }\sgn(a)=\sgn(b)=+\\
%-ab&\text{if }\sgn(a)=\sgn(b)=-\\
%0&\text{otherwise},
%\end{cases}\]
We continue the notation $m(a,b)$ from the proof of Proposition~\ref{neigh good stuff} and we will use the identity $m(a,b)+m(a,-b)=ab$ for $b>0$ several times.
We will also use the facts (from Theorem~\ref{neigh B}) that $b_{(n-1)j}=-b_{nj}\le0$, and $b_{i(n-1)}=-b_{in}\ge0$.
\begin{align*}
\begin{bsmallmatrix*}[r]
b_{ij}&b_{ik}&-b_{in}&b_{in}\\
b_{kj}&0&1&-1\\
-b_{nj}&-1&0&2\\
b_{nj}&1&-2&0
\end{bsmallmatrix*}
&\overset{k}{\longrightarrow}
\begin{bsmallmatrix*}
b_{ij}+m(b_{ik},b_{kj}) & -b_{ik} & -b_{in}+m(b_{ik},1) & b_{in}+m(b_{ik},-1)\\
-b_{kj} & 0 & -1 & 1\\
-b_{nj}+m(b_{kj},-1) & 1 & 0 & 1\\
b_{nj}+m(b_{kj},1) & -1 & -1 & 0
\end{bsmallmatrix*}\\
&\overset{n}{\longrightarrow}
\begin{bsmallmatrix*}
b_{ij}+m(b_{ik},b_{kj}) & b_{in}-m(b_{ik},1) & b_{ik} & -b_{in}-m(b_{ik},-1)\\
b_{nj}-m(b_{kj},-1) & 0 & -1 & -1\\
b_{kj} & 1 & 0 & -1\\
-b_{nj}-m(b_{kj},1) & 1 & 1 & 0
\end{bsmallmatrix*}\\
&\overset{k}{\longrightarrow}
\begin{bsmallmatrix*}
b_{ij}+m(b_{ik},b_{kj}) & -b_{in}+m(b_{ik},1) & b_{in}+m(b_{ik},-1) & -b_{ik}\\
-b_{nj}+m(b_{kj},-1) & 0 & 1 & 1\\
b_{nj}+m(b_{kj},1)  & -1 & 0 & -1\\
-b_{kj}& -1 & 1 & 0
\end{bsmallmatrix*}\\
&\overset{n-1}{\longrightarrow}
\begin{bsmallmatrix*}
b_{ij}+m(b_{ik},b_{kj}) & b_{ik} & -b_{in}-m(b_{ik},-1) & b_{in}-m(b_{ik},1)\\
b_{kj}& 0 & -1 & 1\\
-b_{nj}-m(b_{kj},1)& 1 & 0 & 1\\
b_{nj}-m(b_{kj},-1) & -1 & -1 & 0
\end{bsmallmatrix*}\\
&\overset{k}{\longrightarrow}
\begin{bsmallmatrix*}
b_{ij}+2m(b_{ik},b_{kj}) & -b_{ik} & -b_{in} & b_{in}\\
-b_{kj} & 0 & 1 & -1\\
-b_{nj} & -1 & 0 & 2\\
b_{nj} & 1 & -2 & 0
\end{bsmallmatrix*}.
\end{align*}

It is much easier to find the $ij$-entry of $\mu_k(\Comp(B))$.
We compute
\[
\begin{bsmallmatrix*}
b_{ij}&2b_{ik}\\
b_{kj}&0
\end{bsmallmatrix*}
\overset{k}{\longrightarrow}
\begin{bsmallmatrix*}
b_{ij}+m(2b_{ik},b_{kj})&-2b_{ik}\\
-b_{kj}&0
\end{bsmallmatrix*}
\]
and we conclude that $\Comp(\mu_\kk(B))=\mu_k(\Comp(B))$.
We see also that all entries in affine rows are preserved by $\mu_\kk$.
Thus $\CompPlus(\mu_\kk(B))$ agrees with $\CompPlus(B)$ in the affine rows.
However, $\CompPlus(\mu_\kk(B))\neq\mu_k(\CompPlus(B))$, because the latter disagrees with $\CompPlus(B)$ in the $k\th$ entries of the affine rows (differing by a sign).

The computations above can be reused to prove the assertions about $\eta^{B^T}_\kk$.
%also let us determine mutation maps by taking $j$ to index a new column that has been used to extend $B$ but still taking $i$ to be a non-affine index in~$\set{1,\ldots,n}$.
Proposition~\ref{neigh good stuff}.\ref{neigh im ray} says that the vector $-\frac12B\delta^B$ that spans the imaginary ray is zero in non-affine entries and has affine entries $-1,1$.
Thus, the computations above are valid, replacing the $j\th$ column of $B$ in the calculations with $-\frac12B\delta^B$.
Specifically, we replace $b_{ij}$ and $b_{kj}$ with $0$ and replace $b_{nj}$ with $1$, and the computations show that $\eta^{B^T}_\kk$ fixes $-\frac12B\delta^B$.

Similarly, if $x$ is a vector whose affine entries are zero, it can take the place of the $j\th$ column of $B$ in the above calculations, and will satisfy the requirements $b_{(n-1)j}=-b_{nj}\le0$, and $b_{i(n-1)}=-b_{in}\ge0$ that were used in the calculation.  
We see the affine entries of $\eta^{B^T}_\kk(x)$ are zero and that $\eta^{B^T}_\kk(x)=\eta^{\Comp(B)^T}_k(x)$.
We have proved the proposition in this case.

\medskip

\noindent
\textbf{Case 2.}
The submatrix is of type $C_2^{(1)}$ or $A_4^{(2)}$.
We first consider the case of $C_2^{(1)}$.
The sequence $\kk$ is again $k(n-1)knk$, and mutation acts on the submatrix as shown here:
%Since $\mu_k(DBD^{-1})=D\mu_k(B)D^{-1}$ for any positive diagonal matrix $D$, it suffices to only consider the cases $C_2^{(1)}$, which implies both $A_4^{(2)}$ cases. % and $D_3^{(2)}$, and $G_2^{(1)}$, which implies $D_4^{(3)}$.
%(Let $B'=DBD^{-1}$, then $E'=DED^{-1}$ and $F'=DFD^{-1}$ so that $D(EBF)D^{-1}=E'B'F'$.)
\begin{multline*}
\begin{bsmallmatrix*}[r]
0&2&-2\\
-1&0&2\\
1&-2&0
\end{bsmallmatrix*}
\overset{k}{\longrightarrow}
\begin{bsmallmatrix*}[r]
0&-2&\,\,\,2\\
1&0&0\\
-1&0&0
\end{bsmallmatrix*}
\overset{n}{\longrightarrow}
\begin{bsmallmatrix*}[r]
0&-2&\,\,\,-2\\
1&0&0\\
1&0&0
\end{bsmallmatrix*}\\
\overset{k}{\longrightarrow}
\begin{bsmallmatrix*}[r]
0&2&\,\,\,2\\
-1&0&0\\
-1&0&0
\end{bsmallmatrix*}
\overset{n-1}{\longrightarrow}
\begin{bsmallmatrix*}[r]
0&-2&\,\,\,2\\
1&0&0\\
-1&0&0
\end{bsmallmatrix*}
\overset{k}{\longrightarrow}
\begin{bsmallmatrix*}[r]
0&2&-2\\
-1&0&2\\
1&-2&0
\end{bsmallmatrix*}.
\end{multline*}
Again, Theorem~\ref{neigh B} implies that $\mu_\kk(B)$ is neighboring.

We again compute the remaining entries of $\mu_\kk(B)$ by considering each $4\times4$ submatrix with rows indexed by $i,k,n-1,n$ and columns indexed by $j,k,n-1,n$, for arbitrary $i,j\not\in\set{k,n-1,n}$.
Again, we have $b_{(n-1)j}=-b_{nj}\le0$, and $b_{i(n-1)}=-b_{in}\ge0$.
We again use the identity $m(a,b)+m(a,-b)=ab$ for $b\ge0$ and now also use the identity $m(a,b)-m(a,-b)=|a|b$ for $b\ge0$ and $m(a,pb)=pm(a,b)$ for $p\ge0$.
\begin{align*}
\begin{bsmallmatrix*}[r]
b_{ij}&b_{ik}&-b_{in}&b_{in}\\
b_{kj}&0&2&-2\\
-b_{nj}&-1&0&2\\
b_{nj}&1&-2&0
\end{bsmallmatrix*}
&\overset{k}{\longrightarrow}
\begin{bsmallmatrix*}
b_{ij}+m(b_{ik},b_{kj}) & -b_{ik} & -b_{in}+m(b_{ik},2) & b_{in}+m(b_{ik},-2)\\
-b_{kj} & 0 & -2 & 2\\
-b_{nj}+m(b_{kj},-1) & 1 & 0 & 0\\
b_{nj}+m(b_{kj},1) & -1 & 0 & 0
\end{bsmallmatrix*}\\
&\overset{n}{\longrightarrow}
\begin{bsmallmatrix*}
b_{ij}+m(b_{ik},b_{kj}) & b_{in}-|b_{ik}| & -b_{in}+m(b_{ik},2) & -b_{in}-m(b_{ik},-2)\\
2b_{nj}+|b_{kj}| & 0 & -2 & -2\\
-b_{nj}+m(b_{kj},-1) & 1 & 0 & 0\\
-b_{nj}-m(b_{kj},1) & 1 & 0 & 0
\end{bsmallmatrix*}\\
&\overset{k}{\longrightarrow}
\begin{bsmallmatrix*}
b_{ij}+m(b_{ik},b_{kj}) & -b_{in}+|b_{ik}| & b_{in}+m(b_{ik},-2) & b_{in}-m(b_{ik},2)\\
-2b_{nj}+|b_{kj}| & 0 & 2 & 2\\
b_{nj}+m(b_{kj},1)  & -1 & 0 & 0\\
b_{nj}-m(b_{kj},-1)& -1 & 0 & 0
\end{bsmallmatrix*}\\
&\overset{n-1}{\longrightarrow}
\begin{bsmallmatrix*}
b_{ij}+m(b_{ik},b_{kj}) & b_{ik} & -b_{in}-m(b_{ik},-2) & b_{in}-m(b_{ik},2)\\
b_{kj}& 0 & -2 & 2\\
-b_{nj}-m(b_{kj},1)& 1 & 0 & 0\\
b_{nj}-m(b_{kj},-1) & -1 & 0 & 0
\end{bsmallmatrix*}\\
&\overset{k}{\longrightarrow}
\begin{bsmallmatrix*}
b_{ij}+2m(b_{ik},b_{kj}) & -b_{ik} & -b_{in} & b_{in}\\
-b_{kj} & 0 & 2 & -2\\
-b_{nj} & -1 & 0 & 2\\
b_{nj} & 1 & -2 & 0
\end{bsmallmatrix*}.
\end{align*}

%\begin{align*}
%\begin{bsmallmatrix*}[r]
%b_{ij}&b_{ik}&-b_{in}&b_{in}\\
%b_{kj}&0&2&-2\\
%-b_{nj}&-1&0&2\\
%b_{nj}&1&-2&0
%\end{bsmallmatrix*}
%&\overset{k}{\longrightarrow}
%\begin{bsmallmatrix*}
%a_{ij}+m(x_i,y_j)&-x_i&m(x_i,2)&m(x_i,-2)\\
%-y_j&0&-2&2\\
%m(y_j,-1)&1&0&0\\
%m(y_j,1)&-1&0&0
%\end{bsmallmatrix*}\\
%&\overset{n}{\longrightarrow}
%\begin{bsmallmatrix*}
%a_{ij}+m(x_i,y_j)&-x_i+m(x_i,-2)&m(x_i,2)&-m(x_i,-2)\\
%-y_j+2m(y_j,1)&0&-2&-2\\
%m(y_j,-1)&1&0&0\\
%-m(y_j,1)&1&0&0
%\end{bsmallmatrix*}\\
%&\overset{k}{\longrightarrow}
%\begin{bsmallmatrix*}
%a_{ij}+m(x_i,y_j)&x_i-m(x_i,-2)&m(x_i,-2)&-2x_i+m(x_i,-2)\\
%y_j-2m(y_j,1)&0&2&2\\
%m(y_j,1)&-1&0&0\\
%-y_j+m(y_j,1)&-1&0&0
%\end{bsmallmatrix*}\\
%&\overset{n-1}{\longrightarrow}
%\begin{bsmallmatrix*}
%a_{ij}+m(x_i,y_j)&x_i&-m(x_i,-2)&-2x_i+m(x_i,-2)\\
%y_j&0&-2&2\\
%-m(y_j,1)&1&0&0\\
%-y_j+m(y_j,1)&-1&0&0
%\end{bsmallmatrix*}\\
%&\overset{k}{\longrightarrow}
%\begin{bsmallmatrix*}
%a_{ij}+2m(x_i,y_j)&-x_i&0&0\\
%-y_j&0&2&-2\\
%0&-1&0&2\\
%0&1&-2&0
%\end{bsmallmatrix*}
%\end{align*}
%where we observe in the middle mutation in direction 1 that $-x_i+m(x_i,-2)$ is always non-positive and $-y_j+2m(y_j,1)$ is always non-negative.

The computation of the $ij$-entry of $\mu_k(\Comp(B))$ is the same as in Case~1, and we again conclude that $\Comp(\mu_\kk(B))=\mu_k(\Comp(B))$.
The assertions about $\CompPlus(\mu_\kk(B))$ and about $\eta^{B^T}_\kk$ are proved exactly as in Case~1.
We have proved the proposition in the case of type $C_2^{(1)}$.

The matrices of type $A_4^{(2)}$ shown in Table~\ref{submat tab} are obtained from the matrix of type $C_2^{(1)}$ by conjugating by a diagonal matrix (with diagonal entries either $1,2,1$ or $1,1,2$).
Since $\mu_\kk(D^{-1}BD)=D^{-1}\mu_\kk(B)D$ for any positive diagonal matrix~$D$ \cite[Proposition~4.5]{ca1}, we can reuse the computations above, conjugated by a diagonal matrix, to obtain the proposition in the cases of type $A_4^{(2)}$.

\medskip

\noindent
\textbf{Case 3.}
The submatrix is of type  $G_2^{(1)}$ or $D_4^{(3)}$.
In these cases, the associated root system has roots whose squared lengths are related by a factor of $3$.
The classification of affine root systems thus implies that $n=3$, so that $B$ in fact equals one of the submatrices shown in Table~\ref{submat tab}.
For that reason, the computation is smaller in this case, and the indices $k,(n-1),n$ are $1,2,3$.

We begin with type $G_2^{(1)}$.
The sequence $\kk$ is $1313231$, and the mutation acts on $B$ as shown here.
\begin{multline*}
\begin{bsmallmatrix*}[r]
0&3&-3\\
-1&0&2\\
1&-2&0
\end{bsmallmatrix*}
\overset{1}{\longrightarrow}
\begin{bsmallmatrix*}[r]
0&-3&3\\
1&0&-1\\
-1&1&0
\end{bsmallmatrix*}
\overset{3}{\longrightarrow}
\begin{bsmallmatrix*}[r]
0&0&-3\\
0&0&1\\
1&-1&0
\end{bsmallmatrix*}
\overset{1}{\longrightarrow}
\begin{bsmallmatrix*}[r]
0&0&3\\
0&0&1\\
-1&-1&0
\end{bsmallmatrix*}\\
\overset{3}{\longrightarrow}
\begin{bsmallmatrix*}[r]
0&0&-3\\
0&0&-1\\
1&1&0
\end{bsmallmatrix*}
\overset{2}{\longrightarrow}
\begin{bsmallmatrix*}[r]
0&0&-3\\
0&0&1\\
1&-1&0
\end{bsmallmatrix*}
\overset{3}{\longrightarrow}
\begin{bsmallmatrix*}[r]
0&-3&3\\
1&0&-1\\
-1&1&0
\end{bsmallmatrix*}
\overset{1}{\longrightarrow}
\begin{bsmallmatrix*}[r]
0&3&-3\\
-1&0&2\\
1&-2&0
\end{bsmallmatrix*}.
\end{multline*}
Once again, Theorem~\ref{neigh B} implies that $\mu_\kk(B)$ is neighboring.

There are no additional entries of $\mu_\kk(B)$ to compute, and $\Comp(B)$ is $[0]$, so we see that $\Comp(\mu_\kk(B))=\mu_k(\Comp(B))$.
The assertions about $\CompPlus(\mu_\kk(B))$ is proved as in previous cases.

To prove the assertions about mutation maps, we consider adding an extra column with entries $c_1,c_2,c_3$ with $c_3=-c_2\ge0$.
\begin{multline*}
\begin{bsmallmatrix*}
0&3&-3&c_1\\
-1&0&2&c_2\\
1&-2&0&-c_2
\end{bsmallmatrix*}
\overset{1}{\longrightarrow}
\begin{bsmallmatrix*}
0&-3&3&-c_1\\
1&0&-1&c_2+m(c_1,-1)\\
-1&1&0&-c_2+m(c_1,1)
\end{bsmallmatrix*}\\
\overset{3}{\longrightarrow}
\begin{bsmallmatrix*}
0&0&-3&-3c_2+|c_1|+m(c_1,1)\\
0&0&1&c_2+m(c_1,-1)\\
1&-1&0&c_2-m(c_1,1)
\end{bsmallmatrix*}
\overset{1}{\longrightarrow}
\begin{bsmallmatrix*}
0&0&3&3c_2-|c_1|-m(c_1,1)\\
0&0&1&c_2+m(c_1,-1)\\
-1&-1&0&-2c_2+|c_1|
\end{bsmallmatrix*}\\
\overset{3}{\longrightarrow}
\begin{bsmallmatrix*}
0&0&-3&-3c_2+2|c_1|-m(c_1,1)\\
0&0&-1&-c_2+m(c_1,1)\\
1&1&0&2c_2-|c_1|
\end{bsmallmatrix*}
\overset{2}{\longrightarrow}
\begin{bsmallmatrix*}
0&0&-3&-3c_2+2|c_1|-m(c_1,1)\\
0&0&1&c_2-m(c_1,1)\\
1&-1&0&c_2+m(c_1,-1)
\end{bsmallmatrix*}\\
\overset{3}{\longrightarrow}
\begin{bsmallmatrix*}
0&-3&3&c_1\\
1&0&-1&c_2-m(c_1,1)\\
-1&1&0&-c_2-m(c_1,-1)
\end{bsmallmatrix*}
\overset{1}{\longrightarrow}
\begin{bsmallmatrix*}
0&3&-3&-c_1\\
-1&0&2&c_2\\
1&-2&0&-c_2
\end{bsmallmatrix*}.
\end{multline*}
The imaginary ray is spanned by the vector with entries $c_1=0$, $c_2=-1$, $c_3=1$, and we see that $\eta^{B^T}_\kk$ fixes that vector.
We also see that if $x$ is a vector with entries $c_1,0,0$, then $\eta_\kk^{B^T}(x)$ has entries $-c_1,0,0$.
We compute $\eta_1^{\Comp(B)^T}$ by mutating $[0\,\,c_1]\overset{1}{\longrightarrow}[0\,\,-c_1]$, so $\eta^{B^T}_\kk(x)=\eta^{\Comp(B)^T}_k(x)$ as desired.

The matrix of type $D_4^{(3)}$ in Table~\ref{submat tab} is obtained from the matrix of type $G_2{(1)}$ by conjugating by a diagonal matrix with diagonal entries $3,1,1$.
Since $\mu_\kk(D^{-1}BD)=D^{-1}\mu_\kk(B)D$ for any positive diagonal matrix \cite[Proposition~4.5]{ca1}, we can reuse the computations above, conjugated by a diagonal matrix, to obtain the proposition in the case of type $D_4^{(3)}$.
\end{proof}

Given a  is a sequence $\kk$ of non-affine indices, we define an \newword{expanded sequence}~$\widehat\kk$ 
where each \emph{special} index in $\kk$ is replaced with the corresponding sequence whose existence is proved in Proposition~\ref{special mut}.
Propositions~\ref{nonspecial mut} and~\ref{special mut} combine to say that 
\begin{itemize}
\item
$\Comp(\mu_{\widehat\kk}(B))=\mu_\kk(\Comp(B))$,
\item
the mutation map $\eta^{B^T}_{\widehat\kk}$ fixes the imaginary ray pointwise, and 
\item
if $x$ is a vector whose affine entries are zero, then the affine entries of $\eta^{B^T}_{\widehat\kk}(x)$ are zero and $\eta^{B^T}_{\widehat\kk}(x)=\eta^{\Comp(B)^T}_\kk(x)$.
\end{itemize}
(The fact, in each of Propositions~\ref{nonspecial mut} and~\ref{special mut}, that the mutated matrix is neighboring and has the same affine and special indices as $B$ is crucial to concluding these facts for $\widehat\kk$.)

\begin{proposition}\label{neigh im wall}
Suppose $B$ is neighboring and is indexed as in Theorem~\ref{neigh B}.
Then $\d^B_\infty$ is the half-hyperplane contained in $(\delta^B)^\perp$, containing the vector $-\frac12B\delta^B$, with relative boundary the codimension-$2$ space consisting of vectors that are zero in the affine indices.
\end{proposition}
\begin{proof}
Since $\Comp(B)$ is of finite type, $\F_{\Comp(B)^T}$ is finite and complete.
We consider the mutation fan $\F_{\Comp(B)^T}$ as a complete fan in the subspace of $\reals^n$ consisting of vectors whose affine entries are zero. 
(We refer to this subspace as $\reals^{n-2}$.)
Every maximal cone $U$ of $\F_{\Comp(B)^T}$ is the image of the positive cone in $\reals^{n-2}$ under some mutation map $\eta^{(B')^T}_{\kk^{-1}}$, where $B'=\mu_\kk(\Comp(B))$ and $\kk$ is a sequence of non-affine indices.  
Let $\widehat\kk$ be the expanded sequence, so that $B'=\Comp(\mu_{\widehat\kk}(B))$.
There is an imaginary cone in $\F_{\mu_{\widehat\kk}(B)^T}$ that is the nonnegative linear span of the imaginary ray and the positive cone in $\reals^{n-2}$.
Since $\eta^{\mu_{\widehat\kk}(B)^T}_{\widehat\kk^{-1}}$ is an isomorphism from $\F_{\mu_{\widehat\kk}(B)^T}$ to $\F_{B^T}$ that fixes the imaginary ray and sends the positive cone in $\reals^{n-2}$ to $U$, the cone spanned by $U$ and the imaginary ray is an imaginary cone in $\F_{B^T}$.
We see that the imaginary cones in $\F_{B^T}$ fill the half-hyperplane described in the proposition.
Since the imaginary ray is in the relative interior of that half-hyperplane, there are no other imaginary cones.
\end{proof}

It follows from Proposition~\ref{neigh im wall} that for any neighboring $B$ and $\lambda\in\d_\infty^B$, we may write $\lambda=\lambda_0+\lambda_\infty$ where~$\lambda_0$ has affine entries $0$ and $\lambda_\infty$ is a nonnegative scaling of $-\frac12B\delta^B$.

\begin{proposition}\label{factor eta}
  Suppose $B$ is neighboring and $\lambda\in\d_\infty^B$ with $\lambda=\lambda_0+\lambda_\infty$ as above.
If $\kk$ is a sequence of non-affine indices with expanded sequence $\widehat\kk$, then ${\eta^{B^T}_{\widehat\kk}(\lambda)=\eta_\kk^{\Comp(B)^T}(\lambda_0)+\lambda_\infty}$.
\end{proposition}
\begin{proof}
%Proposition~\ref{neigh im wall} implies that we can write $\lambda$ as $\lambda_0+\lambda_\infty$.
Since $\eta_{\widehat\kk}^{B^T}$ is linear on the imaginary cone containing $\lambda$, and since that cone contains the imaginary ray~$\delta^B$, the proposition follows from Propositions~\ref{nonspecial mut} and~\ref{special mut}.
\end{proof}

We now prove the main theorem.

\begin{proof}[Proof of Theorem~\ref{affine main}]
Take $\lambda\in\d_\infty^B$, consider any $B'$ mutation-equivalent to $B$, and specifically take any sequence $\ll$ of indices in $\set{1,\ldots,n}$ such that $B'=\mu_\ll(B)$.
Write $\lambda'$ for $\eta_\ll^{B^T}(\lambda)$.
The mutation map $\eta_\ll^{B^T}$ acting on $\F_{B^T}$ takes the imaginary ray of $\F_{B^T}$ to the imaginary ray of $\F_{\mu_\ll(B)^T}$, and more specifically takes the vector $-\frac12B\delta^B$ to the vector $-\frac12\mu_\ll(B)\delta^{\mu_\ll(B)}$.
Furthermore, $\eta_\ll^{B^T}$ is linear on each imaginary cone of $\F_{B^T}$, and thus takes the line segment that Theorem~\ref{affine main} claims is equal to~$\P_\lambda^B$ to the line segment that is claimed to equal $\P_{\lambda'}^{B'}$.
Now Lemma~\ref{shift} implies that it is enough to prove the theorem for any one $B$ in the exchange pattern.

We choose $B$ to be a neighboring exchange matrix.
%Proposition~\ref{affine main partial} says that~$\P_\lambda^B$ contains the line segment described in the theorem.
%We need to show the opposite containment.
Suppose now that $\kk$ is a sequence of \emph{non-affine} indices, let~$\widehat\kk$ be the expanded sequence, and take ${B'=\mu_{\widehat\kk}(B)}$.
Write $\lambda$ as $\lambda_0+\lambda_\infty$, where $\lambda_0$ is zero in the affine indices and $\lambda_\infty$ is a nonnegative scaling of $-\frac12B\delta^B$, so that ${\eta^{B^T}_{\widehat\kk}(\lambda)=\eta_\kk^{\Comp(B)^T}(\lambda_0)+\lambda_\infty}$ by Proposition~\ref{factor eta}.
Proposition~\ref{Comp span} implies that 
\begin{multline*}
\set{B'\alpha:\alpha\in\reals^n,\alpha\ge0}\cap(\delta^{B'})^\perp\\
\subseteq\set{\CompPlus(B')\alpha+aB'\delta^{B'}:\alpha\in\reals^{n-2},\alpha\ge0,a\ge0}.
\end{multline*}

Let $M$ be the matrix obtained from $\CompPlus(B)$ by replacing the first $n-2$ rows by zeros (leaving only the affine entries).
Propositions~\ref{nonspecial mut} and~\ref{special mut} imply that $M$ is also the matrix obtained from $\CompPlus(B')$ by replacing the first $n-2$ rows by zeros.
Theorem~\ref{neigh B} and Proposition~\ref{delta is the man} imply that every column of $M$ is parallel to~$B\delta^B$.
In the following formula, we will think of the columns of $\Comp(B')$ as vectors in $\reals^n$ with entries $0$ in the affine positions.

We see that 
\begin{multline*}
  \set{\eta_{\widehat\kk}^{B^T}(\lambda)+B'\alpha:\alpha\in\reals^n,\alpha\ge0}\cap\d^{B'}_\infty\\
\begin{aligned}
&\quad\subseteq\set{\eta_{\widehat\kk}^{B^T}(\lambda)+\CompPlus(B')\alpha+aB'\delta^{B'}:\alpha\in\reals^{n-2},\alpha\ge0,a\ge0}\cap\d^{B'}_\infty\\
&\quad=\set{\eta_\kk^{\Comp(B)}(\lambda_0)+\Comp(B')\alpha+\lambda_\infty+M\alpha+aB'\delta^{B'}:\alpha\ge0,a\ge0}\cap\d^{B'}_\infty.\\
\end{aligned}
\end{multline*}
The projection of this set to $\reals^{n-2}$ is $\set{\eta_\kk^{\Comp(B)}(\lambda_0)+\Comp(B')\alpha:\alpha\ge0}$.

The set $\P^B_{\lambda,\widehat\kk}$ is obtained by applying $(\eta_{\widehat\kk}^{B^T})^{-1}$ to $\set{\eta_{\widehat\kk}^{B^T}(\lambda)+B'\alpha:\alpha\in\reals^n,\alpha\ge0}$.
By Proposition~\ref{factor eta}, if some point $x$ is in $\P^B_{\lambda,\widehat\kk}$  
for every $\widehat\kk$, then the projection of $x$ to $\reals^{n-2}$ is in $\P^{\Comp(B)}_{\lambda_0,\kk}$ for every $\kk$.
Therefore Theorem~\ref{finite P point no} says that the projection of~$x$ is $\lambda_0$.
We have showed that $\P^B_\lambda$ is contained in the line through $\lambda$ in the direction of~$B\delta^B$.
But also $\P^B_\lambda$ is contained in $\d^B_\infty$, so it is contained in the ray in $\d^B_\infty$ parallel to the imaginary ray, containing $\lambda$, and having endpoint in the relative boundary of $\d^B_\infty$.
Arguing as in the first paragraph of this proof, we see that this fact about containment in a ray is true for any $B$ in the exchange pattern, and to restrict from the ray to the line segment, we can choose any $B$ in the exchange pattern.
Proposition~\ref{affine salient} allows us to choose $B$ to be a salient exchange matrix in the exchange pattern.
For this $B$, since $B\delta^B$ is in the nonnegative linear span of the columns of $B$, $-B\delta^B$ is not, so $\P^B_\lambda$ is contained in the set $\set{\lambda+B\delta^Ba:a\ge0}$.
We have shown that $\P^B_\lambda$ is contained in the line segment parallel to the imaginary ray, with one endpoint at $\lambda$ and the other endpoint on the relative boundary of $\d^B_\infty$ and we can conclude that the reverse containment holds by Proposition~\ref{affine main partial}.
\end{proof}

\subsection{Extended exchange matrices of affine type}
In this section, we extend Theorem~\ref{affine main} to arbitrary extensions of $B$.  
Recall that $\Proj_n$ is the projection from $\reals^m$ to $\reals^n$ that ignores the last $m-n$ coordinates.
Recall also that the mutation fan for $\tB^T$ is the set $\F_{\tB^T}$ of cones $\Proj_n^{-1}C$ such that $C$ is a cone in the mutation fan $\F_{B^T}$.
If $B$ is of affine type, then the imaginary wall $\d_\infty^B$ is a union of cones of $\F_{B^T}$.
Accordingly, we define $\d_\infty^\tB$ to be $\Proj_n^{-1}\d_\infty^B$.

\begin{theorem}\label{affine main extended}
Suppose $\tB$ is an extended exchange matrix such that $B$ is of affine type.
%Suppose $B$ is an exchange matrix of affine type and $\tB$ is an extension of $B$.
  If $\tilde\lambda$ is in the relative interior of the imaginary wall~$\d^\tB_\infty$, then the dominance region $\P^\tB_{\tilde\lambda}$ is $\set{\tilde\lambda+a\tB\delta^B:a\ge0}\cap\d^\tB_\infty$, the line segment parallel to $\tB\delta^B$, with one endpoint at $\tilde\lambda$ and the other endpoint on the relative boundary of $\d^\tB_\infty$.
\end{theorem}

The imaginary ray in $\F_{B^T}$, spanned by $-\frac12B\delta^B$, corresponds to a $(1+m-n)$-dimensional cone in $\F_{\tB^T}$.
The statement of Theorem~\ref{affine main extended} suggests that there is a special direction within that $(1+m-n)$-dimensional cone, spanned by $-\frac12\tB\delta^B$.
In preparation for the proof of Theorem~\ref{affine main extended}, we prove a special property of that direction, namely how it behaves under mutation maps.

\begin{proposition}\label{delta is the tilde man}
Suppose $B$ is an exchange matrix of affine type and $\tB$ is an extension of $B$.
Then $\eta^{\BB^T}_\kk(-\frac12\tB\delta^B)=-\frac12\mu_\kk(\tB)\delta^{\mu_\kk(B)}$ for any sequence $\kk$ of indices in $\set{1,\ldots,n}$.
\end{proposition}
The factors $\frac12$ in the statement are for aesthetic reasons, but the negative signs are crucial.
The analogous statement for $-\frac12B\delta^B$ is true because $-\frac12B\delta^B$ is the shortest integer vector in the imaginary ray.

\begin{proof}%[Proof of Proposition~\ref{delta is the tilde man}]
If $t$ is a seed with extended exchange matrix $\tB$, the matrix $\mu_\kk(\tB)$ depends only on the seed $t'$ obtained from $t$ by mutating along the sequence $\kk$, not on the choice of a sequence $\kk$ that sends $t$ to $t'$.
In turn, thanks to Proposition~\ref{delta is the man}, both~$\delta^B$ and $\delta^{\mu_\kk(B)}$ depend only on the seeds $t$ and $t'$ and not on the choice of a sequence connecting them.
Thus, to prove the proposition, it is enough to show that, given any seeds $t$ and $t'$ with extended exchange matrices $\tB$ and $\tB'$ (of affine type), there exists a sequence $\kk$ taking~$t$ to~$t'$ such that $\eta^{\BB^T}_\kk(-\frac12\tB\delta^B)=-\frac12\mu_\kk(\tB)\delta^{\mu_\kk(B)}$.

Choose an acyclic exchange matrix $B_0$, mutation equivalent to $B$, and let $c$ be the Coxeter element determined by~$B_0$.
To define $\delta^B$, we started with an initial seed~$t_0$ and initial exchange matrix $B_0$ and appealed to Proposition~\ref{any B is sort} to choose a seed~$t$ with exchange matrix $B$, with good properties.
The proof of Proposition~\ref{any B is sort} proceded by taking an \emph{arbitrary} seed with exchange matrix $B$ and applying \emph{initial seed mutations} along the sequence $(12\cdots n)^i$ for large enough $i$ so that the desired properties hold. 
These initial seed mutations change which seed we called the initial seed $t_0$ while preserving the fact that the initial exchange matrix is $B_0$.
Further increasing~$i$ does not destroy the desired properties.

Therefore we can carry out this process simultaneously for $B$ and $B'$ and choose~$i$ large enough so that \emph{both} seeds have the desired properties relative to the initial seed $t_0$.
Write $\tB_0$ for the extended exchange matrix at $t_0$.
The output of Proposition~\ref{any B is sort}, for this large enough $i$, is a $c$-sortable element $v$ such that the sequence $\kk=k_r\cdots k_1$ that reverses the indices in the $c$-sorting word for $v$ has ${t_0\overset{k_1}{\edge}t_1\overset{k_2}{\edge}\,\cdots\,\overset{k_r}{\edge}t_r=t}$ such that $\bigl(G_{t_p}^{-B_0^T;t_0}\bigr)^T\delta$ has nonnegative entries for all ${p=0,\ldots,r}$, and analogously for $t'$.
%(That is, a $c$-sortable element $v'$ such that the sequence $\kk'$ that reverses the indices in the $c$-sorting word for $v'$ has ${t_0\overset{k'_1}{\edge}t'_1\overset{k'_2}{\edge}\,\cdots\,\overset{k'_{r'}}{\edge}t'_r=t}'$ such that $\bigl(G_{t'_p}^{-B_0^T;t_0}\bigr)^T\delta$ has nonnegative entries for all ${p=0,\ldots,r'}$.)

The argument in the proof of Proposition~\ref{delta is the man}, with Lemma~\ref{EBF trick ext} replacing Lemma~\ref{EBF trick}, shows that $\eta^{\BB_0^T}_\kk(-\frac12\tB_0\delta^{B_0})=-\frac12\tB\delta^B$.
%
%We argue, by extending the argument in the proof of Proposition~\ref{delta is the man}, that $\eta^{\BB_0^T}_\kk(-\frac12\tB_0\delta^{B_0})=-\frac12\tB\delta^B$.
%Suppose $a_1\cdots a_m$ is the $c$-sorting word for $v$.
%As in the proof of Proposition~\ref{delta is the man}, we argue by induction on $m$.
%The base case is trivial.
%By induction, $\eta^{\BB_0^T}_{k_{m-1}\cdots k_1}(-\frac12\tB_0\delta^{B_0})=-\frac12\tB_{r-1}\delta^{B_{r-1}}$, and we need to check that $\eta_{k_r}^{\BB_{r-1}^T}(-\frac12\tB_{r-1}\delta^{B_{r-1}})=-\frac12\tB\delta^{B}$.
%
%As in the proof of Proposition~\ref{delta is the man}, $\delta^B=F_{-,k_r}^{B_{r-1}}\bigl(G_{t_{r-1}}^{-B_0^T;t_0}\bigr)^T\delta$.
%
%Lemma~\ref{EBF trick ext} says that $\tB=E_{-,k_r}^{\BB_{r-1}}\tB_{r-1}F_{-,k_r}^{B_{r-1}}$, and since $F_{-,k_r}^{B_{r-1}}$ is its own inverse, 
%\[-\frac12\tB\delta^B=-\frac12E_{-,k_r}^{\BB_{r-1}}\tB_{r-1}\bigl(G_{t_{r-1}}^{-B_0^T;t_0}\bigr)^T\delta=E_{-,k_r}^{\BB_{r-1}}(-\frac12\tB_{r-1}\delta^{B_{r-1}}).\]
%Then showing that $\eta_{k_r}^{\BB_{r-1}^T}(-\frac12\tB_{r-1}\delta^{B_{r-1}})=E_{-,k_r}^{\BB_{r-1}}(-\frac12\tB_{r-1}\delta^{B_{r-1}})$ amounts to showing that the $k_r$-entry of $-\frac12\tB_{r-1}\delta^{B_{r-1}}$ is nonpositive.
%This was checked in the proof of Proposition~\ref{delta is the man}.
%We see that $\eta^{\BB_0^T}_\kk(-\frac12\tB_0\delta^{B_0})=-\frac12\tB\delta^B$.
%
The same argument, with~$t'$ replacing $t$, shows that $\eta^{\BB_0^T}_{\kk'}(-\frac12\tB_0\delta^{B_0})=-\frac12\tB'\delta^{B'}$.
Multiplying by the inverse $\eta^{\BB^T}_{(\kk')^{-1}}$ of $\eta^{\BB_0^T}_{\kk'}$, we see that $-\frac12\tB_0\delta^{B_0}=\eta^{\BB^T}_{(\kk')^{-1}}(-\frac12\tB'\delta^{B'})$.
Substituting into $\eta^{\BB_0^T}_\kk(-\frac12\tB_0\delta^{B_0})=-\frac12\tB\delta^B$, we obtain $\eta^{\BB_0^T}_\kk\big(\eta^{\BB^T}_{(\kk')^{-1}}(-\frac12\tB'\delta^{B'})\big)=-\frac12\tB\delta^B$, or in other words $\eta^{\BB^T}_{\kk(\kk')^{-1}}(-\frac12\tB'\delta^{B'})=-\frac12\tB\delta^B$, as desired.
\end{proof}

\begin{proof}[Proof of Theorem~\ref{affine main extended}]
In light of Lemma~\ref{shift extended} and Proposition~\ref{delta is the tilde man}, we may assume that $B$ is a neighboring exchange matrix.
Furthermore, in light of Propositions~\ref{nonspecial mut} and~\ref{special mut}, we can assume that each of the submatrices $B_{11}$, $B_{22}$, or $B_{33}$ is either empty or bipartite.
  Proposition~\ref{contains proj} says that $\P_{\tilde\lambda}^\tB\subseteq\Proj_n^{-1}(\P_{\lambda}^B)$, where $\lambda=\Proj_n\tilde\lambda\in\reals^n$.
By the definition of $\d_\infty^\tB$, we have $\lambda\in\d_\infty^B$, so Theorem~\ref{affine main} says $\P_{\lambda}^B$ is the intersection of $\d_\infty^B$ with the ray that is pointed at $\lambda$ and has direction~$B\delta^B$.  

  Any vector $x\in\P_{\tilde\lambda}^\tB$ is $\tilde\lambda+\tB\alpha$ for some $\alpha\ge0$ in $\reals^n$.
Since $\Proj_n(x)\in\P_{\lambda}^B$, we see that $\Proj_n(x)$ is $\lambda+aB\delta^B$ from some $a\ge0$ in $\reals$.
Thus $B\alpha$ and $B\delta^B$ agree up to a nonnegative scaling.
Inspection of Proposition~\ref{neigh good stuff} and Theorem~\ref{neigh B} implies that $\alpha$ and $a\delta^B$ agree, up to nonnegative scaling, in their affine entries.
Thus, when we compute $B\alpha$, the contribution of the affine columns of $B$ and the affine entries of $\alpha$ is zero.  
Write $B'$ for the submatrix of $B$ indexed by the non-affine entries and $\alpha'$ for the vector of non-affine entries of $\alpha$.
Because $B\delta^B$ is zero in its non-affine entries by Proposition~\ref{neigh good stuff}, we see that $B'\alpha'$ is zero.
But because each of the submatrices $B_{11}$, $B_{22}$, or $B_{33}$ is either empty or bipartite, $B'$ is salient, so $\alpha'=0$.
We see that $\alpha$ and $\delta^B$ agree in all their entries up to a nonnegative scaling.
The theorem follows.
\end{proof}

\subsection{Integral dominance regions in affine type}
\label{sec:integral}
We conclude by characterizing the integral dominance regions of affine type. 
Given an integer point $\tilde\lambda\in\integers^m$, define $\IP^\tB_{\tilde\lambda,\kk}=\bigl(\eta_{\kk}^{\BB^T}\bigr)^{-1}\sett{\eta_\kk^{\BB^T}(\tilde\lambda)+\mu_\kk(\tB)\alpha:\alpha\in\integers^n,\alpha\ge0}$ for all sequences~$\kk$.
The \newword{integral dominance region} $\IP^\tB_{\tilde\lambda}$ of $\tilde\lambda$ with respect to~$\tB$ is the intersection $\bigcap_\kk\IP^\tB_{\tilde\lambda,\kk}$ over all~$\kk$.

It is immediate that $\set{\tilde\lambda}\subseteq\IP^\tB_{\tilde\lambda}\subseteq\P^\tB_{\tilde\lambda}$.   
Thus in particular, Theorem~\ref{affine P point indep} is true with $\IP_{\tilde\lambda}^\tB$ replacing $\P_{\tilde\lambda}^\tB$.
The following is the integral version of Theorem~\ref{affine main extended}.

\begin{theorem}\label{affine main integral}  
  Suppose $\tB$ is an $m\times n$ extended exchange matrix such that $B$ is of affine type.
  If $\tilde\lambda\in\integers^m$ is in the relative interior of the imaginary wall~$\d^\tB_\infty$, then the integral dominance region $\IP^\tB_{\tilde\lambda}$ is $\set{\tilde\lambda+a\tB\delta^B:a\ge0\in\integers}\cap\d^\tB_\infty$, consisting of every other integer point in the line segment $\P^\tB_{\tilde\lambda}$, starting with $\tilde\lambda$.
\end{theorem}

\begin{proof}
This proof is nearly identical to the proof of Theorem~\ref{affine main extended}.
Lemma~\ref{shift extended} applies to integral dominance regions by the same proof.
We appeal to the integral version of Lemma~\ref{shift extended} and to Proposition~\ref{delta is the tilde man} to assume that $B$ is neighboring and make the same assumptions on $B_{11}$, $B_{22}$, or $B_{33}$.
  We have $\IP_{\tilde\lambda}^\tB\subseteq\Proj_n^{-1}(\P_{\lambda}^B)$ for $\lambda=\Proj_n\tilde\lambda\in\reals^n$, and Theorem~\ref{affine main} says $\P_{\lambda}^B=\set{\lambda+aB\delta^B:a\ge0\in\reals}\cap\d^B_\infty$.

Any vector $x\in\IP_{\tilde\lambda}^\tB$ is $\tilde\lambda+\tB\alpha$ for some $\alpha\ge0$ in $\integers^n$.
Arguing just as in the proof of Theorem~\ref{affine main extended}, we see that $\alpha$ and $\delta^B$ agree in all their entries up to a nonnegative scaling.
Since $\alpha$ is an integer vector, by Proposition~\ref{delta is the man}.\ref{im ray pos}, $\alpha$ is a nonnegative integer multiple of $\delta^B$.
Thus, in light of Proposition~\ref{delta is the tilde man} and the fact that $\eta^{\tB^T}_\kk$ is a bijection on the integer vectors in $\reals^m$, the integral dominance region $\IP^\tB_{\tilde\lambda}$ is $\set{\tilde\lambda+a\tB\delta^B:a\ge0\in\integers}\cap\d^\tB_\infty$.
The fact about every other integer point follows by Proposition~\ref{delta is the man}.\ref{im ray}.
\end{proof}

\bibliographystyle{plain}
\bibliography{bibliography}
\vspace{-0.6 em}

\end{document}